\documentclass[1p,authoryear]{elsarticle}

\usepackage{amsmath}                    %
\usepackage{amssymb}                    %
\usepackage{mathtools}                  %
\usepackage{amsthm}                     %
\usepackage{mathrsfs}                   %
\usepackage{stmaryrd}                   %
\usepackage[symbol,perpage]{footmisc}   %
\usepackage[np,autolanguage]{numprint}  %
\usepackage{ifthen}                     %
\usepackage{afterpage}                  %
\usepackage{pdfsync}                    %
\usepackage{subfigure}                  %
\theoremstyle{plain}                    %
\newtheorem{theorem}{Theorem}           %
\newtheorem{lemma}{Lemma}               %
\theoremstyle{definition}               %
\newtheorem{problem}{Problem}           %
\theoremstyle{remark}                   %
\newtheorem{remark}{Remark}             %
\makeatletter                                                              %
\newif\ifFirstPar       \FirstParfalse                                     %
\DeclareRobustCommand\SMC{%                                                %
  \ifx\@currsize\normalsize\small\else                                     %
   \ifx\@currsize\small\footnotesize\else                                  %
    \ifx\@currsize\footnotesize\scriptsize\else                            %
     \ifx\@currsize\large\normalsize\else                                  %
      \ifx\@currsize\Large\large\else                                      %
       \ifx\@currsize\LARGE\Large\else                                     %
        \ifx\@currsize\scriptsize\tiny\else                                %
         \ifx\@currsize\tiny\tiny\else                                     %
          \ifx\@currsize\huge\LARGE\else                                   %
           \ifx\@currsize\Huge\huge\else                                   %
            \small\SMC@unknown@warning                                     %
 \fi\fi\fi\fi\fi\fi\fi\fi\fi\fi                                            %
}                                                                          %
\newcommand\SMC@unknown@warning{\TBWarning{\string\SMC: unrecognised       %
    text font size command -- using \string\small}}                        %
\newcommand\textSMC[1]{{\SMC #1}}                                          %
\newcommand\acro[1]{\textSMC{#1}\@}                                        %
\newcount\T@stCount             \newcount\TestCount                        %
\def\nth#1{%                                                               %
    \def\reserved@a##1##2\@nil{\ifcat##1n%                                 %
           0%                                                              %
   \let\reserved@b\ensuremath                                              %
      \else##1##2%                                                         %
   \let\reserved@b\relax                                                   %
      \fi}%                                                                %
    \TestCount=\reserved@a#1\@nil\relax                                    %
    \ifnum\TestCount <0 \multiply\TestCount by\m@ne \fi % subdue negatives %
    \T@stCount=\TestCount                                                  %
    \divide\T@stCount by 100 \multiply\T@stCount by 100                    %
    \advance\TestCount by-\T@stCount     % n mod 100                       %
    \ifnum\TestCount >20 \T@stCount=\TestCount                             %
      \divide\T@stCount by 10 \multiply\T@stCount by 10                    %
      \advance\TestCount by-\T@stCount   % n mod 10                        %
    \fi                                                                    %
     \reserved@b{#1}%                                                      %
       \textsuperscript{\ifcase\TestCount th%    0th                       %
                        \or   st%                1st                       %
                        \or   nd%                2nd                       %
                        \or   rd%                3rd                       %
                        \else th%                nth                       %
                        \fi}%                                              %
     }                                                                     %
\makeatother%                                                              %
\newcommand{\nsd}[1]{\ensuremath{\mathrm{d}{#1}}}              %
\renewcommand{\d}[1]{\ensuremath{\,\mathrm{d}{#1}}}            %
\newcommand{\bv}[1]{\boldsymbol{#1}}                           %
\newcommand{\tensor}[1]{\mathsf{#1}}                           %
\newcommand{\trans}[1]{#1^{\mathrm{T}}}                        %
\newcommand{\invtrans}[1]{\ensuremath{#1^{-\mathrm{T}}}}       %
\newcommand{\f}{\mathrm{f}}                                    %
\newcommand{\s}{\mathrm{s}}                                    %
\newcommand{\sv}{\ensuremath{1}}                               %
\newcommand{\svs}{\ensuremath{\alpha}}                         %
\newcommand{\sq}{\ensuremath{2}}                               %
\newcommand{\sqs}{\ensuremath{\beta}}                          %
\newcommand{\sy}{\ensuremath{3}}                               %
\newcommand{\sys}{\ensuremath{\gamma}}                         %
\newcommand{\map}[1][Blank]{\ifthenelse{\equal{#1}{Blank}}%    %
{\ensuremath{\bv{\zeta}(\bv{s},t)}}{\bv{s} + \bv{#1}(\bv{s})}} %
\newcommand{\ucdot}{\!\cdot\!}                                 %
\newcommand{\dualV}[1]{\prescript{}{\mathscr{V}^{*}}%          %
{\bigl\langle} #1 \big\rangle_{\mathscr{V}}}                   %
\newcommand{\dualY}[1]{\prescript{}{\mathscr{H}_{Y}^{*}}%      %
{\bigl\langle} #1 \big\rangle_{\mathscr{H}_{Y}}}               %
\DeclareMathOperator{\ldiv}{div}        %
\DeclareMathOperator{\grad}{grad}       %
\DeclareMathOperator{\trace}{tr}        % 
\DeclareMathOperator{\vssp}{span}       %
\journal{Computer Methods in Applied Mechanics and Engineering}
\begin{document}
%%%%%%%%%%%%%%%%%%%%%%%%%%%%%%
% *** BEGIN  FRONTMATTER *** %
%%%%%%%%%%%%%%%%%%%%%%%%%%%%%%
\begin{frontmatter}

\title{Variational Implementation of Immersed Finite Element Methods}

\author[sissa]{Luca Heltai\corref{cor}}
\ead{luca.heltai@sissa.it}
\cortext[cor]{Corresponding author}

\author[PSU]{Francesco Costanzo}
\ead{costanzo@engr.psu.edu}

\address[sissa]{Scuola Internazionale Superiore di Studi Avanzati, Via Bonomea 265, 34136 Trieste, Italy}

\address[PSU]{Department of Engineering Science and Mechanics, The Pennsylvania State University, 212 Earth and Engineering Sciences Building, University Park, PA 16802, USA}

%: Abstract
\begin{abstract}
\emph{Dirac-$\delta$ distributions} are often crucial components of the solid-fluid coupling operators in immersed solution methods for fluid-structure interaction (\acro{FSI}) problems.  This is certainly so for methods like the \emph{Immersed Boundary Method} (\acro{IBM}) or the \emph{Immersed Finite Element Method} (\acro{IFEM}), where Dirac-$\delta$ distributions are approximated via smooth functions. By contrast, a truly variational formulation of immersed methods does not require the use of Dirac-$\delta$ distributions, either formally or practically. This has been shown in the \emph{Finite Element Immersed Boundary Method} (\acro{FEIBM}), where the variational structure of the problem is exploited to avoid Dirac-$\delta$ distributions at both the continuous and the discrete level.

In this paper, we generalize the \acro{FEIBM} to the case where an incompressible Newtonian fluid interacts with a general hyperelastic solid. Specifically, we allow (\emph{i}) the mass density to be different in the solid and the fluid, (\emph{ii}) the solid to be 
either viscoelastic of differential type or purely elastic, and (\emph{iii}) the solid to be either compressible or incompressible.  At the continuous level, our variational formulation combines the natural stability estimates of the fluid and elasticity problems.  In immersed methods, such stability estimates do not transfer to the discrete level automatically due to the non-matching nature of the finite dimensional spaces involved in the discretization.  After presenting our general mathematical framework for the solution of \acro{FSI} problems, we focus in detail on the construction of natural interpolation operators between the fluid and the solid discrete spaces, which guarantee semi-discrete stability estimates and strong consistency of our spatial discretization. 
\end{abstract}


\end{frontmatter}

%%%%%%%%%%%%%%%%%%%%%%%%%%%%%%%%%%%%%%%%%%%%%%%%%%
% This is an AMS-LaTeX command to allow breaking %
% of displayed equations across pages. Note the  %
% closing the "}" just before the bibliography.  %
%%%%%%%%%%%%%%%%%%%%%%%%%%%%%%%%%%%%%%%%%%%%%%%%%%
\allowdisplaybreaks{                             %
%%%%%%%%%%%%%%%%%%%%%%%%%%%%%%%%%%%%%%%%%%%%%%%%%%%

%%: Introduction
\section{Introduction}
There are several approaches to the solution of general fluid-structure interaction (\acro{FSI}) problems.  Among these we find a set of methods, which we will call \emph{immersed methods}, for which the discretization of the fluid domain is completely independent of that of the solid. These methods are more recent than and stand in contrast to established methods like the arbitrary Lagrangian-Eulerian (\acro{ALE}) ones (see, e.g., \citealp{HughesLiu_1981_Lagrangian-Eulerian_0}), where the topologies of the solution grids for the fluid and the solid are constrained.  When the flow can be modeled as a linear Stokes flow, reduced methods, such as the boundary element method (see, e.g., \citealp{AlougesDeSimoneLefebvre-2007-a,AlougesDeSimoneHeltai-2011-a}) can be very efficient. However, for more general situations immersed methods offer appealing features.

In immersed methods the solid domain is surrounded by the fluid.  When the fluid and solid do not slip relative to one another, these methods have three basic features:
\begin{enumerate}
\item
The support of the equations of motion of the fluid\footnote{These equations are typically taken to be the Navier-Stokes equations (see, e.g., \citealp{Peskin_1977_Numerical_0,BoffiGastaldi_2003_A-Finite_0,BoffiGastaldiHeltaiPeskin-2008-a,WangLiu_2004_Extended_0}).  However, formulations in which the fluid is modeled as slightly compressible have also been proposed \citep{LiuKim_2007_Mathematical_0,WangZhang_2009_On-Computational_0}.} is extended to the union of the physical fluid and solid domains.

\item
The equations of motion of the fluid have terms that, from a continuum mechanics viewpoint, are body forces ``informing'' the fluid of its interaction with the solid.

\item
The velocity field of the immersed solid is identified with the restriction to the solid domain of the velocity field in the equations of motion of the fluid.
\end{enumerate}
In many respects, immersed methods can be distinguished from one another depending on how these three elements are treated theoretically and/or are implemented practically.

Immersed methods were pioneered by Peskin and his co-workers (\citealp{Peskin_1977_Numerical_0}; see also \citealp{Peskin_2002_The-immersed_0}, for a comprehensive account) who proposed an approach known as the immersed boundary method (\acro{IBM}).  In the \acro{IBM}, the approximate solution of the extended fluid flow problem is obtained via a finite difference (\acro{FD}) scheme.  The body forces expressing the \acro{FSI} are determined by modeling the solid body as a network of elastic fibers with a contractile element.  As such, this system of forces has singular support (the \emph{boundary} in the method's name) and is implemented via Dirac-$\delta$ distributions.  From a numerical viewpoint, the configuration of the fiber network is identified with that of a discrete set of points. The motion of these points is then related to the motion of the fluid again via Dirac-$\delta$ distributions.  Hence, Peskin's approach relies on Dirac-$\delta$ distributions twice: first for the determination of the \acro{FSI} force system, and again for the determination of the motion of the fiber network.  In the method's numerical implementation, the Dirac-$\delta$ distributions are aproximated as \emph{functions}.  Both the use of \acro{FD} schemes and the approximation of the Dirac-$\delta$ distribution yield inconveniences that can be avoided by reformulating the problem in variational form and adopting corresponding approximation schemes such as finite element methods (\acro{FEM}).

The replacement of the \acro{FD} scheme with a finite element method (\acro{FEM}) was first proposed, almost simultaneously, by \cite{BoffiGastaldi_2003_A-Finite_0}, \cite{WangLiu_2004_Extended_0}, and \cite{ZhangGerstenberger_2004_Immersed_0}.  \cite{BoffiGastaldi_2003_A-Finite_0} show that a variational formulation of the problem presented by \cite{Peskin_1977_Numerical_0} does not necessitate the approximation of Dirac-$\delta$ distributions. The explicit presence of Dirac-$\delta$ distributions pertaining to the body force system ``disappears'' naturally in the weak formulation.  As far as the motion of the solid is concerned, the use of the Dirac-$\delta$ distribution is unnecessary because the finite element solution for the fluid velocity field can be evaluated over the solid domain on a pointwise basis.  The thrust of the work by \cite{WangLiu_2004_Extended_0} and \cite{ZhangGerstenberger_2004_Immersed_0} was instead that of removing the requirement that the immersed solid be a ``boundary.''  The methods proposed in \cite{WangLiu_2004_Extended_0} and \cite{ZhangGerstenberger_2004_Immersed_0} apply to solid bodies of general topological and constitutive characteristics.  However, these approaches still require an approximation of the Dirac-$\delta$ distribution as a function.  Specifically,  \cite{WangLiu_2004_Extended_0} and \cite{ZhangGerstenberger_2004_Immersed_0} rely on the reproducing kernel particle method (\acro{RKPM}) to approximate the Dirac-$\delta$ distribution both for the expression of the interaction forces and for the determination of the velocity of the immersed solid.  For future reference, we point out that the work by \cite{WangLiu_2004_Extended_0} and \cite{ZhangGerstenberger_2004_Immersed_0} pertains to systems consisting of a nearly incompressible solid body immersed in a Newtonian fluid.

The generalization of the approach proposed by \cite{BoffiGastaldi_2003_A-Finite_0} to include regular solid bodies, as opposed to boundaries, has been presented in various publications culminating in the works by \cite{Heltai_2006_The-Finite_0}, \cite{BoffiGastaldiHeltaiPeskin-2008-a}, and~\cite{Heltai-2008-a}, (see bibliographic references in these publications for details).  The constitutive behavior of the immersed solid is assumed to be visco-elastic with the viscous component of the solid stress response being identical to that of the fluid.  Another restrictive assumption of this work was the assumption that the fluid and the solid have the same mass density distribution. From the viewpoint of the treatment of the interaction forces and of the velocity equation for the solid body, these works show that the \acro{FEM} allows one to completely avoid dealing with Dirac-$\delta$ distributions and their approximation.  They also show that the velocity field of the solid domain can be determined variationally, i.e., in a way that is consistent with the \acro{FEM} as a whole.  However, the strength of this idea is not fully demonstrated by \cite{BoffiGastaldiHeltaiPeskin-2008-a}.  This is because they choose a finite element discretization of the solid domain for which the motion of the solid is determined by a direct evaluation the fluid's velocity at the vertices of the grid supporting the discretization in question.  The advantage of a fully variational formulation for immersed methods is the fact the \acro{FEM} machinery offering transparent stability results and error estimates becomes readily available along with the machinery developed for adaptivity.

A fully variational formulation of the immersed problem that does not rely on any approximation of the Dirac-$\delta$ distribution has also been formally discussed by \cite{LiuKim_2007_Mathematical_0} (see Eq.~(40) on p.~215).  However, it is not clear whether or not the variational formalism of \cite{LiuKim_2007_Mathematical_0} has been implemented in actual calculations.  Another fully variational formulation of an immersed method has been proposed by \cite{BlancoFeijo_2008_A-Variational_0}.  This formulation can also cope with a variety of constitutive assumptions for the fluid and solid domains.  While some numerical results have been published for the case of solid structures having incompatible kinematic assumptions (see \citealp{BlancoFeijo_2008_A-Variational_1}), no numerical results seems to have been published for \acro{FSI} problems.

Here we present the generalization of the approach discussed by \cite{BoffiGastaldiHeltaiPeskin-2008-a} so as to be applicable to the case of general visco-elastic compressible and incompressible bodies immersed in an incompressible fluid. The proposed scheme produces a discretization which is strongly consistent and stable, and can easily be extended to the case in which the fluid is also compressible.

As mentioned earlier, in immersed methods the velocity of the solid is set equal to the restriction to the solid domain of the velocity of the fluid.  When enforced variationally, this equality is weakened and transport theorems underlying the classical energy estimates typically obtained in continuum mechanics do not hold any longer.  While the classical transport theorems cannot be invoked directly, we show that energy estimates and corresponding stability results can be obtained for our proposed abstract weak formulation that are formally identical to the classical ones from continuum mechanics.

The proposed variational formulation produces a natural discretization scheme which differs from the one presented by \cite{BoffiGastaldiHeltaiPeskin-2008-a} in the determination of the motion of the solid. In \cite{BoffiGastaldiHeltaiPeskin-2008-a} the velocity field is evaluated  at the discrete level on the vertices of the solid mesh. This procedure renders semi-discrete and discrete stability estimates nontrivial for general approximating spaces of the solid displacement. By contrast, the stability results we prove in the abstract weak formulation are inherited naturally by the discretization scheme, provided that conforming approximating spaces are used for the velocity and displacement fields, thus removing the assumptions on the the triangulation of the solid that were present in \cite{BoffiGastaldiHeltaiPeskin-2008-a}.

Another original contribution of our formulation is that the treatment of the case of a compressible solid in an incompressible fluid is not taken as a limit case of some other set of constitutive assumptions.  This is an important detail in that, again to the best of the authors' knowledge, other approaches have dealt with solid/fluid kinematical incompatibilities indirectly, i.e., as limit cases corresponding to some particular value of a tunable parameter.

In Section~\ref{sec: Problem Formulation}, we present a formulation of the equations of motion of an immersed solid in the context of classical continuum mechanics (i.e., under strong smoothness assumptions).  We also offer a concise exposition of the transport theorems and associated energy estimates that are valid in the aforementioned classical context.  In Section~\ref{sec: Reformulation of the governing equations} we reformulate the problem in variational form and present a discussion of the formulation's underlying functional setting.  We then prove the that proposed formulation is stable.  In Section~\ref{sec: Discretization} we present the discrete formulation we derive from the proposed abstract variational formulation and show that the discrete formulation is strongly consistent and inherits the stability of the abstract formulation.  Some numerical results are presented in Section~\ref{sec: Numerics}.

\section{Classical Formulation}
\label{sec: Problem Formulation}

\subsection{Basic notation and governing equations}
\label{subsec: Basic notation and governing equations}
Referring to Fig.~\ref{fig: current_configuration}, 
%%%
\begin{figure}[htb]
    \centering
    \includegraphics{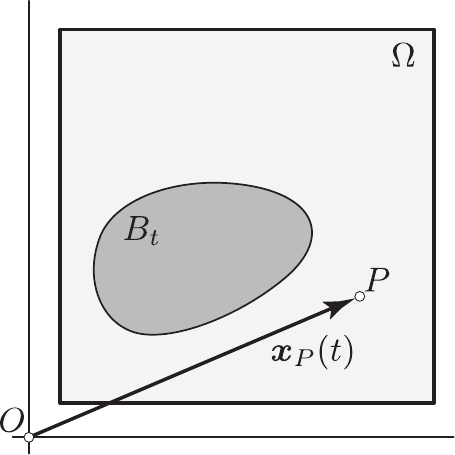}
    \caption{Current configuration $B_{t}$ of a body $\mathscr{B}$ immersed in a fluid occupying the domain $\Omega$.}
    \label{fig: current_configuration}
\end{figure}
%%%
$B_{t}$ represents the configuration of a solid body $\mathscr{B}$ at time $t$.  As a point set, $B_{t}$ is a (possibly multiply connected) proper subset of a fixed control volume $\Omega$.  The domain $\Omega\setminus  B_{t}$ is occupied by a fluid and we refer to $B_{t}$ as the \emph{immersed body}.  The boundaries of $\Omega$ and $B_{t}$, with outer unit normals $\bv{m}$ and $\bv{n}$, respectively, will be denoted by $\partial\Omega$ and $\partial B_{t}$.  For convenience, we select a configuration of $\mathscr{B}$ as a reference configuration and we denote it by $B$.  Both $B$ and $B_{t}$ are viewed as submanifolds of a same Euclidian manifold $\mathscr{E}^{d}$, of dimension $d$ equal to $2$ or $3$, covered by a single rectangular Cartesian coordinate system with origin at $O$.  We denote the position of points of $\mathscr{B}$ in $B$ by $\bv{s}$, whereas we denote the position at time $t$ of a generic point $P \in \Omega$ by $\bv{x}_{P}(t)$.  A motion of $\mathscr{B}$ is a diffeomorphism $\bv{\zeta}: B \to B_{t}$, $\bv{x} = \bv{\zeta}(\bv{s},t)$, with $\bv{s} \in B$, $\bv{x} \in \Omega$, and $t \in [0,T)$, with $T$ a positive real number.

We denote by $\rho(\bv{x}, t)$ the spatial (or Eulerian) description of the mass density at the location $\bv{x}$ at time $t$.  The function $\rho$ can be discontinuous across $\partial B_{t}$.  The local form of the balance of mass requires that, $\forall t \in (0,T)$,
\begin{equation}
\label{eq: Balance of mass}
\dot{\rho} + \rho \ldiv\bv{u} = 0,\quad \bv{x} \in \Omega \setminus (\partial\Omega \cup \partial B_{t}),
\end{equation}
where $\bv{u}(\bv{x},t) = \partial\bv{\zeta}(\bv{s},t)/\partial t \big|_{\bv{s} = \bv{\zeta}^{-1}(\bv{x},t)}$ is the spatial description of the material velocity field, a dot over a quantity denotes the material time derivative of that quantity,%
\footnote{\label{footnote: material time derivative}In continuum mechanics, a physical body is assumed to consist of \emph{material points}.  Each of these is considered an individual entity whose position is a function of time and at which physical quantities, such as mass density or momentum density, can be defined.  By definition, the material time derivative of a property of a material point (e.g., momentum), is the time rate of change of that property measured while following the material point in question.  The material time derivative of a (scalar-, vector-, or tensor-valued) field of the type $\phi = \phi(\bv{s},t)$, with $\bv{s} \in B$, is simply $\dot{\phi} = \partial \phi/\partial t$.  In the case of a scalar-valued function $\psi = \psi(\bv{x},t)$, with $\bv{x} \in \Omega$, $\dot{\psi} = \partial\psi/\partial t + (\grad \psi) \cdot \bv{u}$, where `$\grad$' is the gradient with respect to $\bv{x}$,  $\bv{u}(\bv{x},t)$ is the (material) velocity field, and `$\cdot$' denotes the standard inner product for vectors fields.  For a vector-valued function $\bv{w}(\bv{x},t)$, $\dot{\bv{w}} = \partial\bv{w}/\partial t + (\grad \bv{w}) \bv{u}$, where `$(\grad \bv{w}) \bv{u}$' denotes the action of the second order tensor $\grad \bv{w}$ on the velocity field $\bv{u}$.}
and where `$\ldiv$' represents the divergence operator with respect to $\bv{x}$.

We denote by $\tensor{T}(\bv{x},t)$ the spatial description of the Cauchy stress.  The local form of the momentum balance laws require that, $\forall t \in (0,T)$, $\tensor{T} = \trans{\tensor{T}}$ (the superscript $\mathrm{T}$ denotes the transpose) and
\begin{equation}
\label{eq: Cauchy theorem}
\ldiv \tensor{T} + \rho \bv{b} = \rho \dot{\bv{u}},
\quad \bv{x} \in \Omega \setminus (\partial\Omega \cup \partial B_{t}),
\end{equation}
where $\bv{b}(\bv{x},t)$ describes the external force density per unit mass acting on the system.

In addition to Eqs.~\eqref{eq: Balance of mass} and~\eqref{eq: Cauchy theorem}, we also require the satisfaction of some continuity conditions across $\partial B_{t}$.  Specifically, we demand that the velocity field be continuous (corresponding to a no slip condition between solid and fluid) and that the jump condition of the balance of linear momentum be satisfied across $\partial B_{t}$.  For all $t \in (0,T)$, these two conditions can be expressed as follows:
\begin{equation}
\label{eq: jump conditions}
\bv{u}(\check{\bv{x}}^{+},t)  = \bv{u}(\check{\bv{x}}^{-},t)
\quad \text{and} \quad
\tensor{T}(\check{\bv{x}}^{+},t) \bv{n} = \tensor{T}(\check{\bv{x}}^{-},t) \bv{n},
\quad \check{\bv{x}} \in \partial B_{t},
\end{equation}
where the superscripts $-$ and $+$ denote limits as $\bv{x} \to \check{\bv{x}}$ from within and without $B_{t}$, respectively.

We denote by $\partial\Omega_{D}$ and $\partial\Omega_{N}$ the subsets of $\partial\Omega$ where Dirichlet and Neumann boundary data are  prescribed, respectively. The domains $\partial\Omega_{D}$ and $\partial\Omega_{N}$ are such that
\begin{equation}
\label{eq: ND boundary}
\partial\Omega = \partial\Omega_{D} \cup \partial\Omega_{N}
\quad \text{and} \quad
\partial\Omega_{D} \cap \partial\Omega_{N} = \emptyset.
\end{equation}
We denote by $\bv{u}_{g}(\bv{x},t)$, with $\bv{x} \in \partial\Omega_{D}$, and by $\bv{\tau}_{g}(\bv{x},t)$, with $\bv{x}\in\partial\Omega_{N}$, the prescribed values of velocity (Dirichlet data) and traction (Neumann data), respectively, i.e.,
\begin{equation}
\label{eq: boundary conditions}
\bv{u}(\bv{x},t) = \bv{u}_{g}(\bv{x},t),\quad \text{for $\bv{x} \in \partial\Omega_{D}$,}
\quad \text{and} \quad
\tensor{T}(\bv{x},t) \bv{m}(\bv{x},t) = \bv{\tau}_{g}(\bv{x},t), \quad \text{for $\bv{x} \in \partial\Omega_{N}$,}
\end{equation}
where the subscript $g$ stands for `given.'

Using the principle of virtual work and letting $\bv{v}$ denote any admissible variation of the field $\bv{u}$, Eqs.~\eqref{eq: Cauchy theorem} and~\eqref{eq: jump conditions} can be written as follows:
\begin{equation}
\label{eq: Bmomentum weak}
\int_{\Omega}\rho (\dot{\bv{u}} - \bv{b}) \cdot \bv{v} \d{v}
+
\int_{\Omega}\tensor{T} \cdot \grad\bv{v} \d{v}
-
\int_{\partial\Omega_{N}} \bv{\tau}_{g} \cdot \bv{v} \d{a}
=
0,
\end{equation}
where $\d{a}$ and $\d{v}$ represent infinitesimal area and volume elements, respectively.  We can reformulate Eq.~\eqref{eq: Balance of mass} in variational form as follows:
\begin{equation}
\label{eq: Bmass weak}
\int_{\Omega} \biggl(\frac{\dot{\rho}}{\rho} + \ldiv\bv{u}\biggr) q \d{v} = 0,
\end{equation}
where, from a physical viewpoint, $q$ represents any admissible variation of the pressure in the system.  In the case of incompressible materials, $\dot{\rho} = 0$ and Eq.~\eqref{eq: Bmass weak} yields the traditional weak form of the incompressibility constraint, namely, $\int_{\Omega} q\ldiv \bv{u} \d{v} = 0$.

\subsection{Constitutive behavior}
\label{subsec: Constitute behavior}

\subsubsection{Constitutive response of the fluid.}
\label{subsec: Constitute response of the fluid}
We assume that the fluid is linear viscous and incompressible with uniform mass density $\rho_{\f}$.  Denoting by $\hat{\tensor{T}}_{\f}$ the constitutive response function of the Cauchy stress of the fluid, we have (see, e.g., \citealp{GurtinFried_2010_The-Mechanics_0})
\begin{equation}
\label{eq: incompressible NS fluid}
\hat{\tensor{T}}_{\f} = -p \tensor{I} + 2 \mu_{\f} \tensor{D},
\quad
\tensor{D} = \tfrac{1}{2} \bigl(\tensor{L} + \trans{\tensor{L}}\bigr),
\end{equation}
where $p$ is the pressure of the fluid, $\tensor{I}$ is the identity tensor, $\mu_{\f} > 0$ is the dynamic viscosity of the fluid, and $\tensor{L} = \grad \bv{u}$, and where a ``hat'' ($\hat{\tensor{T}}$) is used to distinguish the constitutive response function for $\tensor{T}$ from $\tensor{T}$ itself.  For convenience, we denote by $\hat{\tensor{T}}^{v}_{\f}$ the viscous component of $\hat{\tensor{T}}_{\f}$, i.e.,
\begin{equation}
\label{eq: viscous T fluid}
\hat{\tensor{T}}^{v}_{\f} = 2 \mu_{\f} \, \tensor{D} = \mu_{\f} \, \bigl(\tensor{L} + \trans{\tensor{L}}\bigr).
\end{equation}

Incompressibility demands that $\dot{\rho}_{\f} = 0$ so that Eq.~\eqref{eq: Balance of mass} yields the kinematic constraint
\begin{equation}
\label{eq: incompressibility constraint fluid}
\ldiv \bv{u} = 0 \quad \text{for $\bv{x} \in \Omega\setminus B_{t}$}.
\end{equation}
Under these conditions, $p$ is a Lagrange multiplier allowing us to enforce Eq.~\eqref{eq: incompressibility constraint fluid}.  In addition, Eq.~\eqref{eq: incompressibility constraint fluid} also implies that $\trace{\tensor{L}} = 0$ so that the term $\hat{\tensor{T}}^{v}_{\f}$ in Eqs.~\eqref{eq: incompressible NS fluid} is the deviatoric part of the Cauchy stress in the fluid.

\subsubsection{Constitutive response of the solid.}
\label{subsec: Constitute response of the solid}
We assume that the body $\mathscr{B}$ is viscoelastic of differential type.  The response function for the Cauchy stress of the solid is assumed to have the following form:
\begin{equation}
\label{eq: solid Cauchy Response Function}
\hat{\tensor{T}}_{\s} = \hat{\tensor{T}}^{e}_{\s} + \hat{\tensor{T}}^{v}_{\s},
\end{equation}
where $\hat{\tensor{T}}^{e}_{\s}$ and $\hat{\tensor{T}}^{v}_{\s}$ denote the elastic and viscous parts of $\hat{\tensor{T}}_{\s}$, respectively. The viscous part of the behavior is assumed to be of the same type as that of the fluid, that is,
\begin{equation}
\label{eq: viscous part solid}
\hat{\tensor{T}}^{v}_{\s} = 2 \mu_{\s} \, \tensor{D} = \mu_{\s} \, \bigl(\tensor{L} + \trans{\tensor{L}} \bigr),
\end{equation}
where $\mu_{\s} \geq 0$ is the dynamic viscosity of the solid.  We do include the possibility that $\mu_{\s}$ might be equal to zero, in which case the solid behaves in a purely elastic manner.  As far as $\hat{\tensor{T}}^{e}_{\s}$ is concerned, we assume that it is given by a strain energy potential.  To describe this part of the behavior in precise terms, we introduce the first Piola-Kirchhoff stress tensor, denoted by $\tensor{P}$ and defined as (see, e.g., \citealp{GurtinFried_2010_The-Mechanics_0}):
\begin{equation}
\label{eq: P defs}
\tensor{P} = J \tensor{T} \invtrans{\tensor{F}},
\end{equation}
where $J = \det\tensor{F}$, and the tensor $\tensor{F}$, called the deformation gradient, is defined as
\begin{equation}
\label{eq: F defs}
\tensor{F} = \frac{\partial \map}{\partial \bv{s}}.
\end{equation}
As is standard in continuum mechanics (see, e.g., \citealp{GurtinFried_2010_The-Mechanics_0}), we require $J$ to satisfy the following assumption:
\begin{equation}
\label{eq: J positive}
J(\bv{s},t) \geq J_m > 0
\quad
\text{$\forall \bv{s} \in B$ and $\forall t \in [0,t)$}.
\end{equation}
Therefore, $\tensor{F}$ always admits an inverse, as required for Eq.~\eqref{eq: P defs} to be meaningful.  Hence, letting $\hat{\tensor{P}}^{e}_{\s} = J \hat{\tensor{T}}^{e}_{\s}\invtrans{F}$ denote the constitutive response function for the elastic  part of the first Piola-Kirchhoff stress tensor, as is typical in  elasticity, we assume that there exists a function $\hat{W}^{e}_{\s}(\tensor{F})$ such that
\begin{equation}
\label{eq: Elastic 1stPK stress}
\hat{\tensor{P}}^{e}_{\s} = \frac{\partial\hat{W}^{e}_{\s}(\tensor{F})}{\partial{\tensor{F}}},
\end{equation}
where $\hat{W}^{e}_{\s}$ is the constitutive response function of the volume density of the elastic strain energy of the solid.  To satisfy invariance under changes of observer, $\hat{W}^{e}_{\s}$ must be a function of an objective strain measure such as $\tensor{C} = \trans{\tensor{F}}\tensor{F}$.  In addition, if the solid is isotropic, $\hat{W}^{e}_{\s}$ will be taken to be a function of the principal invariants of $\tensor{C}$.  Finally, if the solid is incompressible, then its stress response is determined by deformation only up to a hydrostatic component.  In this case, the constitutive response function for the solid has the form
\begin{equation}
\label{eq: Cauchy Response Function}
\hat{\tensor{T}}_{\s} = -p \tensor{I} + \hat{\tensor{T}}^{e}_{\s} + \hat{\tensor{T}}^{v}_{\s},
\end{equation}
where $p$ is the Lagrange multiplier needed to enforce incompressibility, $\hat{\tensor{T}}^{v}_{\s}$ is given by Eq.~\eqref{eq: viscous part solid}, and $\hat{\tensor{T}}^{e}_{\s}$ is obtained from Eq.~\eqref{eq: Elastic 1stPK stress} via Eq.~\eqref{eq: P defs}.

\subsubsection{Elastic strain energy and dissipation}
\label{subsubsection: Elastic strain energy and dissipation}
While more general cases can be considered, we assume that $\hat{W}_{\s}^{e}(\tensor{F})$ is a $C^{1}$ convex function over the set of second order tensor with positive determinant.  As far as the viscous part of the behavior is concerned, we have already assumed that $\mu_{\f} > 0$ and $\mu_{\s} \geq 0$.  These conditions imply that
\begin{equation}
\label{eq: dissipations inequality}
\hat{\tensor{T}}^{v}_{\f} \cdot \tensor{L} > 0
, \qquad
\hat{\tensor{T}}^{v}_{\s} \cdot \tensor{L} \geq 0
\end{equation}
for all $\tensor{L} \neq \tensor{0}$.  Equations~\eqref{eq: dissipations inequality} imply that the viscous part of the behavior is dissipative.

\subsubsection{Mass density distribution}
\label{subsec: Mass density distribution}
As a last aspect of the formulation related to constitutive behavior, we will denote by 
\begin{equation}
\label{eq: density solid}
\rho_{\s_{0}} = \rho_{\s_{0}}(\bv{s}), \quad \bv{s} \in B,
\end{equation}
the referential (or Lagrangian) description of the mass density of the solid.  While Eq.~\eqref{eq: Balance of mass} holds for the solid as well as the fluid, the local form of the balance of mass for a solid is typically expressed in Lagrangian form as follows:
\begin{equation}
\label{eq: BM solid Lagrangian}
\rho_{\s_{0}}(\bv{s}) = \rho_{\s}(\bv{x},t)\big|_{\bv{x} = \bv{\zeta}(\bv{s},t)} J(\bv{s},t),\quad \bv{s} \in B,
\end{equation}
where $\rho_{\s}(\bv{x},t)$ is the spatial description of the mass density of the solid.  We will indicate the general mass density of the system with $\rho = \rho(\bv{x},t)$, with the underlying assumption that
\begin{equation}
  \label{eq:definition of overall rho}
  \rho(\bv{x},t) = \begin{cases}
\rho_{\f}, & \text{for $\bv{x} \in \Omega\setminus B_{t}$},
\\
\rho_{\s}(\bv{x},t), & \text{for $\bv{x} \in B_{t}$},
\end{cases}
\end{equation}
where, as stated earlier, $\rho_{\f}$ is a constant.

\subsection{Transport theorems}
\label{subsec: Transport}
Transport theorems are kinematic results pertaining to the time differentiation of integrals over time-dependent domains.  These results are  useful in the discussion of energy estimates.
\begin{theorem}[Transport theorem for generic time dependent domains]
\label{th: GTT}
Let $\tilde{\Omega}(t) \in \mathscr{E}^{d}$, with $d = 2, 3$ and $t \in \mathbb{R}$, be a regular, possibly multiply-connected time-dependent domain with boundary $\partial\tilde{\Omega}(t)$.  Let $\bv{m}$ be the unit normal field orienting $\partial\tilde{\Omega}(t)$, outward with respect to $\tilde{\Omega}(t)$.  Let $\bv{\nu}$ be the velocity of $\partial\tilde{\Omega}(t)$ according to some convenient time-parametrization of $\partial\tilde{\Omega}(t)$.  Let $\phi(\bv{x},t)$, with $\bv{x} \in \tilde{\Omega}(t)$ be a smooth field defined over $\tilde{\Omega}(t)$.  Then we have
\begin{equation}
\label{eq: General transport theorem}
\frac{\nsd{}}{\nsd{t}} \int_{\tilde{\Omega}(t)} \phi(\bv{x},t) \d{v} = \int_{\tilde{\Omega}(t)} \frac{\partial \phi(\bv{x},t)}{\partial t} \d{v} + \int_{\partial\tilde{\Omega}(t)} \phi(\bv{x},t) \, \bv{\nu} \cdot \bv{m} \d{a}.
\end{equation}
\end{theorem}
Theorem~\ref{th: GTT} is is a well-known result whose proof is available in various textbooks (see, e.g., \citealp{TruesdellToupin-CFTEP-1960-1,GurtinFried_2010_The-Mechanics_0}).  The following form of the transport theorem is a simple but new result that is particularly suited for the analysis of the motion of immersed bodies.  The proof of the theorem below can be found in the Appendix.

\begin{theorem}[Transport theorem for a control volume containing an immersed domain]
\label{th: immersed in control volume}
Let $\Omega$ and $B_{t}$ be the domains defined in Section~\ref{subsec: Basic notation and governing equations}.  That is, let $B_{t}$ be the current configuration of a body immersed in the control volume $\Omega$.  Let $\psi(\bv{x},t)$ denote the Eulerian description of a density per unit mass defined over $\Omega \supset B_{t}$, smooth over the interiors of $\Omega\setminus B_{t}$ and $B_{t}$ but not necessarily continuous across $\partial B_{t}$. Also let $\rho(\bv{x},t)$ be the Eulerian description of the mass density distribution, which need not be continuous across $\partial B_{t}$.  Then
\begin{equation}
\label{eq: TT BM Omega and Bt}
\frac{\nsd{}}{\nsd{t}} \int_{\Omega} \rho \psi \d{v} + \int_{\partial\Omega} \rho \psi \, \bv{u} \cdot \bv{m} \d{a} = \int_{\Omega} \rho \dot{\psi} \d{v}.
\end{equation}
\end{theorem}

\begin{remark}[Generality of Theorem~\ref{th: immersed in control volume}]
Theorem~\ref{th: immersed in control volume} is a straightforward but nontrivial result implied by the combined application of Theorem~\ref{th: GTT} and the balance of mass.  One crucial element of Theorem~\ref{th: immersed in control volume} is that no special assumption on the behavior of the mass density was necessary.  That is, Theorem~\ref{th: immersed in control volume} is valid whether or not $\rho$ is constant or the fluid flow is steady.
\end{remark}

\subsection{Theorem of power expended}
\label{subsec: Theorem of power expended}
In this paper we propose an immersed method for the numerical solution of the problem governed by Eqs.~\eqref{eq: Bmomentum weak} and~\eqref{eq: Bmass weak} under standard physical assumptions concerning the constitutive behavior of the fluid and of the immersed solid.  We will discuss energy estimates and associated stability properties for the proposed method.  To facilitate this discussion, it is useful to relate the power supplied to the system and the system's time rate of change of kinetic energy.  Such a relationship is typically referred to as the theorem of power expended (see, e.g., \citealp{GurtinFried_2010_The-Mechanics_0}).  Here we derive a form of the theorem of power expended that fits our purposes.  Before doing so we introduce the following definitions:
\begin{equation}
\label{eq: TPE defs}
\kappa(\bv{x},t) := \tfrac{1}{2} \rho(\bv{x},t) \bv{u}^{2}(\bv{x},t)
\quad \text{and} \quad
\hat{\tensor{T}}^{v}(\bv{x},t) = 
\begin{cases}
\hat{\tensor{T}}^{v}_{\f}, & \text{for $\bv{x} \in \Omega\setminus B_{t}$},
\\
\hat{\tensor{T}}^{v}_{\s}, & \text{for $\bv{x} \in B_{t}$},
\end{cases}
\end{equation}
where $\kappa$ is the kinetic energy density per unit volume and $\bv{u}^{2} := \bv{u} \cdot \bv{u}$.

\begin{theorem}[Theorem of power expended for a control volume with an immersed domain]
\label{th: TPE}
Let $\Omega$ and $B_{t}$ be the domains defined in Section~\ref{subsec: Basic notation and governing equations}.  That is, let $B_{t}$ be the current configuration of a body immersed in the control volume $\Omega$.  Let the motion of the system be governed by Eqs.~\eqref{eq: Bmomentum weak} and~\eqref{eq: Bmass weak}.  Then
\begin{equation}
\label{eq: TPE rel}
\int_{\Omega} \rho \bv{b} \cdot \bv{u} \d{v} + \int_{\partial\Omega_{N}} \bv{\tau}_{g} \cdot \bv{u} \d{a} = \frac{\nsd{}}{\nsd{t}} \int_{\Omega} \kappa \d{v} + \int_{\partial\Omega} \kappa \, \bv{u} \cdot \bv{m} \d{a} + \frac{\nsd{}}{\nsd{t}} \int_{B} W^{e}_{\s} \d{V} + \int_{\Omega} \hat{\tensor{T}}^{v} \cdot \tensor{L}  \d{v}.
\end{equation}
\end{theorem}

\begin{proof}[Proof of Theorem~\ref{th: TPE}]
Replacing $\bv{v}$ with $\bv{u}$ in Eq.~\eqref{eq: Bmomentum weak} and rearranging, we obtain
\begin{equation}
\label{eq: TPE directly from PVW}
\int_{\Omega} \rho \bv{b} \cdot \bv{u} \d{v} + \int_{\Omega_{N}} \bv{\tau}_{g} \cdot \bv{u} \d{a} = \int_{\Omega} \rho \dot{\bv{u}} \cdot \bv{u} \d{v} + \int_{\Omega} \tensor{T} \cdot \tensor{L} \d{v},
\end{equation}
where we have used the fact that $\tensor{L} = \grad \bv{u}$.  We now observe that
\begin{equation}
\label{eq: TPE kin energy}
\rho\dot{\bv{u}} \cdot \bv{u} = \tfrac{1}{2} \rho \, \dot{\overline{\bv{u}^{2}}},
\end{equation}
where the line over $\bv{u}^{2}$ simply denotes the fact that the material time derivative (denoted by the dot over the line) must be applied to the quantity under the line, namely, $\bv{u}^{2}$.  Therefore, recalling that, by the first of Eqs.~\eqref{eq: TPE defs}, $\kappa = \tfrac{1}{2} \rho \bv{u}^{2}$,  we have that
\begin{equation}
\label{eq: TPE KE int first}
\int_{\Omega} \rho \dot{\bv{u}} \cdot \bv{u} \d{v}
= 
\int_{\Omega} \tfrac{1}{2} \rho \, \dot{\overline{\bv{u}^{2}}} \d{v}
\quad \Rightarrow \quad
\int_{\Omega} \rho \dot{\bv{u}} \cdot \bv{u} \d{v}
= 
\frac{\nsd{}}{\nsd{t}} \int_{\Omega} \kappa \d{v} + \int_{\partial\Omega} \kappa \bv{u} \cdot \bv{m} \d{a},
\end{equation}
where, to obtain this last expression, we have used Theorem~\ref{th: immersed in control volume}.  Next, using the constitutive equations in Section~\ref{subsec: Constitute behavior}, for $\bv{x} \in \Omega\setminus B_{t}$, i.e., in the fluid, we have that
\begin{equation}
\label{eq: TPE fluid stress power}
\tensor{T} \cdot \tensor{L} = -p \tensor{I} \cdot \tensor{L} + \hat{\tensor{T}}^{v}_{\f} \cdot \tensor{L}
\quad \Rightarrow \quad
\tensor{T} \cdot \tensor{L} = -p \ldiv \bv{u} + \hat{\tensor{T}}^{v}_{\f} \cdot \tensor{L}
\quad \Rightarrow \quad
\tensor{T} \cdot \tensor{L} = \hat{\tensor{T}}^{v}_{\f} \cdot \tensor{L},
\end{equation}
where we have used the fact that, in the fluid, $\ldiv \bv{u} = 0$ due to incompressibility.  For $\bv{x} \in B_{t}$, i.e., in the solid, we would normally have to distinguish between the compressible and incompressible cases.  However, the final result is the same due to the fact that, in the incompressible case, the Lagrange multiplier $p$ does not contribute to the stress power as was shown in Eqs.~\eqref{eq: TPE fluid stress power}.  Therefore in the solid we have
\begin{equation}
\label{eq: TPE solid stress power}
\tensor{T} \cdot \tensor{L} = \hat{\tensor{T}}^{v}_{\s} \cdot \tensor{L} + \hat{\tensor{T}}^{e}_{\s} \cdot \tensor{L}.
\end{equation}
Using Eqs.~\eqref{eq: TPE fluid stress power} and~\eqref{eq: TPE solid stress power} along with the definition in the second of Eqs.~\eqref{eq: TPE defs}, whether the solid is compressible or not, we have
\begin{equation}
\label{eq: TPE overall stress power}
\begin{multlined}
\int_{\Omega} \tensor{T} \cdot \tensor{L} \d{v} = \int_{\Omega\setminus B_{t}} \hat{\tensor{T}}^{v}_{\f} \cdot \tensor{L} \d{v}
+ 
\int_{B_{t}} \hat{\tensor{T}}^{v}_{\s} \cdot \tensor{L} \d{v}
+ \int_{B_{t}} \hat{\tensor{T}}^{e}_{\s} \cdot \tensor{L} \d{v}
\\
\Rightarrow \quad
\int_{\Omega} \tensor{T} \cdot \tensor{L} \d{v} = \int_{\Omega} \hat{\tensor{T}}^{v} \cdot \tensor{L} \d{v}
+ \int_{B_{t}} \hat{\tensor{T}}^{e}_{\s} \cdot \tensor{L} \d{v}.
\end{multlined}
\end{equation}
Next, recalling that $\tensor{L} = \dot{\tensor{F}}\tensor{F}^{-1}$, we recall that
\begin{equation}
\label{eq: TPE pull back elastic work}
\int_{B_{t}} \hat{\tensor{T}}^{e}_{\s} \cdot \tensor{L} \d{v}
= 
\int_{B} J \hat{\tensor{T}}^{e}_{\s} \cdot \dot{\tensor{F}} \tensor{F}^{-1} \d{V}
\quad \Rightarrow \quad
\int_{B_{t}} \hat{\tensor{T}}^{e}_{\s} \cdot \tensor{L} \d{v}
= 
\int_{B} \hat{\tensor{P}}^{e}_{\s} \cdot \dot{\tensor{F}} \d{V},
\end{equation}
where we have used Eq.~\eqref{eq: P defs} along with the tensor identity $\tensor{A} \cdot \tensor{B} \tensor{C} = \tensor{A} \trans{\tensor{C}} \cdot \tensor{B}$.  Using Eq.~\eqref{eq: Elastic 1stPK stress} we see that $\hat{\tensor{P}}^{e}_{\s} \cdot \dot{\tensor{F}} = \dot{\hat{W}}^{e}_{\s}$, so that, combining the results in the last of Eqs.~\eqref{eq: TPE overall stress power} and~\eqref{eq: TPE pull back elastic work}, we can write
\begin{equation}
\label{eq: TPE almost done stress power}
\int_{\Omega} \tensor{T} \cdot \tensor{L} \d{v} = \int_{\Omega} \hat{\tensor{T}}^{v} \cdot \tensor{L} \d{v}
+
\int_{B} \dot{\hat{W}}^{e}_{\s} \d{V}.
\end{equation}
We recall that $\hat{W}^{e}_{\s} = \hat{W}^{e}_{\s}(\bv{s},t)$, $\bv{s} \in B$, so that $\dot{\hat{W}}^{e}_{\s} = \partial\hat{W}^{e}_{\s}/\partial t$.  Therefore, observing that $B$ is a fixed domain (with fixed boundary), using Theorem~\ref{th: GTT} with the identification $\tilde{\Omega} \to B$, we can express Eq.~\eqref{eq: TPE almost done stress power} as follows:
\begin{equation}
\label{eq: TPE almost done stress power two}
\int_{\Omega} \tensor{T} \cdot \tensor{L} \d{v} = \int_{\Omega} \hat{\tensor{T}}^{v} \cdot \tensor{L} \d{v}
+
\frac{\nsd{}}{\nsd{t}} \int_{B} \hat{W}^{e}_{\s} \d{V}.
\end{equation}
The proof can be now concluded by substituting the last of Eqs.~\eqref{eq: TPE KE int first} and~\eqref{eq: TPE almost done stress power two} into Eq.~\eqref{eq: TPE directly from PVW}, which yields Eq.~\eqref{eq: TPE rel}.
\end{proof}

\begin{lemma}[Dissipation inequality]
\label{lemma: DI}
Referring to Theorem~\ref{th: TPE}, if the system is provided no power input, i.e., if
\begin{equation}
\label{eq: TPE no power input}
\bv{u}_{g} = \bv{0},\quad
\bv{\tau}_{g} = \bv{0},
\quad \text{and} \quad
\bv{b} = \bv{0},
\end{equation}
then, for all admissible motions of the system,
\begin{equation}
\label{eq: TPE dissipation inequality}
\frac{\nsd{}}{\nsd{t}} \int_{\Omega} \kappa \d{v} + \int_{\partial\Omega_{N}} \kappa \, \bv{u} \cdot \bv{m} \d{a} + \frac{\nsd{}}{\nsd{t}} \int_{B} W^{e}_{\s} \d{v} \leq 0.
\end{equation}
\end{lemma}

\begin{proof}[Proof of lemma~\ref{lemma: DI}]
Inequality~\eqref{eq: TPE dissipation inequality} is a direct consequence of Theorem~\ref{th: TPE} and Eqs.~\eqref{eq: dissipations inequality}.
\end{proof}

\begin{remark}[Energy estimates]
\label{remark: energy estimates}
Lemma~\ref{lemma: DI}, plays an important role in that it provides the form of the energy estimates and corresponding stability condition we strive to satisfy in the proposed numerical scheme.
\end{remark}

\section{Abstract Variational Formulation}
\label{sec: Reformulation of the governing equations}
We now reformulate the governing equations as a problem to be solved via a generalization of the approach proposed by \cite{BoffiGastaldiHeltaiPeskin-2008-a}.  For simplicity, we first consider the case with $\mathscr{B}$ incompressible and then the case with $\mathscr{B}$ compressible.  In either case, the principal unknown describing the motion of the solid is the displacement field, denoted by $\bv{w}$ and defined as
\begin{equation}
\label{eq: S disp def}
\bv{w}(\bv{s},t) := \bv{\zeta}(\bv{s},t) - \bv{s},
\quad \bv{s} \in B.
\end{equation}
The displacement gradient relative to the position in $B$ is denoted by $\tensor{H}$:
\begin{equation}
\label{eq: disp grad def}
\tensor{H} := \frac{\partial \bv{w}}{\partial \bv{s}}
\quad \Rightarrow \quad
\tensor{H} = \tensor{F} - \tensor{I}.
\end{equation}
Equation~\eqref{eq: S disp def} implies
\begin{equation}
\label{eq: w u rel}
\dot{\bv{w}}(\bv{s},t) = \bv{u}(\bv{x},t)\big|_{\bv{x} = \bv{\zeta}(\bv{s},t)}.
\end{equation}

\begin{remark}[Eulerian-Lagrangian information exchange and numerical approximation]
\label{remark: peskin delta use}
On the one hand, $\bv{u}(\bv{x},t)$ and $\dot{\bv{w}}(\bv{s},t)$ can be said to carry the same information in that both represent velocity.  On the other, the information carried by $\bv{u}(\bv{x},t)$ and $\dot{\bv{w}}(\bv{s},t)$ is ``packaged'' in fundamentally different ways in that $\bv{u}(\bv{x},t)$ is Eulerian and $\dot{\bv{w}}(\bv{s},t)$ is Lagrangian.  Equation~\eqref{eq: w u rel} ``regulates'' how the information exchange occurs.  As long as pointwise values of $\bv{u}(\bv{x},t)$ are available, given an $\bv{s} \in B$ and having a full-field representation of $\bv{\zeta}(\bv{s},t)$, then the evaluation of Eq.~\eqref{eq: w u rel} is straightforward.  The evaluation of $\dot{\bv{w}}(\bv{s},t)$ is not straightforward when the field $\bv{u}(\bv{x},t)$ is not available \emph{as a field}.  As stated by Peskin (see the beginning of Section~6 in \citealp{Peskin_1977_Numerical_0}),\footnote{In the cited passage, $\bv{x}_{k}$ is a discrete set of points on the immersed boundary at which the forces $\bv{f}_{k}$ responsible for expressing the \acro{FSI} are defined.}
\begin{quote}
The Lagrangian mesh upon which the boundary forces $\bv{f}_{k}$, and the boundary configuration $\bv{x}_{k}$ are stored as points which do not coincide with fluid mesh points. We therefore have the problem of interpolating	the velocity field from the fluid mesh to the boundary points and spreading the boundary forces from the boundary points to the nearly mesh points of the fluid.
\end{quote}
To understand the quote, it is important to recall that Peskin is solving the problem via \acro{FD}.  Therefore $\bv{u}(\bv{x},t)$ is available only as a set of discrete values at the mesh point defining the \acro{FD} solution domain.  To interpolate the discrete values of $\bv{u}$ at a point $\bv{x}_{k}$ on the immersed boundary (not coinciding with the mesh points for the \acro{FD} grid), Peskin presents what is Eq.~\eqref{eq: w u rel} in this paper in terms of the Dirac-$\delta$ distribution (see Eq.~(2.9) in \citealp{Peskin_1977_Numerical_0}):
\begin{equation}
\label{eq: eq 2.9 peskin}
\nsd{\bv{x}}_{k}/\nsd{t} = \bv{u}(\bv{x}_{k},t) = \int_{\bv{x}\in\Omega} \bv{u}(\bv{x},t) \delta(\bv{x} - \bv{x}_{k}) \d{v},
\end{equation}
where $\nsd{\bv{x}}_{k}/\nsd{t}$ corresponds to what we would denote by $\dot{\bv{w}}(\bv{x}_{k},t)$, and where the integral \emph{defines} the action of the Dirac-$\delta$ distribution on the function $\bv{u}(\bv{x},t)$.  In Section~6 of \cite{Peskin_1977_Numerical_0}, the $\delta$ in Eq.~\eqref{eq: eq 2.9 peskin} is replaced by an actual function whose purpose is to \emph{approximate} the behavior of the $\delta$ and allow one to carry out the convolution integral explicitly.  This strategy allows one to interpolate the discrete velocity field information at points that are not on the \acro{FD} grid.  What is important to notice here is that, formally, Eq.~\eqref{eq: eq 2.9 peskin} is Eq.~\eqref{eq: w u rel}, i.e., they serve the same purpose of transferring Eulerian information into Lagrangian information.  Peskin's rationale for choosing to work with Eq.~\eqref{eq: eq 2.9 peskin} vs.\ Eq.~\eqref{eq: w u rel} is due to the nature of his numerical scheme.  We  therefore maintain that any numerical approximation scheme for which the fields $\bv{u}(\bv{x},t)$ and $\bv{\zeta}(\bv{x},t)$ are known \emph{as fields}, does not need to confront the issue of introducing and, \emph{a fortiori}, approximating Dirac-$\delta$ distributions.  In this paper, the immersed problem is solved by \acro{FEM} and therefore it does not require the introduction of the Dirac-$\delta$ distribution either at a formal or at a practical level.  In our proposed approach we enforce Eq.~\eqref{eq: w u rel} weakly, consistently with the variational nature of the solution method we adopt.
\end{remark}

\subsection{Functional setting}
\label{subsec: Functional setting}
The principal unknowns of our fluid-structure interaction problem are the fields
\begin{equation}
\label{eq: unknowns}
\bv{u}(\bv{x},t), \quad
p(\bv{x},t), \quad \text{and} \quad
\bv{w}(\bv{s},t),
\quad\text{with $\bv{x} \in \Omega$, $\bv{s} \in B$, and $t \in [0,T)$.}
\end{equation}
The functional spaces for these fields are selected as follows:
\begin{gather}
\label{eq: functional space u}
\bv{u} \in \mathscr{V} = H_{D}^{1}(\Omega)^{d} := \Bigl\{ \bv{u} \in L^{2}(\Omega)^{d} \,\big|\, \nabla_{\bv{x}} \bv{u} \in L^{2} (\Omega)^{d \times d},  \bv{u}|_{\partial\Omega_{D}} = \bv{u}_{g} \Bigr\},
\\
\label{eq: functional space p}
p \in \mathscr{Q} := L^{2}(\Omega), \\
\label{eq: functional space w}
\bv{w} \in \mathscr{Y} := \Bigl\{ \bv{w} \in L^{2}(B)^{d} \,\big|\, \nabla_{\bv{s}} \bv{w} \in L^{\infty} (B)^{d \times d}  \Bigr\},
\end{gather}
where $\nabla_{\bv{x}}$ and $\nabla_{\bv{s}}$ denote the gradient operators relative to $\bv{x}$ and $\bv{s}$, respectively.

For convenience, we will use a prime to denote partial differentiation with respect to time:
\begin{equation}
\label{eq: prime notation}
\bv{u}'(\bv{x},t) := \frac{\partial \bv{u}(\bv{x},t)}{\partial t}
\quad \text{and} \quad
\bv{w}'(\bv{s},t) := \frac{\partial \bv{w}(\bv{s},t)}{\partial t}.
\end{equation}
Hence, in view of the discussion in the footnote on page~\pageref{footnote: material time derivative},  we have
\begin{equation}
\label{eq: material time derivatives and primes}
\dot{\bv{u}}(\bv{x},t) = \bv{u}'(\bv{x},t) + \bigl(\nabla_{\bv{x}}\bv{u}(\bv{x},t)\bigr) \bv{u}(\bv{x},t)
\quad \text{and} \quad
\dot{\bv{w}}(\bv{s},t) = \bv{w}'(\bv{s},t).
\end{equation}

\begin{remark}[Domains of definition of the fluid's behavior]
\label{rem: domains of definition}
As in every immersed method, a crucial element of our formulation is the extension of the domain of definition of the fluid's behavior to $\Omega$ as a whole.  The definitions in Eqs.~\eqref{eq: functional
  space u} and~\eqref{eq: functional space p} imply that the fields
$\bv{u}$ and $p$ are defined everywhere in $\Omega$.  Because $\bv{u}$
is defined everywhere in $\Omega$, the function
$\hat{\tensor{T}}^{v}_{\f}$ is defined everywhere in $\Omega$ as
well. For consistency, we must also extend the domain of definition of the mass density of the fluid.  Hence, we formally assume that
\begin{equation}
\label{eq: rho f domain of definition}
\rho_{\f} \in L^{\infty}(\Omega).
\end{equation}
\end{remark}

\begin{remark}[Space of test functions for the velocity]
Referring to Eq.~\eqref{eq: functional space u}, we will denote the function space containing the test functions for the velocity field by $\mathscr{V}_{0}$ and define it as:
\begin{equation}
\label{eq: space of test functions v}
\mathscr{V}_{0} = H_{0}^{1}(\Omega)^{d} := \Bigl\{ \bv{v} \in L^{2}(\Omega)^{d} \,\big|\, \nabla_{\bv{x}} \bv{v} \in L^{2} (\Omega)^{d \times d},  \bv{v}|_{\partial\Omega_{D}} = \bv{0} \Bigr\}.
\end{equation}
\end{remark}

\begin{remark}[Functional spaces for time derivatives]
The functions $\bv{u}'$ and $\bv{w}'$ are generally not expected to be elements of $\mathscr{V}$ and $\mathscr{Y}$, respectively.  Referring to the first term in Eq.~\eqref{eq: Bmomentum weak} and the first of Eq.~\eqref{eq: material time derivatives and primes}, the regularity of $\bv{u}'$ is related to the regularity of the given field $\bv{b}$ and of the boundary conditions.  The field $\bv{b}$ is often assumed to be an element of $H^{-1}(\Omega)$.  The latter can therefore be viewed as a baseline in terms of the minimum regularity that $\bv{u}'$ could have.  However, since the regularity of $\bv{b}$ is not the only factor at play, here we limit ourselves to state that $\bv{u}'$ is an element of a pivot space $\mathscr{H}_{V}$ such that
\begin{equation}
\label{eq: functional space u'}
\mathscr{V}  \subseteq \mathscr{H}_{V} \subseteq \mathscr{H}_{V}^{*} \subseteq \mathscr{V}^{*},
\end{equation}
where $\mathscr{H}_{V}^{*}$ and $\mathscr{V}^{*}$ are the dual spaces of $\mathscr{H}_{V}$ and $\mathscr{V}$, respectively.  We can be more specific in the case of $\bv{w}'$.  We start with saying that $\bv{w}'$ is an element of a pivot space $\mathscr{H}_{Y}$ such that
\begin{equation}
\label{eq: functional space w'}
\mathscr{Y}  \subseteq \mathscr{H}_{Y} \subseteq \mathscr{H}_{Y}^{*} \subseteq \mathscr{Y}^{*},
\end{equation}
where $\mathscr{H}_{Y}^{*}$ and $\mathscr{Y}^{*}$ are the dual spaces of $\mathscr{H}_{Y}$ and $\mathscr{Y}$, respectively.  Then, if Eq.~(\ref{eq: J positive}) is satisfied, using Eq.~\eqref{eq: w u rel} and standard Sobolev inequalities (see, e.g., \citealp{Evans_2010_Partial_0}), we have that, for $\bv{w} \in \mathscr{Y}$ and $\bv{u} \in \mathscr{V}$,
\begin{equation}
\label{eq: identification of HY}
\mathscr{Y}  \subseteq \mathscr{H}_{Y} \subseteq H^{1}(B)^{d}.
\end{equation}
In fact the  $H^1(B)^d$ norm of the displacement velocity can be controlled by
\begin{equation}
  \label{eq: estimate of w' norm in H^1}
  \begin{split} 
    \| \bv{w}' \|^2_{H^1(B)^d} := & \int_B \bv{w}'\cdot\bv{w}' \d V + 
    \int_B \nabla_{\bv{s}}\bv{w}'\cdot \nabla_{\bv{s}}\bv{w}' \d V \\
    = &\int_B (\bv{u}\circ\zeta)^2 \d V + 
    \int_B \Bigl(\bigl((\nabla_{\bv{x}}\bv{u})\circ\zeta\bigr) \tensor{F}\Bigr)^2\d V \\ 
    = &\int_{B_t} (\bv{u})^2 \bigl(J\circ\zeta^{-1}\bigr)^{-2}\d v + 
    \int_{B_t} \Bigl(\nabla_{\bv{x}}\bv{u} \bigl(\tensor{F}\circ\zeta^{-1}\bigr)\Bigr)^2
    \bigl(J\circ\zeta^{-1}\bigr)^{-2}\d v \\
    \leq & J_m^{-2}\left(\|\bv{u}\|^2_{L^2(B_t)^d} + 
      \|\tensor{I}+\nabla_{\bv{s}}\bv{w}\|^2_{L^\infty(B)^{d\times d}} \|\nabla_{\bv{x}}\bv{u}\|^2_{L^2(B_t)^d} 
    \right)
    \\
    \leq & J_m^{-2}\bigl(1+ \|\tensor{I}+\nabla_{\bv{s}}\bv{w}\|^2_{L^\infty(B)^{d\times d}}\bigr)
    \|\bv{u}\|^2_{H^1(B_t)^d} \\
    \leq & J_m^{-2}\bigl(1+\|\bv{w}\|^2_{\mathscr{Y}}\bigr) \| \bv{u} \|^2_{\mathscr{V}}, 
  \end{split}
\end{equation}
and we therefore take the pivot space $\mathscr{H}_Y$ to be ${H^1(B)^d}$.

\end{remark}

\subsection{Governing equations: incompressible solid}
\label{subsec: Governing equations: incompressible solid}
When the solid is incompressible, the mass density of both the fluid and the solid are constant so that $\dot{\rho} = 0$ (almost) everywhere in $\Omega$.  Cognizant of Remark~\ref{rem: domains of definition}, referring to Eqs.~\eqref{eq: boundary conditions}, Eqs.~\eqref{eq: functional space u}--\eqref{eq: functional space w}, and the constitutive response functions of both the fluid and the solid, Eqs.~\eqref{eq: Bmomentum weak} and~\eqref{eq: Bmass weak} can be written as
\begin{gather}
\label{eq: Bmomentum weak partitioned first}
\begin{multlined}[b]
\int_{\Omega} \rho_{\f}(\dot{\bv{u}} - \bv{b}) \cdot \bv{v} \d{v}
+ 
\int_{B_{t}} (\rho_{\s} - \rho_{\f}) (\dot{\bv{u}} - \bv{b}) \cdot \bv{v} \d{v}
\\
+
\int_{\Omega} \hat{\tensor{T}}_{\f} \cdot \nabla_{\bv{x}}\bv{v} \d{v}
+
\int_{B_{t}} \bigr(\hat{\tensor{T}}_{\s} - \hat{\tensor{T}}_{\f}\bigl)\cdot \nabla_{\bv{x}}\bv{v} \d{v} - \int_{\partial\Omega_{N}} \bv{\tau}_{g} \cdot \bv{v} \d{a} = 0
\quad \forall \bv{v} \in \mathscr{V}_{0}
\end{multlined}
\shortintertext{and}
\label{eq: Bmass weak partitioned}
\int_{\Omega} q \ldiv \bv{u} \d{v} = 0
\quad \forall q \in \mathscr{Q}.
\end{gather}
In addition to the momentum and mass balance laws, we need to enforce Eq.~\eqref{eq: w u rel}.   We do so weakly as follows:
\begin{equation}
\label{eq: w u rel weak}
\Phi_{B}
\int_{B} \Bigl[\dot{\bv{w}}(\bv{s},t) - \bv{u}(\bv{x},t)\big|_{\bv{x} = \map}\Bigr] \cdot \bv{y}(\bv{s}) \d{V} = 0
\quad
\forall \bv{y} \in \mathscr{H}_Y,
\end{equation}
where $\nsd{V}$ denotes the volume of an infinitesimal element of $B$, and where $\Phi_{B}$ is a constant with dimensions of mass over time divided by length cubed, i.e., dimensions such that, in 3D, the volume integral of the quantity $\Phi_{B} \dot{\bv{w}}$ has the same dimensions as a force. 

\begin{remark}[Equation~\eqref{eq: w u rel weak} and comparison with other formulations]
As discussed in the introduction, a key element of any fully variational formulation of immersed methods is (the variational formulation of) the equation enabling the tracking of the motion of the solid.  Equation~\eqref{eq: w u rel weak} is the equation in question. When discretized, it yields a set of ordinary differential equations (\acro{ODE}) relating the degrees of freedom of the extended fluid domain with the degrees of freedom of the immersed domain.  In practical applications, this relation is as general as the choice of the finite-dimensional functional subspaces approximating $\mathscr{V}$ and $\mathscr{Y}$.  Equation~\eqref{eq: w u rel weak} plays a crucial role in ensuring that the proposed finite element formulation is stable.  Equations similar to Eq.~\eqref{eq: w u rel weak} have appeared in other variational formulation of immersed methods.  With this in mind, it is important to remark that the set of \acro{ODE} for tracking the motion of the immersed solid in \cite{BoffiGastaldi_2003_A-Finite_0} was not obtained via a variational formulation.  Rather, it was obtained by setting the value of $\dot{\bv{w}}$ equal to that of $\bv{u}$ at the vertices of the triangulation discretizing the solid domain.  The first fully variational formulation of the equations of motion of the solid domain was presented almost simultaneously by \cite{Heltai_2006_The-Finite_0} and \cite{LiuKim_2007_Mathematical_0}.  However, the present authors could not find in the literature evidence pertaining to the practical implementation of the the work by \cite{LiuKim_2007_Mathematical_0}.  In \cite{BoffiGastaldiHeltaiPeskin-2008-a} the solid and the fluid mass densities are the same and are equal to one (a restriction which was removed recently in~\citealp{BoffiCavalliniGastaldi-2011-a}).  In addition, $\mathscr{Y}$ is chosen as the space of globally continuous piecewise affine functions over triangles in two-dimensions and over tetrahedrons in three dimensions (see Eq.(52) on p.~2218 in \citealp{BoffiGastaldiHeltaiPeskin-2008-a}).  Under these assumptions, Eq.~\eqref{eq: w u rel weak} yields a system of \acro{ODE} of the type $\dot{\bv{w}}_{k}(t) = \bv{u}(\bv{x}_{k},t)$, where $k$ ranges over the index set of the vertices of the triangulation of the solid domain.  That is, for the purpose of tracking the motion of the solid, the formulation by \cite{BoffiGastaldiHeltaiPeskin-2008-a} yields the same equations as those in \cite{BoffiGastaldi_2003_A-Finite_0}.  Finally, again as indicated in the introduction, to the best of the authors' knowledge, the approach to the determination of the motion of the solid expressed via Eq.~\eqref{eq: w u rel weak} has only been explicitly discussed by \cite{Heltai_2006_The-Finite_0}, \cite{LiuKim_2007_Mathematical_0}, and~\cite{BlancoFeijo_2008_A-Variational_0}.  However, no general numerical implementations have been demonstrated.  This particular aspect of the current formulation is one of the thrusts of this paper.  A discussion of how Eq.~\eqref{eq: w u rel weak} is practically implemented in a \acro{FEM} code is presented later.
%Bathe + WK Liu.
\end{remark}

Going back to the discussion of the problem's governing equations, we now anticipate that our proposed numerical approximation of Eqs.~\eqref{eq: Bmomentum weak partitioned first}--\eqref{eq: w u rel weak} is based on the use of two independent triangulations, namely, one of $\Omega$ and one of $B$.  The fields $\bv{u}$ and $p$, as well as their corresponding test functions, will be expressed via finite element spaces supported by the triangulation of $\Omega$.  By contrast, the field $\bv{w}$ will be expressed via a finite element space supported by the triangulation of $B$.  Motivated by this fact, we now reformulate every integral over $B_{t}$ as a corresponding integral over $B$.  Such a reformulation affects only Eq.~\eqref{eq: Bmomentum weak partitioned first}, which can be rewritten as
\begin{multline}
\label{eq: Bmomentum weak partitioned last}
\int_{\Omega} \rho_{\f} (\dot{\bv{u}} - \bv{b}) \cdot \bv{v} \d{v}
- \int_{\Omega} p \ldiv \bv{v} \d{v}
+
\int_{\Omega} \hat{\tensor{T}}^{v}_{\f} \cdot \nabla_{\bv{x}}\bv{v} \d{v}
-\int_{\partial\Omega_{N}} \bv{\tau}_{g} \cdot \bv{v} \d{a}
\\
+ 
\int_{B} \bigl\{[\rho_{\s_{0}}(\bv{s}) - \rho_{\f} J(\bv{s},t)] [\dot{\bv{u}}(\bv{x},t) - \bv{b}(\bv{x},t)]\cdot \bv{v}(\bv{x})\bigr|_{\bv{x} = \map} \d{V}
\\
+
\int_{B} J(\bv{s},t) \bigl(\hat{\tensor{T}}^{v}_{\s} - \hat{\tensor{T}}^{v}_{\f}\bigr) \cdot \nabla_{\bv{x}}\bv{v}(\bv{x})\bigr|_{\bv{x} = \map} \d{V}
\\
+
\int_{B} \hat{\tensor{P}}^{e}_{\s} \, \trans{\tensor{F}}(\bv{s},t) \cdot \nabla_{\bv{x}}\bv{v}(\bv{x})\bigr|_{\bv{x} = \map} \d{V}
=
0
\quad \forall \bv{v} \in \mathscr{V}_{0}.
\end{multline}
The last three terms in Eq.~\eqref{eq: Bmomentum weak partitioned last} have been written so as to explicitly express their evaluation process.  While it is true that, for an incompressible solid $J(\bv{s},t) = 1$ for all $\bv{s} \in B$ and for all $t \in [0,T)$, this occurrence may not be satisfied in an approximate formulation of the problem.  Therefore, we prefer to retain the term $J(\bv{s},t)$ in our formulation to contribute to its stability.

\begin{remark}[Dirac-$\delta$s are not intrinsic to immersed methods]
As eloquently stated by \cite{BoffiGastaldi_2003_A-Finite_0} in their introduction, ``The \acro{IB} method is at the same time a mathematical formulation and a numerical scheme.''  As such, and as argued in Remark~\ref{remark: peskin delta use}, the use of the Dirac-$\delta$ distribution was justified by convenience and a preference for a specific solution method rather by a necessity intrinsic to the physics of the problem.  One of the main thrusts of the works by \cite{BoffiGastaldi_2003_A-Finite_0,Heltai_2006_The-Finite_0,BoffiGastaldiHeltaiPeskin-2008-a,LiuKim_2007_Mathematical_0,BlancoFeijo_2008_A-Variational_0} is precisely that of showing that an immersed method can be formulated without any reference whatsoever to the use of Dirac-$\delta$ distributions.  Again, we wish to point out that one of the objectives of the present work is precisely that of demonstrating an implementation technique that does not rely on the approximation of the Dirac-$\delta$ distribution.  This fact is one of the distinguishing features of our work when compared to other approaches currently in the literature (see, e.g., \citealp{WangZhang_2009_On-Computational_0}).
\end{remark}

We now define various operators that will be used to state our finite element formulation.  These definitions rely on the concept of duality.  To make explicit the declaration of the spaces in duality, we will use the following notation:
\begin{equation}
\label{eq: duality notation}
\prescript{}{V^{*}}{\bigl\langle} \psi, \phi \big\rangle_{V},
\end{equation}
in which, given a vector space $V$ and its dual $V^{*}$, $\psi$ and $\phi$ are elements of the vector spaces $V^{*}$ and $V$, respectively, and where $\prescript{}{V^{*}}{\bigl\langle} \bullet, \bullet \big\rangle_{V}$ identifies the duality product between $V^{*}$ and  $V$.  Also, to be explicit on how certain terms depend on the selected unknown fields, we introduce the following shorthand notation
\begin{align}
\label{eq: fv stress abbreviated}
\hat{\tensor{T}}^{v}_{\f}[\bv{u}] &= \mu_{\f} \bigl[\nabla_{\bv{x}}\bv{u}(\bv{x},t) + \trans{(\nabla_{\bv{x}}\bv{u}(\bv{x},t))} \bigr],
\\
\label{eq: sv stress abbreviated}
\hat{\tensor{T}}^{v}_{\s}[\bv{u}] &= \mu_{\s} \bigl[\nabla_{\bv{x}}\bv{u}(\bv{x},t) + \trans{(\nabla_{\bv{x}}\bv{u}(\bv{x},t))} \bigr],
\\
\label{eq: Fw abbreviated}
\tensor{F}[\bv{w}] &= \tensor{I} + \nabla_{s}\bv{w}(\bv{s},t),
\\
\label{eq: Jw abbreviated}
J[\bv{w}] &= \det \tensor{F}[\bv{w}],
\\
\label{eq: se stress abbreviated}
\hat{\tensor{P}}^{e}_{\s}[\bv{w}] &= \frac{\partial\hat{W}^{e}_{\s}(\tensor{F})}{\partial{F}}\bigg|_{\tensor{F} = \tensor{F}[\bv{w}]}.
\end{align}
Finally, to help identify the domain and range of these operators, we establish the following convention.  We will use the numbers $1$, $2$, and $3$ to identify the spaces $\mathscr{V}$, $\mathscr{Q}$, and $\mathscr{Y}$, respectively.  We will use the Greek letter $\alpha$, $\beta$, and $\gamma$ to identify the spaces $\mathscr{V}^{*}$, $\mathscr{Q}^{*}$, and $\mathscr{Y}^{*}$, respectively.  Then, a Greek letter followed by a number will identify an operator whose domain is the space corresponding to the number, and whose co-domain is in the space corresponding to the Greek letter.  For example, the notations
\begin{equation}
\label{eq: space convention}
\mathcal{E}_{\alpha 2}
\quad \text{and} \quad
\mathcal{E}_{\alpha 2} \, p
\end{equation}
will identify a map ($\mathcal{E}_{\alpha 2}$) from $\mathscr{Q}$ into  $\mathscr{V}^{*}$ and the action of this map ($\mathcal{E}_{\alpha 2}\, p \in \mathscr{V}^{*}$) on the field $p \in \mathscr{Q}$, respectively.  If an operators has only one subscript, that subscript identifies the space containing the range of the operator. For simplicity, the pivot spaces $\mathscr{H}_{V}$ and $\mathscr{H}_{Y}$ and their duals will inherit the same notation as $\mathscr{V}$ and $\mathscr{Y}$. With this in mind, let
\begin{alignat}{3}
\label{eq: MOmega def}
\mathcal{M}_{\svs\sv} &: \mathscr{H}_{V} \to \mathscr{V}^{*},
&\quad
\prescript{}{\mathscr{V}^{*}}{\bigl\langle}
\mathcal{M}_{\svs\sv}\bv{u},\bv{v}
\big\rangle_{\mathscr{V}} 
&:= 
\int_{\Omega} \rho_{\f} \, \bv{u} \cdot \bv{v} \d{v}
&\quad
&\forall \bv{u} \in \mathscr{H}_{V}, \forall \bv{v} \in \mathscr{V}_{0},
\\
\label{eq: NOmega def}
\mathcal{N}_{\svs\sv}(\bv{u})
&:
\mathscr{V} \to \mathscr{V}^{*},
&\quad
\prescript{}{\mathscr{V}^{*}}{\bigl\langle}
\mathcal{N}_{\svs\sv}(\bv{u}) \bv{w} , \bv{v}
\big\rangle_{\mathscr{V}} 
&:=
\int_{\Omega} \rho_{\f} (\nabla_{\bv{x}} \bv{w})\bv{u} \cdot \bv{v} \d{v}
&\quad
&\forall \bv{u},\bv{w} \in \mathscr{V}, \forall \bv{v} \in \mathscr{V}_{0},
\\
\label{eq: AOmega def}
\mathcal{D}_{\svs\sv} &: \mathscr{V} \to \mathscr{V}^{*},
&\quad
\prescript{}{\mathscr{V}^{*}}{\bigl\langle}
\mathcal{D}_{\svs\sv}\bv{u},\bv{v}
\big\rangle_{\mathscr{V}} 
&:= 
\int_{\Omega} \hat{\tensor{T}}^{v}_{\f}[\bv{u}] \cdot \nabla_{\bv{x}}\bv{v} \d{v}
&\quad
&\forall \bv{u} \in \mathscr{V}, \forall\bv{v} \in \mathscr{V}_{0},
\\
\label{eq: BBeta1 def}
\mathcal{B}_{\sqs\sv} &: \mathscr{V} \to \mathscr{Q}^{*},
&\quad
\prescript{}{\mathscr{Q}^{*}}{\bigl\langle}
\mathcal{B}_{\sqs\sv} \bv{u}, q
\big\rangle_{\mathscr{Q}} &:= -\int_{\Omega} q \ldiv \bv{u} \d{v}
&\quad
&\forall q \in \mathscr{Q}, \forall \bv{u} \in \mathscr{V},
\\
\label{eq: BBeta1T def}
\trans{\mathcal{B}_{\sqs\sv}} &: \mathscr{Q} \to \mathscr{V}^{*},
&\quad
\prescript{}{\mathscr{V}^{*}}{\bigl\langle}
\trans{\mathcal{B}}_{\sqs\sv} q, \bv{u}
\big\rangle_{\mathscr{V}} &:= -\int_{\Omega} q \ldiv \bv{u} \d{v}
&\quad
&\forall q \in \mathscr{Q}, \forall \bv{u} \in \mathscr{V}.
\end{alignat}
The operators defined in Eqs.~\eqref{eq: MOmega def}--\eqref{eq: BBeta1T def} concern terms that are typical of the Navier-Stokes equations and will be referred to as the Navier-Stokes component of the problem.  As in other immersed methods, these operators have their support in $\Omega$ as a whole.

We now define those operators in our formulation that have their support over $B$ but do not contain prescribed body forces or boundary terms.
\begin{align}
% Info su B a V
\label{eq: pseudo mass BOmega def}
\begin{split}
&\delta\mathcal{M}_{\svs\sv}(\bv{w}) : \mathscr{H}_{V} \to \mathscr{V}^{*},~\forall \bv{w} \in \mathscr{Y}, \forall\bv{u} \in \mathscr{H}_{V}, \forall\bv{v} \in \mathscr{V}_{0},
\\
&\qquad
\prescript{}{\mathscr{V}^{*}}{\bigl\langle}
\delta\mathcal{M}_{\svs\sv}(\bv{w}) \bv{u}, \bv{v}
\big\rangle_{\mathscr{V}} := \int_{B} \bigl\{\bigl(\rho_{\s_{0}}(\bv{s}) - \rho_{\f}J[\bv{w}] \bigr) \bv{u}(\bv{x}) \cdot \bv{v}(\bv{x})\bigl\}_{\bv{x}=\map[w]} \d{V},
\end{split}
\\
% Trilinear su B a V
\label{eq: pseudo trilinear BOmega def}
\begin{split}
&\delta\mathcal{N}_{\svs\sv}(\bv{w},\bv{\ell},\bv{z}) : \mathscr{V} \to \mathscr{V}^{*},~\forall \bv{w}, \bv{\ell} \in \mathscr{Y}, \forall \bv{u},\bv{z} \in \mathscr{V}, \forall\bv{v}\in\mathscr{V}_{0},
\\
&\qquad
\begin{aligned}
\prescript{}{\mathscr{V}^{*}}{\bigl\langle}
\delta\mathcal{N}_{\svs\sv}(\bv{w},\bv{\ell},\bv{z})\bv{u}, \bv{v}
\big\rangle_{\mathscr{V}} &:= \int_{B} 
\bigl\{\bigl[
(\rho_{\s_{0}}(\bv{s}) \nabla_{\bv{x}}\bv{u}(\bv{x})\bv{\ell}(\bv{s})
\\
&\qquad\quad\quad -\rho_{\f} J[\bv{w}] \nabla_{\bv{x}}\bv{u}(\bv{x})\bv{z}(\bv{x})\bigr] \cdot \bv{v}(\bv{x})\bigr\}_{\bv{x}=\map[w]} \d{V}.
\end{aligned}
\end{split}
\\
% pseudo stiffness BOmega
\label{eq: A BOmega def}
\begin{split}
&\delta\mathcal{D}_{\svs\sv}(\bv{w}) : \mathscr{V} \to \mathscr{V}^{*},~\forall \bv{w} \in \mathscr{Y}, \forall \bv{u}\in\mathscr{V}, \forall\bv{v} \in \mathscr{V}_{0},
\\
&\qquad
\begin{aligned}[b]
\prescript{}{\mathscr{V}^{*}}{\bigl\langle}
\delta\mathcal{D}_{\svs\sv}(\bv{w}) \bv{u}, \bv{v}
\big\rangle_{\mathscr{V}} &:= 
\int_{B}
\Bigl[ J[\bv{w}]
\bigl(
\hat{\tensor{T}}^{v}_{\s}[\bv{u}] - \hat{\tensor{T}}^{v}_{\f}[\bv{u}] \bigr)
\cdot \nabla_{\bv{x}} \bv{v}(\bv{x})\Bigr]_{\bv{x}=\map[w]} \d{V},
\end{aligned}
\end{split}
\\
% pseudo stiffness BOmega
\label{eq: pseudo stiffness BOmega def}
\begin{split}
&\mathcal{A}_{\svs}(\bv{w},\bv{h}) \in \mathscr{V}^{*},~\forall \bv{w}, \bv{h} \in \mathscr{Y},
%\forall \bv{u}\in\mathscr{V},
\forall\bv{v} \in \mathscr{V}_{0}
\\
&\qquad
\begin{aligned}[b]
\prescript{}{\mathscr{V}^{*}}{\bigl\langle}
\mathcal{A}_{\svs}(\bv{w},\bv{h}), \bv{v}
\big\rangle_{\mathscr{V}} &:= 
\int_{B}
\bigl[
\hat{\tensor{P}}^{e}_{\s}[\bv{w}] \trans{\tensor{F}}[\bv{h}]
\cdot \nabla_{\bv{x}} \bv{v}(\bv{x})\bigr]_{\bv{x}=\map[h]} \d{V}.
\end{aligned}
\end{split}
\end{align}

We now define operators with support in $B$ that express the coupling of the velocity fields defined over $\Omega$ and over $B$.  Specifically, we have
\begin{align}
\label{eq: MB def}
\begin{split}
&\mathcal{M}_{\sys\sy} : \mathscr{H}_{Y} \to \mathscr{H}_{Y}^{*},~\forall \bv{w},\bv{y} \in \mathscr{H}_{Y},
\\
&\qquad
\prescript{}{\mathscr{H}_{Y}^{*}}{\bigl\langle}
\mathcal{M}_{\sys\sy}\bv{w}, \bv{y}
\big\rangle_{\mathscr{H}_Y} := \Phi_{B} \int_{B} \bv{w} \cdot \bv{y}(\bv{s}) \d{V},
\end{split}
\\
% Mixed Mass matrix on B
\label{eq: MGamma def}
\begin{split}
&\mathcal{M}_{\sys\sv}(\bv{w}) : \mathscr{V} \to \mathscr{H}_{Y}^{*},~\forall \bv{u} \in \mathscr{V}, \forall \bv{w} \in \mathscr{Y}, \forall \bv{y} \in \mathscr{H}_{Y},
\\
&\qquad
\prescript{}{\mathscr{H}_{Y}^{*}}{\bigl\langle}
\mathcal{M}_{\sys\sv}(\bv{w}) \bv{u}, \bv{y}
\big\rangle_{\mathscr{H}_{Y}} := \Phi_{B} \int_{B} \bv{u}(\bv{x})\big|_{\bv{x} = \map[w]} \cdot \bv{y}(\bv{s}) \d{V}
\end{split}
\\
\label{eq: Mgamma1T def}
\begin{split}
&\trans{\mathcal{M}}_{\sys\sv}(\bv{w}) : \mathscr{H}_{Y} \to \mathscr{V}^{*},~\forall \bv{u} \in \mathscr{V}, \forall \bv{w} \in \mathscr{Y}, \forall \bv{y} \in  \mathscr{H}_{Y} 
\\
&\qquad
\prescript{}{\mathscr{V}^{*}}{\bigl\langle}
\trans{\mathcal{M}}_{\sys\sv}(\bv{w}) \bv{y}, \bv{u}
\big\rangle_{\mathscr{V}} :=  \Phi_{B} \int_{B} \bv{u}(\bv{x})\big|_{\bv{x} = \map[w]} \cdot \bv{y}(\bv{s}) \d{V}
\end{split}
\end{align}

Finally, we define the operators that express the action of prescribed body and surface forces.
\begin{align}
% Forcing term Omega
\label{eq: Forcing Omega def}
\begin{split}
&\mathcal{F}_{\svs} \in \mathscr{V}^{*},~\forall \bv{b} \in H^{-1}(\Omega), \forall \bv{\tau}_{g} \in H^{-\frac 1 2}(\partial \Omega_N), \forall \bv{v} \in \mathscr{V}_{0}
\\
&\qquad\prescript{}{\mathscr{V}^{*}}{\bigl\langle}
\mathcal{F}_{\svs}, \bv{v}
\big\rangle_{\mathscr{V}} :=
\int_{\Omega} \rho_{\f} \, \bv{b} \cdot \bv{v} \d{v} + \int_{\partial\Omega_{N}} \bv{\tau}_g \cdot \bv{v} \d{a}
\end{split}
\\
% Forcing term B
\label{eq: Forcing B def}
\begin{split}
&\mathcal{G}_{\svs}(\bv{w}) \in \mathscr{V}^{*},~\forall \bv{w} \in \mathscr{Y}, \forall \bv{b} \in H^{-1}(\Omega), \forall \bv{v} \in \mathscr{V}_{0}
\\
&\qquad\prescript{}{\mathscr{V}^{*}}{\bigl\langle}
\mathcal{G}_{\svs}(\bv{w}), \bv{v}
\big\rangle_{\mathscr{V}} :=
\int_{B} \bigl(\rho_{\s_{0}}(\bv{s}) - \rho_{\f} J[\bv{w}] \bigr) \bv{b} \cdot \bv{v}(\bv{x})\bigr|_{\bv{x} = \map[w]} \d{v}.
\end{split}
\end{align}

\begin{remark}[Dependence on the motion of the solid]
In defining the operators in Eqs.~\eqref{eq: MOmega def}--\eqref{eq: Forcing B def}, we have used a notation meant to point out explicitly the role played by the field $\bv{w}$ in the evaluation of integrals over $B$.  For the operator $\mathcal{A}_{\svs}$ in Eq.~\eqref{eq: pseudo stiffness BOmega def}, the motion of the solid plays a double role, one pertaining to the elastic response of the solid (through $\bv{w}$) and the other pertaining to the map (through $\bv{h}$) functioning as a change of variables of integration.
\end{remark}

It is convenient to explicitly separate the double role of the displacement $\bv{w}$ in the elastic operator $\mathcal{A}_{\svs}(\bv{w},\bv{w})$, by reformulating it in terms of a \emph{change of variable} operator and in terms of a purely Lagrangian elastic operator:
\begin{align}
\label{eq: S def}
\begin{split}
&\mathcal{S}_{\svs\sys}(\bv{h}): \mathscr{H}_{Y}^{*} \to \mathscr{V}^{*},~\forall \bv{y}^{*} \in \mathscr{H}_{Y}^{*}, \forall \bv{h} \in \mathscr{Y}, \forall \bv{v} \in \mathscr{V}_{0}
\\
&\qquad
\prescript{}{\mathscr{V}^{*}}{\bigl\langle}
\mathcal{S}_{\svs\sys}(\bv{h}) \bv{y}^{*},\bv{v}
\big\rangle_{\mathscr{V}} 
:= 
\prescript{}{\mathscr{H}_{Y}^{*}}{\bigl\langle}
\bv{y}^{*},\bv{v}(\bv{x})\big|_{\bv{x} = \bv{s} + \bv{h}(\bv{s})}
\big\rangle_{\mathscr{H}_{Y}},
\end{split}
\\
\label{eq: Agamma def}
\begin{split}
&\mathcal{A}_{\sys}(\bv{w}) \in \mathscr{H}_{Y}^{*},~\forall \bv{w}\in \mathscr{Y}, \forall \bv{y} \in \mathscr{H}_{Y}
\\
&\qquad
\prescript{}{\mathscr{H}_{Y}^{*}}{\bigl\langle}
\mathcal{A}_{\sys}(\bv{w}), \bv{y}
\big\rangle_{\mathscr{H}_{Y}} := \int_{B}
\hat{\tensor{P}}_{\s}^{e}[\bv{w}] \cdot \nabla_{\bv{s}}\bv{y} \d{V}.
\end{split}
\end{align}

The operator $\mathcal{S}_{\svs\sys}(\bv{h})$ is the map that allows us to express the duality over $\mathscr{H}_{Y}$ in terms of that over $\mathscr{V}$ through the deformation $\bv{h}$.  As such, $\mathcal{S}_{\svs\sys}(\bv{h})$ puts into communication the Lagrangian and Eulerian descriptions of the motion of the immersed domain. The operator in Eq.~\eqref{eq: Agamma def} is a typical component of classical \acro{FEM} approaches to elasticity and is the (fully Lagrangian form of the) stiffness operator of the immersed solid.  

One of the crucial components of any solutions method for \acro{FSI} problems is the communication between the Lagrangian and Eulerian descriptions of the physics of the solid domain.  In this context, the operator $\mathcal{A}_{\svs}(\bv{w},\bv{h})$, defined in Eq.~\eqref{eq: pseudo stiffness BOmega def}, can be said to be the Eulerian counterpart of the operator $\mathcal{A}_{\sys}(\bv{w})$ as is shown by the the following result.
\begin{theorem}[Eulerian and Lagrangian elastic stiffness operators of the immersed domain]
\label{th: eulerian vs lagrangian stiffness}
With reference to the definitions in Eqs.~\eqref{eq: pseudo stiffness BOmega def}, \eqref{eq: S def}, and~\eqref{eq: Agamma def}, we have
\begin{equation}
\label{eq: EulerianLagrangianElasticity}
\mathcal{A}_{\svs}(\bv{w},\bv{h}) = \mathcal{S}_{\svs\sys}(\bv{h}) \mathcal{A}_{\sys}(\bv{w})
\quad \text{and} \quad 
\mathcal{S}_{\svs\sys}(\bv{h}) = \trans{\mathcal{M}}_{\sys\sv}(\bv{h}) \mathcal{M}_{\sys\sy}^{-1},
\end{equation}
where $\mathcal{S}_{\svs\sys}(\bv{h}) \mathcal{A}_{\sys}(\bv{w})$ and $\trans{\mathcal{M}}_{\sys\sv}(\bv{h}) \mathcal{M}_{\sys\sy}^{-1}$ indicate the composition of the operators $\mathcal{S}_{\svs\sys}(\bv{h})$ and $\mathcal{A}_{\sys}(\bv{w})$ and of the operators $\trans{\mathcal{M}}_{\sys\sv}(\bv{h})$ and $\mathcal{M}_{\sys\sy}^{-1}$, respectively.
\end{theorem}
\begin{proof}
By the definitions in Eqs.~\eqref{eq: S def} and~\eqref{eq: Agamma def}, $\forall \bv{w},\bv{h} \in \mathscr{Y}$ and $\forall \bv{v} \in \mathscr{V}_{0}$, we have
\begin{equation}
\prescript{}{\mathscr{V}^{*}}{\bigl\langle}
\mathcal{S}_{\svs\sys}(\bv{h}) \mathcal{A}_{\sys}(\bv{w}), \bv{v}
\big\rangle_{\mathscr{V}}  = \prescript{}{\mathscr{H}_{Y}^{*}}{\bigl\langle}
\mathcal{A}_{\sys}(\bv{w}), \bv{v}(\bv{x})\big|_{\bv{x} = \bv{s} + \bv{h}(\bv{s})}
\big\rangle_{\mathscr{H}_{Y}},
\end{equation}
which, using again the definition in Eq.~\eqref{eq: Agamma def}, gives
\begin{equation}
\label{eq: SAGamma}
\begin{aligned}[b]
\prescript{}{\mathscr{V}^{*}}{\bigl\langle}
\mathcal{S}_{\svs\sys}(\bv{h}) \mathcal{A}_{\sys}(\bv{w}), \bv{v}
\big\rangle_{\mathscr{V}} &= \int_{B}
\hat{\tensor{P}}_{\s}^{e}[\bv{w}] \cdot \nabla_{\bv{s}}\bv{v}(\bv{x})\big|_{\bv{x} = \map[h]} \d{V}
\\
&= \int_{B}
\hat{\tensor{P}}_{\s}^{e}[\bv{w}] \cdot \nabla_{\bv{x}}\bv{v}(\bv{x})\big|_{\bv{x} = \map[h]} \tensor{F}[\bv{h}] \d{V}
\\
&= \int_{B}
\hat{\tensor{P}}_{\s}^{e}[\bv{w}] \trans{\tensor{F}}[\bv{h}] \cdot \nabla_{\bv{x}}\bv{v}(\bv{x})\big|_{\bv{x} = \map[h]} \d{V},
\end{aligned}
\end{equation}
where the second line of the above equation was obtained by a standard application of the chain rule.  Comparing the result in Eq.~\eqref{eq: SAGamma} with the definition in Eq.~\eqref{eq: pseudo stiffness BOmega def}, the first of Eqs.~\eqref{eq: EulerianLagrangianElasticity} follows.  Next, again applying the definition in Eq.~\eqref{eq: S def}, $\forall \bv{w},\bv{h} \in \mathscr{Y}$ and $\forall \bv{v} \in \mathscr{V}_{0}$, we have
\begin{equation}
\prescript{}{\mathscr{V}^{*}}{\bigl\langle}
\mathcal{S}_{\svs\sys}(\bv{h}) \mathcal{M}_{\sys\sy}\bv{w}, \bv{v}
\big\rangle_{\mathscr{V}}  = \prescript{}{\mathscr{H}_{Y}^{*}}{\bigl\langle}
\mathcal{M}_{\sys\sy}\bv{w}, \bv{v}(\bv{x})\big|_{\bv{x} = \bv{s} + \bv{h}(\bv{s})}
\big\rangle_{\mathscr{H}_{Y}},
\end{equation}
which, by the definitions in Eq.~\eqref{eq: MB def} and Eq.~\eqref{eq: Mgamma1T def}, gives
\begin{equation}
\begin{aligned}[b]
\prescript{}{\mathscr{V}^{*}}{\bigl\langle}
\mathcal{S}_{\svs\sys}(\bv{h}) \mathcal{M}_{\sys\sy}\bv{w}, \bv{v}
\big\rangle_{\mathscr{V}} &= \Phi_{B} \int_{B} \bv{w} \cdot \bv{v}(\bv{x})\big|_{\bv{x} = \bv{s} + \bv{h}(\bv{s})} \d{V}
\\
&= 
\prescript{}{\mathscr{V}^{*}}{\bigl\langle}
\trans{\mathcal{M}}_{\sys\sv}(\bv{h}) \bv{w}, \bv{v}
\big\rangle_{\mathscr{V}},
\end{aligned}
\end{equation}
from which we deduce that 
\begin{equation}
\mathcal{S}_{\svs\sys}(\bv{h}) \mathcal{M}_{\sys\sy} = \trans{\mathcal{M}}_{\sys\sv}(\bv{h}).
\end{equation}
Since the operator $\mathcal{M}_{\sys\sy}$ is the Riesz identity between $\mathscr{H}_Y$ and $\mathscr{H}^*_Y$,  it is invertible and the second of Eqs.~\eqref{eq: EulerianLagrangianElasticity} follows.

\end{proof}
The operators defined above allow us to formally restate the overall problem described by Eqs.~\eqref{eq: Bmomentum weak partitioned last}, \eqref{eq: Bmass weak partitioned}, and~\eqref{eq: w u rel weak} as follows:
\begin{problem}[Incompressible fluid, incompressible solid: dual formulation]
\label{prob: IFIS}
Given initial conditions $\bv{u}_{0} \in \mathscr{V}$ and $\bv{w}_{0} \in \mathscr{Y}$, for all $t \in (0,T)$ find $\bv{u}(\bv{x},t) \in \mathscr{V}$, $p(\bv{x},t) \in \mathscr{Q}$, and $\bv{w}(\bv{s},t) \in \mathscr{Y}$ such that
\begin{align}
\label{eq: BLM Formal dual}
% Navier-Stokes part of the problem
\begin{aligned}[b]
&\mathcal{M}_{\svs\sv}\bv{u}' +  \mathcal{N}_{\svs\sv}(\bv{u})\bv{u} + 
\mathcal{D}_{\svs\sv}\bv{u} + \trans{(\mathcal{B}_{\sqs\sv})}p
\\
% Coupling part
&\quad+ \delta\mathcal{M}_{\svs\sv}(\bv{w})\bv{u}'
+ \delta\mathcal{N}_{\svs\sv}(\bv{w},\bv{w}',\bv{u})\bv{u}
+ \delta\mathcal{D}_{\svs\sv}(\bv{w})\bv{u}
+ \mathcal{S}_{\svs\sys}(\bv{w})\mathcal{A}_{\sys}(\bv{w})
\end{aligned}
&= \mathcal{F}_{\svs} + \mathcal{G}_{\svs}(\bv{w}),
\\
\label{eq: incompressibility Formal dual}
\mathcal{B}_{\sqs\sv}\bv{u} &= 0,
\\
\label{eq: velocity coupling dual}
\mathcal{M}_{\sys\sy}\bv{w}' - \mathcal{M}_{\sys\sv}(\bv{w})\bv{u} &= \bv{0}.
\end{align}
\end{problem}
\begin{remark}[Eulerian vs.\ Lagrangian elastic operators]
\label{rem:eulerian vs lagrangian elasticity}
Referring to Eq.~\eqref{eq: BLM Formal dual}, Theorem~\ref{th: eulerian vs lagrangian stiffness} shows that we could have formulated Problem~\ref{prob: IFIS} using the Eulerian elastic operator $\mathcal{A}_{\svs}(\bv{w},\bv{w})$ instead of the composition $\mathcal{S}_{\svs\sys}(\bv{w})\mathcal{A}_{\sys}(\bv{w})$.  This is because, in the infinite dimensional context of our abstract variational formulation, the operators $\mathcal{A}_{\svs}(\bv{w},\bv{w})$ and $\mathcal{S}_{\svs\sys}(\bv{w})\mathcal{A}_{\sys}(\bv{w})$ are equivalent.  However, as will be shown in Section~\ref{sec: Stability of the continuum problem}, the use of $\mathcal{S}_{\svs\sys}(\bv{w})\mathcal{A}_{\sys}(\bv{w})$ is justified by the fact that this operator lends itself more naturally to the derivation of stability estimates that rely solely on the \emph{weak form} of the velocity coupling in Eq.~\eqref{eq: velocity coupling dual}.  Moreover, anticipating a result discussed in Section~\ref{sec:semi discrete stability}, it turns out that (\emph{i}) the equivalence between $\mathcal{A}_{\svs}(\bv{w},\bv{w})$ and $\mathcal{S}_{\svs\sys}(\bv{w})\mathcal{A}_{\sys}(\bv{w})$ fails to hold for the discrete version of these operators, and (\emph{ii}) only the discrete version of $\mathcal{S}_{\svs\sys}(\bv{w})\mathcal{A}_{\sys}(\bv{w})$ can be shown to yield a satisfactory semi-discrete stability estimate.
\end{remark}

\subsection{Governing equations: compressible solid}
\label{subsec: Governing equations: compressible solid}
When the solid is compressible, incompressibility must be restricted to the physical (as opposed to the extended) fluid domain.  In addition, since the stress response in the solid is completely determined by the solid's stress constitutive response function, the field $p$ contributes to the balance of momentum equation only over the domain $\Omega\setminus B_{t}$.  Therefore, for the balance of linear momentum, we write
\begin{multline}
\label{eq: Bmomentum weak partitioned last compressible solid}
\int_{\Omega} \rho_{\f} (\dot{\bv{u}} - \bv{b}) \cdot \bv{v} \d{v}
- \int_{\Omega} p \ldiv \bv{v} \d{v}
+
\int_{\Omega} \hat{\tensor{T}}^{v}_{\f} \cdot \nabla_{\bv{x}}\bv{v} \d{v}
-\int_{\partial\Omega_{N}} \bv{\tau}_{g} \cdot \bv{v} \d{a}
\\
+ 
\int_{B} \bigl\{[\rho_{\s_{0}}(\bv{s}) - \rho_{\f_0}] [\dot{\bv{u}}(\bv{x},t) - \bv{b}(\bv{x},t)]\cdot \bv{v}(\bv{x})\bigr|_{\bv{x} = \map} \d{V}
\\
+ \int_{B} J(\bv{s},t) p(\bv{x},t) \ldiv \bv{v}(\bv{x})\bigr|_{\bv{x} = \map} \d{V}
\\
+
\int_{B} J(\bv{s},t) \bigl(\hat{\tensor{T}}^{v}_{\s} - \hat{\tensor{T}}^{v}_{\f}\bigr) \cdot \nabla_{\bv{x}}\bv{v}(\bv{x})\bigr|_{\bv{x} = \map} \d{V}
\\
+
\int_{B} \hat{\tensor{P}}^{e}_{\s} \, \trans{\tensor{F}}(\bv{s},t) \cdot \nabla_{\bv{x}}\bv{v}(\bv{x})\bigr|_{\bv{x} = \map} \d{V}
=
0
\quad \forall \bv{v} \in \mathscr{V}_{0}.
\end{multline}
Equation~\eqref{eq: Bmomentum weak partitioned last compressible solid}   is identical to Eq.~\eqref{eq: Bmomentum weak partitioned last} except for the term appearing as the third line of Eq.~\eqref{eq: Bmomentum weak partitioned last compressible solid}.  This term can be viewed as a correction to the second term on the first line that restricts the contribution of the field $p$ to $\Omega\setminus B_{t}$.

The restriction of the balance of mass equation to the domain $\Omega\setminus B_{t}$ can be written as follows:
\begin{equation}
\label{eq: balance of mass restricted}
\int_{\Omega} q \ldiv \bv{u} \d{v} - \int_{B_{t}} q \ldiv \bv{u} \d{v} = 0.
\end{equation}
To determine the motion of the solid domain, we adopt the same equation presented in the case of incompressible solids:
\begin{equation}
\label{eq: w u rel weak compressible solid}
\Phi_{B} \int_{B} \Bigl[\dot{\bv{w}}(\bv{s},t) - \bv{u}(\bv{x},t)\big|_{\bv{x} = \map}\Bigr] \cdot \bv{y}(\bv{s}) \d{V} = 0
\quad
\forall \bv{y} \in \mathscr{H}_{Y}.
\end{equation}

Equations~\eqref{eq: Bmomentum weak partitioned last compressible solid}--\eqref{eq: w u rel weak compressible solid} would allow us to determine a unique solution if the field $p$ were restricted to the domain $\Omega\setminus B_{t}$.  However, our numerical scheme still requires that $p$ be defined everywhere in $\Omega$.  To formulate a problem admitting a unique solution for the field $p \in \mathscr{Q}$, we must sufficiently constraint the behavior of $p$ over $B_{t}$.  The strategy to enforce such a constraint is not unique.  In some sense, $p$ can be restricted to $\Omega\setminus B_{t}$ by requiring that $p = 0$ over $B_{t}$.  Another, and perhaps more physically motivated, approach is to observe that, for a Newtonian fluid, $p$ represents the mean normal stress in the fluid.  Therefore, one may choose to constraint the field $p$ in a such a way that it represents the mean normal stress everywhere in $\Omega$.  Since the solid is compressible, its mean normal stress is completely determined by the solid's stress constitutive response functions.  Specifically, letting $\hat{p}_{\s}[\bv{u},\bv{w}]$ denote the constitutive response function for the mean normal stress in the solid, we have
\begin{equation}
\label{eq: Mean normal stress def}
\hat{p}_{\s}[\bv{u},\bv{w}] = -\frac{1}{\trace{\tensor{I}}}
\Bigr[ \hat{\tensor{T}}_{\s}^{v}[\bv{u}] \cdot \tensor{I} + J^{-1}[\bv{w}] \hat{\tensor{P}}_{\s}^{e}[\bv{w}] \cdot \tensor{F}[\bv{w}]\Bigl].
\end{equation}
Therefore, in addition to enforcing Eq.~\eqref{eq: balance of mass restricted}, we can enforce the requirement that $p - \hat{p}_{\s}[\bv{u},\bv{w}] = 0$ over $B_{t}$.  With this in mind, we replace Eq.~\eqref{eq: balance of mass restricted} with the following equation:
\begin{multline}
\label{eq: Bmass weak partitioned compressible solid}
-\int_{\Omega} q \ldiv \bv{u} \d{v} 
+ \int_{B} J(\bv{s},t) q(\bv{x}) \ldiv \bv{u}(\bv{x},t)\bigr|_{\bv{x} = \map} \d{V} 
\\
+ \int_{B} c_{1}
J(\bv{s},t) \bigr[ p(\bv{x},t) - c_{2} \hat{p}_{\s}[\bv{u},\bv{w}] \bigl]
q(\bv{x})\bigr|_{\bv{x} = \map} \d{V} = 0
\quad \forall q \in \mathscr{Q},
\end{multline}
where $c_{1} > 0$ is a constant parameter with dimensions of length times mass over time, and where $c_{2}$ is a dimensionless constant that can take on only the values $0$ or $1$.  For $c_{2} = 0$, the last term on the left-hand side of Eq.~\eqref{eq: Bmass weak partitioned compressible solid} is a (weak) requirement that $p = 0$ over $B_{t}$, whereas for $c_{2} = 1$, the field $p$ is (weakly) constrained to be equal to the mean normal stress in the solid domain and therefore everywhere in $\Omega$.

\begin{remark}[True incompressibility vs.\ near incompressibility]
When a compressible solid is immersed in an incompressible fluid, the classes of motions for the solid and the fluid, respectively, are not necessarily the same.  In this case, problems are typically formulated in such a way that the incompatibility is removed by assuming that both the fluid and the solid are compressible and then tuning the constitutive parameters of the fluid to approximate a nearly incompressible behavior (see, e.g., \citealp{BlancoFeijo_2008_A-Variational_0,WangZhang_2009_On-Computational_0}).  From elasticity it is known that nearly-incompressible models with the same incompressible limit behavior may behave differently from one another (see, e.g., \citealp{LevinsonBurgess_1971_Comparison_0}).  For this reasons, the present authors feel that it is important to offer a numerical approach to the solution of the problem of a compressible solid in an incompressible fluid as its own individual case.
\end{remark}

As was done in the incompressible case, we now reformulate our equations in terms of operators defined via duality.  Most of the operators defined in the incompressible case appear in the formulation of the compressible solid case.  Hence, we now define only those operators that did not appear in the previous case.  Specifically, we define the following three operators:
\begin{align}
\label{eq: Cpdiff def}
\begin{split}
&\delta\mathcal{B}_{\sqs\sv}(\bv{w}) : \mathscr{V} \to \mathscr{Q}^{*},~\forall \bv{w},\in \mathscr{Y},\forall \bv{u} \in \mathscr{V},\forall q \in \mathscr{Q}
\\
&\qquad
\prescript{}{\mathscr{Q}^{*}}{\bigl\langle}
\delta\mathcal{B}_{\sqs\sv}(\bv{w}) \bv{u}, q
\big\rangle_{\mathscr{Q}} := \int_{B} J[\bv{w}] q(\bv{x}) \ldiv\bv{u}(\bv{x})\bigr|_{\bv{x} = \map[w]} \d{V},
\end{split}
\\
\label{eq: Cpdiff transpose def}
\begin{split}
&\trans{\delta\mathcal{B}}_{\sqs\sv}(\bv{w}) : \mathscr{Q} \to \mathscr{V}^{*},~\forall \bv{w},\in \mathscr{Y}, \forall p \in \mathscr{Q}, \forall \bv{v} \in \mathscr{V}_{0}
\\
&\qquad
\prescript{}{\mathscr{V}^{*}}{\bigl\langle}
\trans{\delta\mathcal{B}}_{\sqs\sv}(\bv{w}) p, \bv{v}
\big\rangle_{\mathscr{V}} := \int_{B} J[\bv{w}] p(\bv{x}) \ldiv\bv{v}(\bv{x})\bigr|_{\bv{x} = \map[w]} \d{V},
\end{split}
\\
\label{eq: CPdet def}
\begin{split}
&\delta\mathcal{P}_{\sqs\sq}(\bv{w}) : \mathscr{Q} \to \mathscr{Q}^{*},~\forall p,q \in \mathscr{Q}, \forall \bv{w} \in \mathscr{Y}
\\
&\qquad
\prescript{}{\mathscr{Q}^{*}}{\bigl\langle}
\delta\mathcal{P}_{\sqs\sq}(\bv{w}) p, q
\big\rangle_{\mathscr{Q}} := \int_{B} J[\bv{w}] p(\bv{x}) q(\bv{x}) \big|_{\bv{x} = \map[w]} \d{V},
\end{split}
\\
\label{eq: CPextradet def}
\begin{split}
&\delta\mathcal{E}_{\sqs}(\bv{u},\bv{w},\bv{h}) \in \mathscr{Q}^{*},~\forall \bv{u} \in \mathscr{V}, \forall \bv{w},\bv{h} \in \mathscr{Y}
\\
&\qquad
\prescript{}{\mathscr{Q}^{*}}{\bigl\langle}
\delta\mathcal{E}_{\sqs}(\bv{u},\bv{w},\bv{h}), q
\big\rangle_{\mathscr{Q}} := -\int_{B}
\frac{1}{\trace{\tensor{I}}}
\Bigr[J[\bv{h}] \hat{\tensor{T}}_{\s}^{v}[\bv{u}] \cdot \tensor{I} + \hat{\tensor{P}}_{\s}^{e}[\bv{w}] \cdot \tensor{F}[\bv{h}]\Bigl]
q(\bv{x}) \big|_{\bv{x} = \map[h]} \d{V}.
\end{split}
\end{align}

These operators, along with those defined earlier, allow us to formally restate the overall problem described by Eqs.~\eqref{eq: Bmomentum weak partitioned last compressible solid}, \eqref{eq: Bmass weak partitioned compressible solid}, and~\eqref{eq: w u rel weak compressible solid} as follows:
\begin{problem}[Incompressible fluid, compressible solid: dual formulation]
\label{prob: IFCS}
Given constant coefficients $c_{1} > 0$ and $c_{2} = 0 \lor 1$, and given initial conditions $\bv{u}_{0} \in \mathscr{V}$ and $\bv{w}_{0} \in \mathscr{Y}$, for all $t \in (0,T)$ find $\bv{u}(\bv{x},t) \in \mathscr{V}$, $p(\bv{x},t) \in \mathscr{Q}$, and $\bv{w}(\bv{s},t) \in \mathscr{Y}$ such that
\begin{align}
\label{eq: BLM Formal dual compressible}
% Navier-Stokes part of the problem
\begin{aligned}[b]
&\mathcal{M}_{\svs\sv}\bv{u}' +  \mathcal{N}_{\svs\sv}(\bv{u})\bv{u} + 
\mathcal{D}_{\svs\sv}\bv{u} + \bigl[\trans{\mathcal{B}}_{\sqs\sv} + \trans{\delta\mathcal{B}}_{\sqs\sv}(\bv{w})\bigr] p
\\
% Coupling part
&\quad+ \delta\mathcal{M}_{\svs\sv}(\bv{w})\bv{u}'
+ \delta\mathcal{N}_{\svs\sv}(\bv{w},\bv{w}',\bv{u})\bv{u}
+ \delta\mathcal{D}_{\svs\sv}(\bv{w})\bv{u}
+ \mathcal{S}_{\svs\sys}(\bv{w})\mathcal{A}_{\sys}(\bv{w})
\end{aligned}
&= \mathcal{F}_{\svs} + \mathcal{G}_{\svs}(\bv{w}),
\\
\label{eq: compressibility Formal dual compressible}
\bigl[ \mathcal{B}_{\sqs\sv} + \delta\mathcal{B}_{\sqs\sv}(\bv{w}) \bigr]\bv{u} + c_{1} \bigl[\delta\mathcal{P}_{\sqs\sq}(\bv{w}) p - c_{2} \delta\mathcal{E}_{\sqs}(\bv{u},\bv{w},\bv{w}) \bigr] &= 0,
\\
\label{eq: velocity coupling dual compressible}
\mathcal{M}_{\sys\sy}\bv{w}' - \mathcal{M}_{\sys\sv}(\bv{w})\bv{u} &= \bv{0}.
\end{align}
\end{problem}

\begin{remark}[Complementarity of operators in $\mathscr{Q}^{*}$]
%$\mathscr{H}_{Y}^*$]
In Eq.~\eqref{eq: compressibility Formal dual compressible}, the supports of the terms $\bigl[\mathcal{B}_{\sqs\sv} + \delta\mathcal{B}_{\sqs\sv}(\bv{w})\bigr]$ and $\bigl[\delta\mathcal{P}_{\sqs\sq}(\bv{w}) + c_{2} \delta\mathcal{E}_{\sqs}(\bv{u},\bv{w},\bv{w})\bigr]$ are  $\Omega\setminus B_{t}$ and $B_{t}$, respectively.  That is, the supports of the terms in question are complementary subsets of $\Omega$. Consequently, the terms $\bigl[\mathcal{B}_{\sqs\sv} + \delta\mathcal{B}_{\sqs\sv}(\bv{w})\bigr]$ and $\bigl[\delta\mathcal{P}_{\sqs\sq}(\bv{w}) + c_{2} \delta\mathcal{E}_{\sqs}(\bv{u},\bv{w},\bv{w})\bigr]$ are equal to zero individually: 
\begin{equation} 
\label{eq: zero divergence in fluid compressible and condition on pressure in solid compressible}
\bigl[ \mathcal{B}_{\sqs\sv} + \delta\mathcal{B}_{\sqs\sv}(\bv{w}) \bigr]\bv{u} = 0
\quad \text{and} \quad
c_{1}
\bigl[\delta\mathcal{P}_{\sqs\sq}(\bv{w}) p - c_{2}
\delta\mathcal{E}_{\sqs}(\bv{u},\bv{w},\bv{w}) \bigr] = 0.
\end{equation}
This also implies that the constant $c_{1}$ in Eqs.~\eqref{eq: BLM Formal dual compressible} and the second of Eqs.~\eqref{eq: zero divergence in fluid compressible and condition on pressure in solid compressible} should not be interpreted as a penalization parameter but as a way to ensure that the equations are dimensionally correct.
\end{remark}

Problems~\ref{prob: IFIS} and~\ref{prob: IFCS} can be formally presented in terms of the Hilbert space $\mathscr{Z} := \mathscr{V}\times \mathscr{Q}\times \mathscr{Y}$, and $\mathscr{Z}_{0} := \mathscr{V}_{0} \times \mathscr{Q} \times \mathscr{H}_{Y}$ with inner product given by the sum of the inner products of the generating spaces. Defining $\mathscr{Z} \ni \xi := \trans{[\bv{u}, p, \bv{w}]}$ and $\mathscr{Z}_{0} \ni \psi := \trans{[\bv{v}, q, \bv{y}]}$, then Problems~\ref{prob: IFIS} and~\ref{prob: IFCS} can be compactly stated as
\begin{problem}[Grouped dual formulation]
\label{prob: IFG}
Given an initial condition $\xi_0 \in \mathscr{Z}$, for all $t \in (0,T)$ find $\xi(t) \in \mathscr{Z}$, such that
\begin{equation}
\label{eq:formal grouped dual}
\langle \mathcal{F}(t, \xi, \xi') , \psi \rangle =0, \quad \forall \psi \in \mathscr{Z}_0,
\end{equation}
where the full expression of $\mathcal{F} : \mathscr{Z} \mapsto \mathscr{Z}_0^*$ is defined as in Problem~\ref{prob: IFIS} or Problem~\ref{prob: IFCS}.
\end{problem}
\begin{remark}[Initial condition for the pressure]
  In Problem~\ref{prob: IFG}, an initial condition for the triple $\xi_0 = \trans{[\bv{u}_0, p_0, \bv{w}_0]}$ is required, just as a matter of compact representation of the problem. However, only the initial conditions  $\bv{u}_0$ and $\bv{w}_0$ are used, since we have no time derivative for the pressure which is a Lagrange multiplier for the incompressible part of the problem, or completely determined by the solution in the compressible part.  
\end{remark}

\subsection{Stability of the abstract variational formulation}
\label{sec: Stability of the continuum problem}
The definition of the operators $\mathcal{M}_{\svs\sv}$ and $\mathcal{N}_{\svs\sv}(\bv{u})$, along with the concept of material time derivative for Eulerian fields, yield the following result:
\begin{equation}
  \label{eq:missing term in seimidiscrete}
  \dualV{ \mathcal{M}_{\svs\sv} \bv{u}',\bv{u}} + \dualV{\mathcal{N}_{\svs\sv}(\bv{u}) \bv{u}, \bv{u}} = \int_{\Omega} \tfrac{1}{2} \rho_{\f} \, \dot{\overline{\bv{u}^{2}}} \d{v}.
\end{equation}
Keeping in mind that, for $\bv{x} = \bv{s} + \bv{w}(\bv{s},t)$, we have
\begin{equation}
\label{eq: MTD of u on B}
\dot{\bv{u}} = \frac{\partial}{\partial t} \bv{u}(\bv{s}+\bv{w}(\bv{s},t),t) = \frac{\partial\bv{u}(\bv{x},t)}{\partial t} \bigg|_{\bv{x} = \bv{s} + \bv{w}(\bv{s},t)} + \nabla_{\bv{x}}\bv{u}(\bv{x},t)\big|_{\bv{x}=\bv{s}+\bv{w}(\bv{s},t)} \frac{\partial \bv{w}(\bv{s},t)}{\partial t},
\end{equation}
and as a straightforward application of Eq.~\eqref{eq: General transport theorem classic control volume} in Lemma~\ref{lemma: transport th for control volumes and bodies} (with $\Omega$ replaced by $B$), we have that our definition of the operators $\delta \mathcal{M}_{\svs\sv}(\bv{w})$ and~$\delta\mathcal{N}_{\svs\sv}(\bv{w},\bv{w}',\bv{u})$ is such that
\begin{multline}
  \label{eq:approx delta k}
  \dualV{\delta \mathcal{M}_{\svs\sv}(\bv{w}) \bv{u},\bv{u}} + \dualV{\delta \mathcal{N}_{\svs\sv}(\bv{w}, \bv{w}',\bv{u}) \bv{u}, \bv{u}}
\\
= \frac{\nsd{}}{\nsd{t}} \int_{B} \tfrac{1}{2} \rho_{s_0}(\bv{s}) \bv{u}^2\big|_{\bv{x}=\map[\bv{w}]} \d{V} - \int_{B_{t}} \tfrac{1}{2} \rho_{\f} \, \dot{\overline{\bv{u}^{2}}} \, \d{v},
\end{multline}
where the above result is due to the fact that we selected $\bv{w}'$ instead of $\bv{u}$ in the nonlinear advection term of the acceleration of the solid.  Therefore, from Eqs.~\eqref{eq:missing term in seimidiscrete} and~\eqref{eq:approx delta k}, our formulation is such that
\begin{multline}
  \label{eq:approx kin e complete}
  \dualV{\mathcal{M}_{\svs\sv} \bv{u}',\bv{u}}+ \dualV{ \mathcal{N}_{\svs\sv}(\bv{u}) \bv{u},\bv{u}}
\\
+ \dualV{\delta \mathcal{M}_{\svs\sv}(\bv{w}) \bv{u},\bv{u}} +  \dualV{\delta \mathcal{N}_{\svs\sv}(\bv{w}, \bv{w}',\bv{u}) \bv{u}, \bv{u}}
\\
= 
\int_{\Omega\setminus B_{t}} \tfrac{1}{2} \rho_{\f} \, \dot{\overline{\bv{u}^{2}}} \d{v}  + \frac{\nsd{}}{\nsd{t}} \int_{B} \tfrac{1}{2} \rho_{s_0}(\bv{s}) \bv{u}^2\big|_{\bv{x}=\map[\bv{w}]} \d{V}.
\end{multline}
Invoking Eq.~\eqref{eq: TTMB ID} in Lemma~\eqref{lemma: transport with mass densities}, we obtain
\begin{equation}
\label{Eq: Navier-Stokes wish}
\int_{\Omega\setminus B_{t}} \tfrac{1}{2} \rho_{\f} \, \dot{\overline{\bv{u}^{2}}} \d{v}
=
\frac{\nsd{}}{\nsd{t}} \int_{\Omega\setminus B_{t}} \tfrac{1}{2} \rho_{\f} \, \bv{u}^{2} \d{v} + \int_{\partial\Omega_N} \tfrac{1}{2} \rho_{\f} \, \bv{u}^{2} \,  \bv{u} \cdot \bv{m} \d{a}.
\end{equation}
Equations~\eqref{eq:missing term in seimidiscrete} and~\eqref{eq:approx delta k} taken together can be written as follows:
\begin{multline}
\label{eq: kin energy estimate}
\dualV{ \mathcal{M}_{\svs\sv} \bv{u}',\bv{u}} + \dualV{\mathcal{N}_{\svs\sv}(\bv{u}) \bv{u}, \bv{u}} 
\\
+ \dualV{\delta \mathcal{M}_{\svs\sv}(\bv{w}) \bv{u},\bv{u}} +  \dualV{\delta \mathcal{N}_{\svs\sv}(\bv{w}, \bv{w}',\bv{u}) \bv{u}, \bv{u}}
=
\frac{\nsd{}}{\nsd{t}} \int_{\Omega} \kappa \d{v} + \int_{\partial\Omega_N}  \kappa \bv{u} \cdot \bv{m} \d{a},
\end{multline}
where $\kappa = \tfrac{1}{2} \rho \bv{u}^{2}$.  The remaining terms in the stability estimates are the viscous dissipative terms
\begin{equation}
  \label{eq:viscous dissipative term}
  \dualV{ \mathcal{D}_{\svs\sv} \bv{u},\bv{u}} = \int_\Omega \hat{\tensor{T}}^{v}_{\f}[ \bv{u}]\cdot \nabla \bv{u} \d{v}, 
\end{equation}
and
\begin{equation}
  \label{eq:delta viscous dissipative term}
  \dualV{\delta \mathcal{D}_{\svs\sv}(\bv{w}) \bv{u},\bv{u}} = \int_{B} \bigl\{ J[\bv{w}]\bigl( \tensor{T}^{v}_{\s}[ \bv{u}] - \tensor{T}^v_f[ \bv{u}] \bigr) \cdot \nabla \bv{u} \bigr\}_{\bv{x}=\map[\bv{w}]}\d{v}, 
\end{equation}
where the combination of the two yields the dissipative term
\begin{equation}
  \label{eq:total viscous dissipative}
 \dualV{\mathcal{D}_{\svs\sv}\bv{u},\bv{u}} + \dualV{\delta \mathcal{D}_{\svs\sv}(\bv{w}) \bv{u},\bv{u}} = \int_\Omega \hat{\tensor{T}}^{v}[ \bv{u}]\cdot \nabla \bv{u} \d{v}.
\end{equation}

The final term needed for the derivation of the full energy estimate pertains to the time rate of change of the elastic energy.  In the proposed immersed method, the coupling between the Eulerian and Lagrangian frameworks is embodied by a variety of operators.  These operators take different forms depending on whether the velocity coupling between $\bv{u}$ and $\bv{w}'$ is used in its strong form, as in Eq.~\eqref{eq: w u rel}, or in its weak form, as in Eq.~\eqref{eq: velocity coupling dual} or~\eqref{eq: velocity coupling dual compressible}.  If we could use directly the strong form of the velocity coupling, namely Eq.~\eqref{eq: w u rel}, we would have that $\bv{u}(\bv{s}+\bv{w}(\bv{s},t),t) = \bv{w}'(\bv{s},t)$, and since the chain rule gives
\begin{equation}
  \label{eq:F grad u = Fdot}
  \tensor{F}[\bv{w}'] = \nabla_{\bv{s}} \bv{w}' = \Bigl( \nabla_{\bv{x}} \bv{u}(\bv{x})\big|_{\bv{x}=\map[w]} \Bigr) \tensor{F}[\bv{w}],
\end{equation}
we would then obtain the usual elastic energy estimates from the definition of the operator $\mathcal{A}_\alpha$:
\begin{equation}
  \label{eq:elastic energy estimate}
  \prescript{}{\mathscr{V}^{*}}{\bigl\langle}
  \mathcal{A}_{\svs}(\bv{w},\bv{w}), \bv{u}
  \big\rangle_{\mathscr{V}} = \frac{\d{}}{\d{t}} \int_B W^{e}_{\s}(\tensor{F}[\bv{w}]) \d{V}. 
\end{equation}
However, for solutions $\bv{u}$ and $\bv{w}'$ of Problem~\ref{prob: IFIS} or~\ref{prob: IFCS}, Eq.~(\ref{eq: w u rel}) holds only in $\mathscr{H}^*_Y$, that is, in its weak form and Eq.~\eqref{eq:elastic energy estimate} can no longer be obtained as just illustrated.  We must therefore proceed in a different way.  Our starting point is the standard estimate in the $\mathscr{H}_{Y}$ space for the (fully Lagrangian form of the) stiffness operator $\mathcal{A}_{\sys}(\bv{w})$: 
\begin{equation}
  \label{eq:lagrangian elastic energy estimate}
  \dualY{\mathcal{A}_{\sys}(\bv{w}), \bv{w}'} = \frac{\d{}}{\d{t}} \int_B W_{\s}^{e}(\tensor{F}[\bv{w}]) \d{V},
\end{equation}
which is valid for $\bv{w}$ in $\mathscr{Y}$ and $\bv{w}'$ in $\mathscr{H}_{Y}$. Using Eq.~(\ref{eq:lagrangian elastic energy estimate}) and Theorem~\ref{th: eulerian vs lagrangian stiffness}, we can prove the following Lemma:
\begin{lemma}[Energy estimate for the immersed elastic operator]
  \label{th: IEO abstract variational formulation}
  Given an elasticity operator $\mathcal{A}_{\sys}(\bv{w})$, then the operator 
$\mathcal{S}_{\svs\sys}(\bv{h})\mathcal{A}_{\sys}(\bv{w})$ 
  satisfies the following energy estimate whenever Eq.~\eqref{eq: velocity coupling dual} or Eq.~\eqref{eq: velocity coupling dual compressible} are satisfied:
  \begin{equation}
    \label{eq:coupled energy estimate}
    \dualV{\mathcal{S}_{\svs\sys}(\bv{w})\mathcal{A}_{\sys}(\bv{w}),\bv{u}}  = \frac{\d{}}{\d{t}} \int_B W^{e}_{\s}(\tensor{F}[\bv{w}]) \d{V}.
  \end{equation}
\end{lemma}
\begin{proof}
Using Eq.~\eqref{eq: velocity coupling dual} or Eq.~\eqref{eq: velocity coupling dual compressible} along with the invertibility of $\mathcal{M}_{\sys\sy}$, the Riesz identity on $\mathscr{H}_{Y}$, we can write
  \begin{equation}
    \label{eq:w' rel u dual}
    \bv{w}' = \mathcal{M}_{\sys\sy}^{-1}\mathcal{M}_{\sys\sv}(\bv{w}) \bv{u}.
  \end{equation}
Substituting Eq.~\eqref{eq:w' rel u dual} into Eq.~\eqref{eq:lagrangian elastic energy estimate}, we have
  \begin{equation}
    \label{eq:proof dual elastic energy estimate starting point}
    \dualY{\mathcal{A}_{\sys}(\bv{w}), \mathcal{M}_{\sys\sy}^{-1}\mathcal{M}_{\sys\sv}(\bv{w}) \bv{u}} =  
\frac{\d{}}{\d{t}} \int_B W^{e}_{\s}(\tensor{F}[\bv{w}]) \d{V}.
  \end{equation}
Focusing on the left-hand side of Eq.~\eqref{eq:proof dual elastic energy estimate starting point}, we can write\footnote{%
Referring to Eq.~\eqref{eq: MB def}, $\mathcal{M}_{\sys\sy}$ is such that $\prescript{}{\mathscr{H}_{Y}^{*}}{\bigl\langle}
\mathcal{M}_{\sys\sy}\bv{w}, \bv{y}
\big\rangle_{\mathscr{H}_Y} = \prescript{}{\mathscr{H}_{Y}^{*}}{\bigl\langle}
\mathcal{M}_{\sys\sy}\bv{y}, \bv{w}
\big\rangle_{\mathscr{H}_Y}$ for all $\bv{w},\bv{y} \in \mathscr{H}_{Y}$.  As a consequence, $\mathcal{M}_{\sys\sy}^{-1}$ is such that $\prescript{}{\mathscr{H}_{Y}^{*}}{\bigl\langle} \bv{y}^{*},
\mathcal{M}_{\sys\sy}^{-1}\bv{w}^{*}
\big\rangle_{\mathscr{H}_Y} = \prescript{}{\mathscr{H}_{Y}^{*}}{\bigl\langle} \bv{w}^{*}, \mathcal{M}_{\sys\sy}^{-1}\bv{y}^{*}, 
\big\rangle_{\mathscr{H}_Y}$, for all $\bv{w}^{*},\bv{y}^{*} \in \mathscr{H}_{Y}^{*}$.}
  \begin{multline}
    \label{eq:proof dual elastic energy estimate}
    \dualY{\mathcal{A}_{\sys}(\bv{w}), \mathcal{M}_{\sys\sy}^{-1}\mathcal{M}_{\sys\sv}(\bv{w}) \bv{u}} =  
    \dualY{\mathcal{M}_{\sys\sv}(\bv{w}) \bv{u}, \mathcal{M}_{\sys\sy}^{-1} \mathcal{A}_{\sys}(\bv{w})}
    \\
     = \dualV{\trans{\mathcal{M}}_{\sys\sv}(\bv{w}) \mathcal{M}_{\sys\sy}^{-1} \mathcal{A}_{\sys}(\bv{w}),\bv{u}},
  \end{multline}
where we have applied the definitions in Eqs.~\eqref{eq: MGamma def} and~\eqref{eq: Mgamma1T def} to obtain the last of the above expressions.  Using the result in Eq.~\eqref{eq:proof dual elastic energy estimate}, Eq.~\eqref{eq:proof dual elastic energy estimate starting point} can be rewritten as
  \begin{equation}
    \label{eq:proof dual elastic energy estimate arrival point}
    \dualV{\trans{\mathcal{M}}_{\sys\sv}(\bv{w}) \mathcal{M}_{\sys\sy}^{-1} \mathcal{A}_{\sys}(\bv{w}),\bv{u}} =  
\frac{\d{}}{\d{t}} \int_B W^{e}_{\s}(\tensor{F}[\bv{w}]) \d{V},
  \end{equation}
and Eq.~\eqref{eq:coupled energy estimate} follows from the application of Theorem~\ref{th: eulerian vs lagrangian stiffness}.
 \end{proof}

Combining the results in Eqs.~\eqref{eq: kin energy estimate}, \eqref{eq:total viscous dissipative}, and \eqref{eq:coupled energy estimate} allows us to state the following theorem:
\begin{theorem}%[Continuous energy estimate]
[Energy estimate for the abstract variational formulation]
\label{th: Continuous energy estimate}
  Let $\bv{u}$, $p$ and $\bv{w}$, be the solutions of either Problem~\ref{prob: IFIS} or Problem~\ref{prob: IFCS}. Then the following energy estimate is satisfied
\begin{multline}
\label{eq: TPE rel continuous}
  \int_{\Omega} \rho \bv{b} \cdot \bv{u} \d{v} +
  \int_{\partial\Omega_{N}} \bv{\tau}_{g} \cdot \bv{u} \d{a}
  = 
  \frac{\nsd{}}{\nsd{t}} \int_{\Omega} \kappa \d{v} + \int_{\partial\Omega_N}  \kappa \bv{u} \cdot \bv{m} \d{a}
  \\
  + \int_{\Omega} \hat{\tensor{T}}^{v}[\bv{u}] \cdot \nabla\bv{u} \d{v} + \frac{\nsd{}}{\nsd{t}} \int_B W^{e}_{\s}(\tensor{F}[\bv{w}]) \d{V},
\end{multline}
where $\kappa = \tfrac{1}{2} \rho \bv{u}^{2}$.
\end{theorem}

\begin{remark}[Validity of Theorem~\ref{th: Continuous energy estimate} in the compressible case]
  When computing the duality between the balance of linear momentum in Eq.~(\ref{eq: BLM Formal dual compressible}) and the exact solution $\bv{u}$, one can use explicitly the properties expressed in Eqs.~\eqref{eq: zero divergence in fluid compressible and condition on pressure in solid compressible} to remove the terms involving the pressure from the estimate, obtaining formally the same result which is obtained in the incompressible case. 
\end{remark}

When the external forces and the boundary conditions are identically zero, we obtain the following result:
\begin{gather}
\label{eq:energy estimates cont prob}
\frac{\nsd{}}{\nsd{t}}\int_{\Omega} \kappa \d{V} + 
\int_{\partial\Omega_{N}} \kappa \bv{u} \cdot \bv{m} \, \d{a}
+ \frac{\nsd{}}{\nsd{t}} \int_{B} W_{\s}^{e} \d{V}
+ \int_{\Omega} \hat{\tensor{T}}^{v} \cdot \nabla_{\bv{x}}\bv{u}\d{V}
= 0.
\shortintertext{and}
\label{eq:energy estimates cont inequality}
\frac{\nsd{}}{\nsd{t}}\int_{\Omega} \kappa \d{V} + 
\int_{\partial\Omega_{N}} \kappa \bv{u} \cdot \bv{m} \, \d{a}
+ \frac{\nsd{}}{\nsd{t}} \int_{B} W_{\s}^{e} \d{V}
\leq 0.
\end{gather}
We observe that the presence of the kinetic energy flux ($\kappa \, \bv{u} \cdot \bv{m}$) is due to the fact that we have posed our problem over a control volume with mixed boundary conditions.  With this in mind, Eq.~\eqref{eq:energy estimates cont prob} states a classical result: the instantaneous variation of the total energy of the system, namely the sum of the kinetic energy, the potential energy, and the kinetic energy flux, is equal to the negative of the internal dissipation.  Furthermore, inequality~\eqref{eq:energy estimates cont inequality} implies that for any physically admissible set of constitutive equations for the solid and the fluid, the abstract variational formulation we propose is asymptotically  stable whenever $\partial\Omega_{N} = \emptyset$.

\section{Discrete Formulation}
\label{sec: Discretization}
\subsection{Spatial Discretization by finite elements}
\label{section: discretization by FEM}
To approximate the continuous problem, we introduce the decompositions $\Omega_{h}$ for $\Omega$ and $B_{h}$ for $B$ into (closed) cells $K$ (triangles or quadrilaterals in 2D, and tetrahedra or hexahedra in 3D) such that the usual regularity assumptions are satisfied:
\begin{enumerate}
\item
$\overline{\Omega} = \cup \{ K \in \Omega_{h} \}$, and $\overline{B} = \cup \{ K \in B_{h} \}$;

\item
Any two cells $K,K'$ only intersect in common faces, edges, or vertices;

\item
The decomposition $\Omega_{h}$ matches the decomposition $\partial \Omega = \partial\Omega_{D} \cup \partial\Omega_{N}$.
\end{enumerate}
On the decompositions $\Omega_{h}$ and $B_{h}$, we consider the finite dimensional subspaces $\mathscr{V}_{h} \subset \mathscr{V}$, $\mathscr{Q}_{h} \subset \mathscr{Q}$, and $\mathscr{Y}_{h} \subset \mathscr{Y}$ defined as
\begin{alignat}{5}
\label{eq: functional space u h}
\mathscr{V}_h &:= \Bigl\{ \bv{u}_h \in \mathscr{V} \,&&\big|\, \bv{u}_{h|K}  &&\in \mathcal{P}_V(K), \, K &&\in \Omega_h \Bigr\} &&\equiv \vssp\{ \bv{v}_{h}^{i} \}_{i=1}^{N_{V}}
\\
\label{eq: functional space p h}
\mathscr{Q}_h &:= \Bigl\{ p_h \in \mathscr{Q} \,&&\big|\, p_{h|K}  &&\in \mathcal{P}_Q(K), \, K &&\in \Omega_h \Bigr\} &&\equiv \vssp\{ q_{h}^{i} \}_{i=1}^{N_{Q}}\\
\label{eq: functional space w h}
\mathscr{Y}_{h} &:= \Bigl\{ \bv{w}_h \in \mathscr{Y} \,&&\big|\, \bv{w}_{h|K} &&\in \mathcal{P}_Y(K), \, K &&\in B_h \Bigr\} &&\equiv \vssp\{ \bv{y}_{h}^{i} \}_{i=1}^{N_{Y}},
\end{alignat}
where $\mathcal{P}_{V}(K)$, $\mathcal{P}_{Q}(K)$ and $\mathcal{P}_{Y}(K)$ are polynomial spaces of degree $r_{V}$, $r_{Q}$ and $r_{Y}$ respectively on the cells $K$, and $N_V$, $N_Q$ and $N_Y$ are the dimensions of each finite dimensional space.

Our choice of finite dimensional spaces $\mathscr{V}_h$ and $\mathscr{Y}_h$ are included in the pivot spaces $\mathscr{H}_V$ and $\mathscr{H}_Y$, respectively, which allow us to use only one discrete space for both $\bv{u}$ and $\bv{u}'$ and one for $\bv{w}$ and $\bv{w}'$.

In the examples, we chose the pair $\mathscr{V}_{h}$ and $\mathscr{Q}_{h}$ so as to satisfy the inf-sup condition for existence, uniqueness, and stability of the approximate solution pertaining to the Navier-Stokes component of the problem (see, e.g., \citealp{BrezziFortin-1991-a}).  Note that the definitions in Eqs.~\eqref{eq: functional space u h}--\eqref{eq: functional space w h} imply that the functions in $\mathscr{V}_{h}$ and $\mathscr{Y}_{h}$ are continuous over $\Omega_{h}$ and $B_{h}$, respectively.

To state the discrete versions of Problems~\ref{prob: IFIS} and~\ref{prob: IFCS}, we first introduce some additional notation.  Given a discrete functional space, say, $\mathscr{V}_{h}$, one of its elements $\bv{u}_{h}$ is identified by the column vector of time dependent coefficients $u_{h}^{j}(t)$, $j = 1,\ldots,N_V$, such that $\bv{u}_{h}(\bv{x},t) = \sum u_{h}^{j}(t) \bv{v}_{h}^{j}(\bv{x})$, where $\bv{v}_{h}^{j}$ is the $\nth{j}$ base element of $\mathscr{V}_{h}$.  With a slight abuse of notation, we will write  $M_{\svs\sv} \bv{u}_h$ to mean the multiplication of the column vector $\bv{u}_h$ by the matrix whose elements $M_{\svs\sv}^{ij}$ are given by
\begin{equation}
  \label{eq:eq:notation matrix element}
  M_{\svs\sv}^{ij} := 
  \prescript{}{\mathscr{V}^{*}}{\bigl\langle}
  \mathcal{M}_{\svs\sv}\bv{v}_h^j,\bv{v}_h^i
  \big\rangle_{\mathscr{V}},
\end{equation}
where the operator in angle brackets is the one defined earlier. The same is intended for all other previously defined operators.

\begin{remark}[Dimension of matrices]
The subscript convention that we adopted for the continuous operators allows one to determine the dimensions of the matrices and of the column vectors involved. For example, the discrete operator $M_{\svs\sv}$ is a matrix with dimensions $N_V\times N_V$, while the matrix $M_{\sys\sv}(\bv{w}_h)$ has dimensions $N_Y\times N_V$.
\end{remark}

\begin{remark}[Discrete duality product]
With the above notation and due to the linearity of the integral operator, we can express duality products in the discrete spaces by simple scalar products in $\mathbb{R}^N$, where $N$ depends on the dimension of the system at hand. For example, given the matrix $M_{\svs\sv}$, then
  \begin{equation}
    \label{eq:duality product in discrete spaces}
    \prescript{}{\mathscr{V}^{*}}{\bigl\langle}
    \mathcal{M}_{\svs\sv}\bv{u}_h,\bv{v}_h
    \big\rangle_{\mathscr{V}} = \bv{v}_h \cdot M_{\svs\sv} \bv{u}_h,
  \end{equation}
where the dot-product on the right hand side is the scalar product in $\mathbb{R}^{N_V}$.
\end{remark}

For a given choice of $\Omega_{h}$ and $B_{h}$, along with corresponding choices of the finite dimensional spaces $\mathscr{V}_{h}$, $\mathscr{Q}_{h}$, and $\mathscr{Y}_{h}$, we reformulate Problem~\ref{prob: IFIS} as follows:
\begin{problem}
\label{prob: IFIS - discrete}
Given $\bv{u}_{0} \in \mathscr{V}_{h}$, $\bv{w}_{0} \in \mathscr{Y}_{h}$, for all $t \in (0,T)$, find $\bv{u}_{h}(t) \in \mathscr{V}_{h}$, $p_{h}(t) \in \mathscr{Q}_{h}$, and $\bv{w}_{h}(t) \in \mathscr{Y}_{h}$ such that
\begin{align}
\label{eq: BLM Formal dual discrete}
% Navier-Stokes part of the problem
\begin{multlined}[b][9cm]
{M}_{\svs\sv}\bv{u}_{h}' +  {N}_{\svs\sv}(\bv{u}_{h})\bv{u}_{h} + 
{D}_{\svs\sv}\bv{u}_{h} + \trans{({B}_{\sqs\sv})}p_{h}
\\
% Coupling part
+ \delta{M}_{\svs\sv}(\bv{w}_{h})\bv{u}_{h}'
+ \delta{N}_{\svs\sv}(\bv{w}_{h},\bv{w}'_{h},\bv{u}_{h})\bv{u}_{h}
\\
+ \delta{D}_{\svs\sv}(\bv{w}_{h})\bv{u}_{h}
+ S_{\svs\sys}(\bv{w}_{h}){A}_{\sys}(\bv{w}_{h})
\end{multlined}
&= {F}_{\svs} + {G}_{\svs}(\bv{w}_{h}),
\\
\label{eq: incompressibility Formal dual discrete}
{B}_{\sqs\sv}\bv{u}_{h} &= 0,
\\
\label{eq: velocity coupling dual discrete}
{M}_{\sys\sy}\bv{w}_{h}' - {M}_{\sys\sv}(\bv{w}_{h})\bv{u}_{h} &= \bv{0},
\end{align}
where $\bv{u}'_{h}(\bv{x},t) = \sum [u_h^j(t)]' \bv{v}_h^j(\bv{x})$ and $\bv{w}'_{h}(\bv{s},t) = \sum [w_h^j(t)]' \bv{y}_h^j(\bv{s})$, and where the prime denotes ordinary differentiation with respect to time.
\end{problem}

\noindent
Similarly, we reformulate Problem~\ref{prob: IFCS} as follows:
\begin{problem}
\label{prob: IFCS - discrete}
Given constant coefficients $c_{1} > 0$ and $c_{2} = 0 \lor 1$, and given initial conditions $\bv{u}_{0} \in \mathscr{V}_{h}$ and $\bv{w}_{0} \in \mathscr{Y}_{h}$, for all $t \in (0,T)$ find $\bv{u}_{h}(\bv{x},t) \in \mathscr{V}_{h}$, $p_{h}(\bv{x},t) \in \mathscr{Q}_{h}$, and $\bv{w}_{h}(\bv{s},t) \in \mathscr{Y}_{h}$ such that
\begin{align}
\label{eq: BLM Formal dual compressible - discrete}
% Navier-Stokes part of the problem
\begin{multlined}[b][9cm]
{M}_{\svs\sv}\bv{u}'_{h} +  {N}_{\svs\sv}(\bv{u}_{h})\bv{u}_{h} + 
{D}_{\svs\sv}\bv{u}_{h} + \trans{\bigl[{B}_{\sqs\sv} + \delta{B}_{\sqs\sv}(\bv{w}_{h})\bigr]}p_{h}
\\
% Coupling part
+ \delta{M}_{\svs\sv}(\bv{w}_{h})\bv{u}'_{h}
+ \delta{N}_{\svs\sv}(\bv{w}_{h},\bv{w}'_{h},\bv{u}_{h})\bv{u}_{h}
\\
+ \delta{D}_{\svs\sv}(\bv{w}_{h})\bv{u}_{h}
+  S_{\svs\sys}(\bv{w}_{h}){A}_{\sys}(\bv{w}_{h})
\end{multlined}
&= {F}_{\svs} + {G}_{\svs}(\bv{w}_{h}),
\\
\label{eq: compressibility Formal dual compressible discrete}
\bigl[ {B}_{\sqs\sv} + \delta{B}_{\sqs\sv}(\bv{w}_{h}) \bigr]\bv{u}_{h} + c_{1} \bigl[\delta{P}_{\sqs\sq}(\bv{w}_{h}) p_{h} - c_{2} \delta{E}_{\sqs\sq}(\bv{u}_{h},\bv{w}_{h},\bv{w}_{h}) \bigr] &= 0,
\\
\label{eq: velocity coupling dual compressible discrete}
{M}_{\sys\sy}\bv{w}'_{h} - {M}_{\sys\sv}(\bv{w}_{h})\bv{u}_{h} &= \bv{0},
\end{align}
where $\bv{u}'_{h}(\bv{x},t) = \sum [u_h^j(t)]' \bv{v}_h^j(\bv{x})$ and $\bv{w}'_{h}(\bv{s},t) = \sum [w_h^j(t)]' \bv{y}_h^j(\bv{s})$, and where the prime denotes ordinary differentiation with respect to time.
\end{problem}

In compact notation, Problems~\ref{prob: IFIS - discrete} and~\ref{prob: IFCS - discrete} can be cast as semi-discrete problems in the space $\mathscr{Z} \supset \mathscr{Z}_h := \mathscr{V}_h \times \mathscr{Q}_h \times \mathscr{Y}_h$ as 
\begin{problem}
\label{prob: Combo prob discrete}
Given an initial condition $\xi_0 \in \mathscr{Z}_h$, for all $t \in (0,T)$ find $\xi_h(t) \in \mathscr{Z}_h$, such that
\begin{equation}
  \label{eq: dae formulation}
  F(t, \xi_h, \xi_h') = 0,
\end{equation}
where
\begin{equation}
  \label{eq: dae formulation bis}
  F^i(t, \xi_h, \xi_h') := \langle \mathcal{F}(t, \xi_h, \xi_h') , \psi^i_h \rangle, \quad i=0, \dots, N_V+N_Q+N_Y,
\end{equation}
and $\mathcal{F}$ has the same meaning as in Eq.~\eqref{eq:formal grouped dual}, with $\psi^i_h$ being the basis function for the spaces $\mathscr{V}_h$, $\mathscr{Q}_h$, or $\mathscr{Y}_h$ corresponding to the given value of $i$.
\end{problem}

\begin{theorem}[Semi-discrete strong consistency]
  \label{th: semi-discrete strong consistency}
  The discrete formulations presented in Problem~\ref{prob: IFIS - discrete} and Problem~\ref{prob: IFCS - discrete}, compactly represented in Problem~\ref{prob: Combo prob discrete}, are strongly consistent. 
\end{theorem}
\begin{proof}
The claim follows immediately from observing that, for the exact solution $\xi := \trans{[\bv{u}, p, \bv{w}]}$, the equalities 
\begin{equation}
  \label{eq: dae formulation consistency}
  F^i(t, \xi, \xi') := \langle \mathcal{F}(t, \xi, \xi') , \psi^i_h \rangle = 0, \quad i=0, \dots, N_V+N_Q+N_Y,
\end{equation}
are satisfied for any conforming approximation, that is,  whenever $\mathscr{V}_h \subseteq \mathscr{V}$, $\mathscr{Q}_h \subseteq \mathscr{Q}$, and $\mathscr{Y}_h\subseteq\mathscr{Y}$.
\end{proof}

\subsection{Variational velocity coupling}
\label{subsection: FEM implementation}
Earlier in the paper we argued that the use of Dirac-$\delta$ distributions is not a theoretical or practical necessity of immersed methods.  We now illustrate our implementation of the operators embodying the \acro{FSI} using the standard ``infrastructure'' of typical \acro{FEM} codes.

The operators $M_{\svs\sv}$, $N_{\svs\sv}(\bv{u}_{h})$, $D_{\svs\sv}$, $B_{\sqs\sv}$, and $F_{\svs}$ in Problems~\ref{prob: IFIS - discrete} and~\ref{prob: IFCS - discrete} are common in variational formulations of the Navier-Stokes problem.  We implemented them in a standard fashion.  The operator $M_{\sys\sy}$ is the mass matrix of the space $\mathscr{Y}_{h}$ and, again, its implementation is standard.  The non-standard operators in our formulation are those with a nonlinear parametric dependence on the field $\bv{w}$ (the motion of the solid).  We now discuss the construction of the matrix $M_{\sys\sv}(\bv{w})$ (corresponding to the operator defined in Eq.~\eqref{eq: MGamma def}) which is responsible for a successful coupling of velocities between the fluid and the solid domain. 

For convenience, we recall that the entries of the matrix ${M}_{\sys\sv}(\bv{w}_{h})$ are given by:
\begin{equation}
  \begin{aligned}[b]
  %   Mixed Mass matrix on B
  \label{eq: MGamma def bis}
  {M}_{\sys\sv}^{ij}(\bv{w}_{h}) & = \prescript{}{\mathscr{H}_{Y}^{*}}{\bigl\langle}
  \mathcal{M}_{\sys\sv}(\bv{w}_{h}) \bv{v}^j_{h}, \bv{y}^i_{h}
  \big\rangle_{\mathscr{H}_{Y}} \\
  & = \Phi_{B} \int_{B}  \bv{v}^j_{h}(\bv{x})\big|_{\bv{x} = \bv{s} + \bv{w}_{h}(\bv{s},t)} \cdot \bv{y}^i_{h}(\bv{s}) \d{V}.
 \end{aligned}
\end{equation}
The construction of $M_{\sys\sv}(\bv{w}_{h})$ requires that we compute the integral in Eq.~\eqref{eq: MGamma def bis}.  As is typical in \acro{FEM}, this is done by summing the contributions due to each cell $K$ of the triangulation $B_{h}$.  These contributions are computed  using quadrature rules with $N_Q$ points.  That is, the contribution of an individual cell is computed by summing the value of the products of the integrand at the quadrature points times the corresponding quadrature weight.  The integrand consists of the functions $\bv{y}^i_{h}(\bv{s})$, whose support is defined over the triangulation of $B_{h}$, and of the functions $\bv{v}^j_{h}(\bv{x})$ (with $\bv{x} = \bv{s} + \bv{w}_{h}(\bv{s},t)$) whose support is instead defined over the triangulation $\Omega_{h}$.
\begin{figure}[htb]
    \centering
    \includegraphics{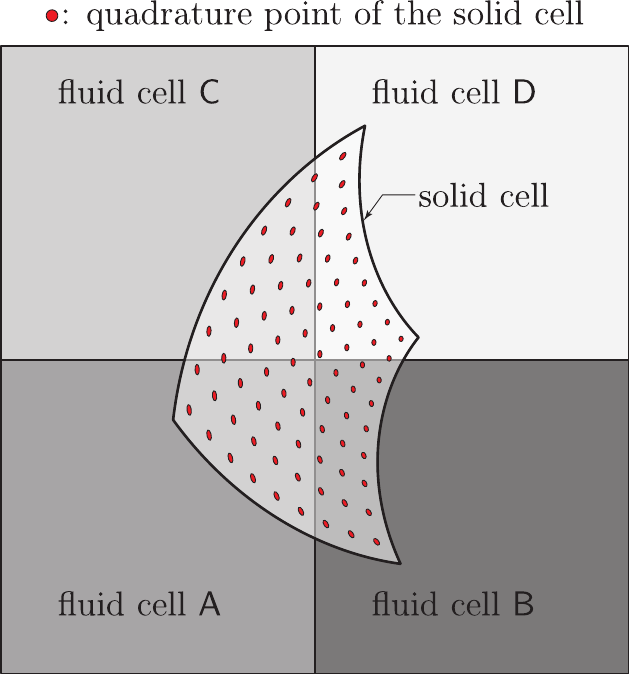}
    \caption{Cells denote as \textsf{A}--\textsf{D} represent a four-cell patch of the triangulation of the fluid domain.  The cell denoted as ``solid cell'' represents a cell of the triangulation of the immersed solid domain that is contained in the union of cells \textsf{A}--\textsf{D} of the fluid domain.  The filled dots represent the quadrature points of the quadrature rule adopted to carry out integration over the cells of the immersed domain.}
    \label{fig: integration}
\end{figure}
Operationally, we perform this calculation as follows.  First we determine the position of the quadrature points of the solid element, both relative to the reference unit element and relative to the global coordinate system adopted for the calculation, through the mappings:
\begin{alignat}{3}
  \label{eq:mapping Khat K solid}
  \bv{s}_K & : \hat{K} := [0,1]^d &&\mapsto K \in B_h, \\
  \label{eq:mapping K K solid}
  I+\bv{w}_h & : K && \mapsto \text{solid cell}.
\end{alignat}
%5
Next, the global coordinates of the quadrature points (obtained through the mappings in Eqs.~\eqref{eq:mapping Khat K solid} and~\eqref{eq:mapping K K solid}) are passed to a search algorithm that identifies the fluid cells in $\Omega_{h}$ containing the points in question, at which the functions $\bv{v}_h^j$ are evaluated. The outcome of this operation is sketched in Fig.~\ref{fig: integration} where, as a way of example, we show the image (under the motion $\bv{s} + \bv{w}_{h}(\bv{s},t)$) of a cell of $B_{h}$ straddling four cells of $\Omega_{h}$ denoted fluid cells \textsf{A}--\textsf{D}.  The quadrature points over the solid cell are denoted by filled circles.  The contribution to the integral in Eq.~\eqref{eq: MGamma def bis} due to the solid cell is then computed by summing the partial contributions corresponding to each of the fluid cells intersecting the solid cell in question:
\begin{equation}
  \begin{aligned}[b]
  %   Mixed Mass matrix on B
  \label{eq: MGamma def tris}
  {M}_{\sys\sv}^{ij}(\bv{w}_{h}) & =
  \sum_{K\in B_h} \int_{K} \bv{v}^j_{h}(\bv{x})\big|_{\bv{x} = \bv{s} + \bv{w}_{h}(\bv{s},t)} \cdot \bv{y}^i_{h}(\bv{s}) \d{V},
  \\
  & \sim \sum_{K\in B_h} \sum_{q=1}^{N_{K,q}} \bv{v}^j_{h}(\bv{x})\big|_{\bv{x} = \bv{s_{K,q}} + \bv{w}_{h}(\bv{s_{K,q}},t)} \cdot \bv{y}^i_{h}(\bv{s}_{K,q}) \omega_{K,q},
 \end{aligned}
\end{equation}
where we denoted with $\bv{s}_{K,q}$ the transformation of the $q$-th quadrature point under the mapping $\bv{s}_K$, defined in Eq.~\eqref{eq:mapping Khat K solid}, and with $\omega_{K,q}$ the corresponding quadrature weight.   In general, the number of quadrature points corresponding to each partial contribution varies. The implementation of an efficient search algorithm responsible for identifying the fluid cells that define the partition of an individual solid cell is the only technically challenging part of the proposed immersed method.  However, several standard techniques are available to deal with this task (see, e.g., \citealp{GridGenHandbook_1998_0,BergCheong_2008_Computational_0}).  Once the fluid cells containing the quadrature points of a given solid cell are found, we determine the value of $\bv{v}^j_{h}$ at the quadrature points using the interpolation infrastructure inherent in the finite element representation of fields defined over $\Omega_{h}$. Our finite element code was developed using the finite element library \texttt{deal.II} \citep{BangerthHartmannKanschat-2007-a}, which provides built-in facilities to carry out precisely this type of calculation. 

\subsection{Variational force coupling}
\label{sec:force coupling}
The coupling between the Eulerian and Lagrangian frameworks is embodied by the operator $\mathcal{S}_{\svs\sys}(\bv{w})$. The discrete version of this operator is constructed using the discrete versions of the operators $\mathcal{M}_{\sys\sv}(\bv{w})$ and $\mathcal{M}_{\sys\sy}$, and Theorem~\ref{th: eulerian vs lagrangian stiffness}:
\begin{equation}
  \label{eq:defition S discrete}
  S_{\svs\sys}(\bv{w}_h) :=  \trans{M}_{\sys\sv}(\bv{w}_h) M^{-1}_{\sys\sy},
\end{equation}
where $M_{\sys\sy}$ is the usual mass matrix for the space $\mathscr{Y}_{h}$, and $\trans{M}_{\sys\sv}(\bv{w}_h)$ is the transpose of the coupling matrix discussed in Section~\ref{subsection: FEM implementation}.

As we observed in Remark~\ref{rem:eulerian vs lagrangian elasticity}, the operators $\mathcal{A}_{\svs}(\bv{h},\bv{w})$ and $\mathcal{S}_{\svs\sys}(\bv{h}) \mathcal{A}_{\sys}(\bv{w})$ are equivalent in the abstract variational formulation. At the discrete level, however, this is no longer the case. When approximating $\mathcal{A}_{\sys}(\bv{w})$, one needs to integrate terms that contain the \emph{gradient} of the basis functions $\bv{y}_h^i$ of the space $\mathscr{Y}_h$. By contrast, the approximation of $\mathcal{A}_{\svs}(\bv{h},\bv{w})$ requires the evaluation of the gradients of the basis functions $\bv{v}_h^i$, in the space $\mathscr{V}_h$, under the map $\bv{s}+\bv{w}_h(\bv{s},t)$. In general, we have that
\begin{equation}
  \label{eq:difference between Ah and the other}
   A_{\svs}(\bv{h}_h, \bv{w}_h)^j := \prescript{}{\mathscr{V}^{*}}{\bigl\langle}
   \mathcal{A}_{\svs}(\bv{h}_h, \bv{w}_h), \bv{v}^j_h
    \big\rangle_{\mathscr{V}} \neq 
    \bigl(\trans{M}_{\sys\sv}(\bv{h}_h) M^{-1}_{\sys\sy} A_{\sys}(\bv{w}_h)\bigr)^j =: \bigl(S_{\svs\sys}(\bv{h}_h) A_{\sys}(\bv{w}_h)\bigr)^j
\end{equation}
and no equivalence can be shown in the discrete space between the two discrete operators in Eq.~(\ref{eq:difference between Ah and the other}).  In principle one could use in the discretization the natural definition of the discrete operator $A_{\svs}(\bv{h}_h, \bv{w}_h)$.  However, only the right hand side of Eq.~\eqref{eq:difference between Ah and the other} can be shown to satisfy a discrete energy estimate equal to that in Eq.~\eqref{eq:elastic energy estimate}:

\begin{theorem}[Discrete energy estimate for immersed elastic operator]
  \label{th: DEE-IEO}
  Given an elasticity operator $\mathcal{A}_{\sys}(\bv{w})$ and its discrete counterpart $A_\sys(\bv{w}_h)$, then the discrete operator 
  \begin{equation}
    \label{eq:discrete immersed elastic operator}
    S_{\svs\sys}(\bv{h}_h) A_{\sys}(\bv{w}_h) := 
    \trans{M}_{\sys\sv}(\bv{h}_h) M^{-1}_{\sys\sy} A_{\sys}(\bv{w}_h)
  \end{equation}
  satisfies the following semi-discrete energy estimate, whenever Eq.~\eqref{eq: velocity coupling dual discrete} or Eq.~\eqref{eq: velocity coupling dual compressible discrete} are satisfied:
  \begin{equation}
    \label{eq:discrete coupled energy estimate}
    \bigl(S_{\svs\sys}(\bv{w}_h)  A_{\sys}(\bv{w}_h)\bigr)\cdot\bv{u}_{h} =  \frac{\d{}}{\d{t}} \int_B W^{e}_{\s}(\tensor{F}[\bv{w}_h]) \d{V}.
  \end{equation}
\end{theorem}
\begin{proof}
  The proof follows closely that of Lemma~\ref{th: IEO abstract variational formulation}.  If we take the scalar product of the semi-discrete version
  of the velocity coupling Eq.~\eqref{eq: velocity coupling dual
    discrete} with the term $\bigl(M^{-1}_{\sys\sy} A_{\sys}(\bv{w}_h)
  \bigr)$, we obtain
  \begin{multline}
    \label{eq:discrete velocity coupling}
    \bigl(M^{-1}_{\sys\sy} A_{\sys}(\bv{w}_h)
    \bigr)\cdot{M}_{\sys\sy}\bv{w}_{h}' - \bigl(M^{-1}_{\sys\sy}
    A_{\sys}(\bv{w}_h) \bigr)\cdot{M}_{\sys\sv}(\bv{w}_{h})\bv{u}_{h}
    \\
    =
    {A}_{\sys}(\bv{w}_h) \cdot\bv{w}_{h}' -
    \bigl(\trans{M}_{\sys\sv}(\bv{w}_{h})M^{-1}_{\sys\sy}
    A_{\sys}(\bv{w}_h) \bigr)\cdot\bv{u}_{h} = 0. \end{multline}
  The discrete estimate deriving from
  Eq.~\eqref{eq:lagrangian elastic energy estimate} then gives immediately the semi-discrete estimate of
  Eq.~\eqref{eq:discrete coupled energy estimate}.
 \end{proof}

\begin{remark}[Spread operator]
Using the approximation strategy described in Eq.~\eqref{eq:difference between Ah and the other}, the only operator that couples \emph{directly} the Eulerian and the Lagrangian framework is $\mathcal{M}_{\sys\sv}(\bv{w})$, whose implementation details have been discussed in Section~\ref{subsection: FEM implementation}. Notice that it is essential to use the same operator in the momentum conservation equation (specifically, its adjoint) to obtain the discrete stability estimate in Eq.~\eqref{eq:discrete coupled energy estimate}.  In the \acro{IBM} literature, the adjoint of $\mathcal{M}_{\sys\sv}(\bv{w})$ is also known as the \emph{spread} operator, because of its role in distributing the forces due to the elastic deformation of the immersed domain to the underlying fluid domain.
\end{remark}

\subsection{Semi discrete stability estimates}
\label{sec:semi discrete stability}
Repeating all passages from Eq.~(\ref{eq:missing term in seimidiscrete}) to Eq.~(\ref{eq: TPE rel continuous}) in the discrete space $\mathscr{V}_h$, we wish we could show semi-discrete stability estimates equivalent to those of the abstract variational formulation.  Unfortunately, contrary to what can be done in the continuous case, in the discrete problem we \emph{cannot} invoke Eq.~\eqref{eq: TTMB ID} in Lemma~\eqref{lemma: transport with mass densities} to say
\begin{equation}
\label{Eq: Navier-Stokes wish discrete}
\int_{\Omega\setminus B_{t}} \tfrac{1}{2} \rho_{\f} \, \dot{\overline{\bv{u}_{h}^{2}}} \d{v}
=
\frac{\nsd{}}{\nsd{t}} \int_{\Omega\setminus B_{t}} \tfrac{1}{2} \rho_{\f} \, \bv{u}_{h}^{2} \d{v} + \int_{\partial\Omega_N} \tfrac{1}{2} \rho_{\f} \, \bv{u}_{h}^{2} \,  \bv{u}_h \cdot \bv{m} \d{a}.
\end{equation}
This is a well known fact in the discretization of the Navier-Stokes equations.  That is, there are stability issues related with the non-linear transport term $\mathcal{N}_{\svs\sv}(\bv{u})$ defined in Eq.~\eqref{eq: NOmega def}. These stability issues originate from the fact that the approximation of Eq.~\eqref{eq: Balance of mass} in the numerical scheme is not satisfied pointwise.  And it is this fact that prevents us from a direct application of Theorems~\ref{th: immersed in control volume} and~\ref{th: TPE}.  With this in mind, we also know that stabilization techniques for the operator in question exist that lead to stable formulations when $\partial\Omega_{N} = \emptyset$ (see, e.g., 
\citealp{HeywoodRannacher_1982_Finite_0}).  Therefore, in the present paper we will limit ourselves to appealing to such stabilization techniques and assume that Eq.~\eqref{Eq: Navier-Stokes wish discrete} is satisfied also at the discrete level.

\begin{theorem}[Semi discrete energy estimate]
  Let $\bv{u}_h$, $p_h$ and $\bv{w}_h$, be the discrete solutions of either Problem~\ref{prob: IFIS - discrete} or Problem~\ref{prob: IFCS - discrete}. Assuming that a stabilized non-linear term $ N_{\svs\sv}(\bv{u}_h) $ is used, such that Eq.~\eqref{Eq: Navier-Stokes wish discrete} is satisfied, then the following semi discrete energy estimate is satisfied
\begin{multline}
\label{eq: TPE rel semi discrete}
  \int_{\Omega} \rho \bv{b} \cdot \bv{u}_h \d{v} +
  \int_{\partial\Omega_{N}} \bv{\tau}_{g} \cdot \bv{u}_h \d{a}
  = 
  \frac{\nsd{}}{\nsd{t}} \int_{\Omega} \kappa_{h} \d{v} + \int_{\partial\Omega_N}  \kappa \bv{u}_h \cdot \bv{m} \d{a}
  \\
  + \int_{\Omega} \hat{\tensor{T}}^{v}[\bv{u}_h] \cdot \nabla\bv{u}_h \d{v} + \frac{\nsd{}}{\nsd{t}} \int_B W^{e}_{\s}(\tensor{F}[\bv{w}_h]) \d{V},
\end{multline}
where $\kappa_{h} = \tfrac{1}{2} \rho \bv{u}_{h}^{2}$.
\end{theorem}

\subsection{Time discretization}
\label{sec:time-discretization}
Equation~\eqref{eq: dae formulation} represents a system of nonlinear differential algebraic equations (\acro{DAE}), which we solve using the \acro{IDA} package of the \acro{SUNDIALS} OpenSource library~\citep{HindmarshBrownGrant-2005-a}.  As stated in the package's documentation (see p.~374 and~375 in \citealp{HindmarshBrownGrant-2005-a})\footnote{We quoted directly from the \acro{SUNDIALS} documentation. However, we adjusted the notation so as to be consistent with ours and we numbered equations according to their order in this paper.}
\begin{quote}
The integration method in \acro{IDA} is variable-order,
  variable-coefficient \acro{BDF} [backward difference formula], in fixed-leading-coefficient form. The method order ranges from 1 to 5, with the \acro{BDF} of order $q$ given by the multistep formula
  \begin{equation}
    \label{eq: bdf or order q}
    \sum_{i=0}^{q} \alpha_{n,i} \xi_{n-i} = h_n \xi_n',
  \end{equation}
where $\xi_{n}$ and $\xi_{n}'$ are the computed approximations to $\xi(t_{n})$ and $\xi'(t_{n})$, respectively, and the step size is $h_{n} = t_{n} - t_{n-1}$. The coefficients $\alpha_{n,i}$ are uniquely determined by the order $q$, and the history of the step sizes.  The application of the \acro{BDF} [in Eq.]~\eqref{eq: bdf or order q} to the \acro{DAE} system [in Eq.]~\eqref{eq: dae formulation} results in a nonlinear algebraic system to be solved at each step:
  \begin{equation}
    \label{eq:dae algebraic system}
    G(\xi_{n}) \equiv F\biggl(t_{n}, \xi_{n},  h_{n}^{-1} \sum_{i=0}^{q} \alpha_{n,i} \xi_{n-i}\biggr) = 0.
  \end{equation}
Regardless of the method options, the solution of the nonlinear system [in Eq.]~\eqref{eq:dae algebraic system} is accomplished with some form of Newton iteration. This leads to a linear system for each Newton correction, of the form
  \begin{equation}
    \label{eq:dae newton correction}
    J[\xi_{n,m+1}-\xi_{n,m}] = -G(\xi_{n,m}),
  \end{equation}
where $\xi_{n,m}$ is the $m$th approximation to $\xi_{m}$. Here $J$ is  some approximation to the system Jacobian
\begin{equation}
  \label{eq:dae Jacobian}
  J = \frac{\partial G}{\partial \xi} = \frac{\partial F}{\partial \xi} +\alpha \frac{\partial F}{\partial \xi'},
\end{equation}
where $\alpha = \alpha_{n,0}/h_{n}$. The scalar $\alpha$ changes whenever the step size or method order changes. 
\end{quote}
In our finite element implementation, we assemble the residual $G(\xi_{n,m})$ at each Newton correction, and let the Sacado package of the Trilinos library \citep{BartlettGayPhipps-2006-a,Gay-1991-a,HerouxBartlettHoekstra-2003-a} compute the Jacobian in Eq.~\eqref{eq:dae Jacobian}. The detailed procedure used in our code to compute the Jacobian through Sacado was taken almost verbatim from the tutorial program \texttt{step-33} of the deal.II library~(\citealp{BangerthHartmann-deal.II-Differential--0}). The final system is solved using a preconditioned GMRES iterative method (see, e.g., \citealp{GolubVan-Loan-1996-a}).

\section{Numerics}
\label{sec: Numerics}

We present a numerical experiment designed to test the main characteristics of the proposed immersed method and those elements that distinguish it from other methods in the literature.  Specifically, we consider the case of a solid with mass density different from that of the fluid.  The solid is assumed to be compressible and viscoelastic.  The dynamic viscosity of the solid is taken to be twice that of the fluid and the elastic part of the behavior was chosen to be a compressible neo-Hookean material. The fluid is modeled as truly incompressible, as opposed to nearly incompressible.

\subsection{Discretization}

The approximation spaces we used in our simulations are the piecewise bi-quadratic spaces of continuous vector functions over $\Omega$ and over $B$ for the approximations of the velocity field $\bv{u}_h$ and of the displacement field $\bv{w}_h$ (usually referred to as the continuous $\mathcal{Q}^2$ space), and the piecewise discontinuous linear space  $\mathcal{P}^1$ over $\Omega$ for the approximation of the pressure field $p$.

The $\mathcal{Q}^2-\mathcal{P}^1$ pair of spaces is known to satisfy the inf-sup condition for the approximation of the Navier-Stokes part of our equations (see, e.g., \citealp{BrezziFortin-1991-a}), while the choice of the space $\mathcal{Q}^2$ for the displacement variable $\bv{w}_h$ is a natural choice, given the underlying velocity field $\bv{u}_h$. With this choice of spaces, Eqs.~\eqref{eq: velocity coupling dual compressible discrete} and~\eqref{eq:  velocity coupling dual discrete} can be satisfied exactly when the solid and the fluid meshes are matching. 

One of the advantages of immersed methods is the possibility to select the meshes over the fluid and the solid domains independently.  However, accuracy issues may arise if the mesh over the solid domain is not sufficiently refined relative to that for the fluid domain. It has been observed (see, e.g., \citealp{Peskin_2002_The-immersed_0}) that a reasonable choice is to take the mapped solid mesh size $h_{\s}$ to be at least one half of the fluid mesh size $h_{\f}$. This choice finds its justification in the approximation properties of both the velocity and the force coupling schemes presented in Section~\ref{subsection: FEM implementation} and Section~\ref{sec:force coupling}. It is essential for the success of immersed methods that the integrals presented in Eq.~\eqref{eq: MGamma def tris} be approximated as accurately as possible. Independently on the choice of approximating spaces, there will be errors in the approximation of these integrals due to the non-matching nature of the fluid and solid meshes. If one uses a fixed number of quadrature points (as in our case), reducing $h_{\s}$ while maintaing $h_{\f}$ constant increases the accuracy of those integrals up to the point in which one element of the solid mesh is entirely contained in an element of the fluid mesh. Further reduction of the solid mesh size beyond this point is not useful, since it only increases the computational cost, without adding accuracy to the method, which is bounded anyway by the fluid mesh size $h_{\f}$. 

The choice $h_{\s} \approx \tfrac{1}{2} h_{\f}$ is a reasonable compromise, for which most of the solid elements are fully contained in a fluid element, and each solid element spans \emph{at most} four elements of the fluid mesh. At run time, whenever a solid mesh element is distorted to span more than four fluid mesh elements, the element in question should be refined to increase the accuracy of the method. Currently, such tests are not implemented in our code, and we select a slightly finer solid mesh to prevent distortion from causing a drift in the accuracy of the method.

An alternative solution is to use adaptive quadrature rules in the approximation of the integrals in Eq.~\eqref{eq: MGamma def tris}, as done, for example, in~\cite{GriffithLuo-2012-a}. This approach allows one to choose $h_{\s}$ independently from $h_{\f}$, and it works effectively even in the case where the solid cell spans several fluid cells. Conservation of mass in this case may be an issue, since the details at which the fluid evolves in the background may not be captured accurately enough by the solid mesh (see~\cite{Griffith-2012-a} for detailed discussion on volume conservation in Immersed Boundary Methods).

\subsection{Constitutive settings}
\label{sec:setting}

We present a simple numerical example concerning a two-dimensional rubber disk, modeled by a viscoelastic compressible material, where the elastic part of the behavior is that of a compressible neo-Hookean material.  The disk is pre-deformed with a uniform compression that changes its diameter to a fraction of its diameter in the reference configuration, and then it is released from rest in a two-dimensional box containing a fluid, also at rest. The dynamic viscosity and mass density of the fluid are $\mu_{\f} = \np[kg/(m^2\ucdot s)]{e-2}$ and density $\rho_{\f} = \np[kg/m^{2}]{1}$, respectively. On the top side of the box we impose homogeneous Neumann boundary conditions, to allow the fluid to enter and exit the box, while on the other three sides we impose homogeneous Dirichlet boundary conditions. 

The reference configuration of the solid is a disk of diameter $\phi = \np[m]{0.125}$, centered at the origin. Its initial displacement field is given by
\begin{equation}
  \label{eq:W0 test 1}
   \bv{w}_0 :=
     \begin{pmatrix}
       -{0.3} s_{1} + \np[m]{0.6} \\
       -{0.3} s_{2} + \np[m]{0.4}
     \end{pmatrix},
\end{equation}
where $s_{1}$ $s_{2}$, expressed in meters, denote the coordinates of points in $B$ relative to the chosen Cartesian coordinate system.  Referring to Eq.~\eqref{eq: Cauchy Response Function}, the constitutive response function of the solid is $\hat{\tensor{T}}_{\s} = \hat{\tensor{T}}^{e}_{\s} + \hat{\tensor{T}}^{v}_{\s}$ with
\begin{equation}
\label{eq: Cauchy Response Function example}
\hat{\tensor{T}}^{e}_{\s} = J^{-1} \hat{\tensor{P}}_{\s}^{e} \trans{\tensor{F}},
\quad
 \hat{\tensor{P}}_{\s}^{e} := G \Bigl[\tensor{F} - J^{-2\nu/(1 - 2 \nu)}\tensor{F}^{-T}\Bigr],
\quad
\hat{\tensor{T}}^{v}_{\s} = 2 \mu_{\s} \tensor{D},
\end{equation}
where $G = \np[Pa\ucdot m]{20}$, $\nu = 0.3$, $\mu_{\s} = \np[kg/(m^2\ucdot s)]{2e-2}$. The mass density of the solid in the reference configuration is $\rho_{\s_{0}} = 0.8 \rho_{\f}$. We add a constant external body force density (gravity) directed downwards:
\begin{equation}
  \label{eq:g test 1}
   \bv{b} :=
     \begin{pmatrix}
         0 \\
       \np[m^{2}/s^2]{-10}
     \end{pmatrix}.
\end{equation}
The initial deformation of the disk is such that its density (in the deformed state) exceeds that of the surrounding fluid.  Under these conditions, the disk would sink.  However, as soon as the disk is released, the disk will expand rapidly and regain a size such that the disk will start floating almost from the start of the motion.  

\subsection{Results}
In Figs.~\ref{fig:pressure test 1} and~\ref{fig:velocity test 1}%
\begin{figure}
  \centering
  \subfigure[$t=0$]{
    \includegraphics[width=.35\textwidth]{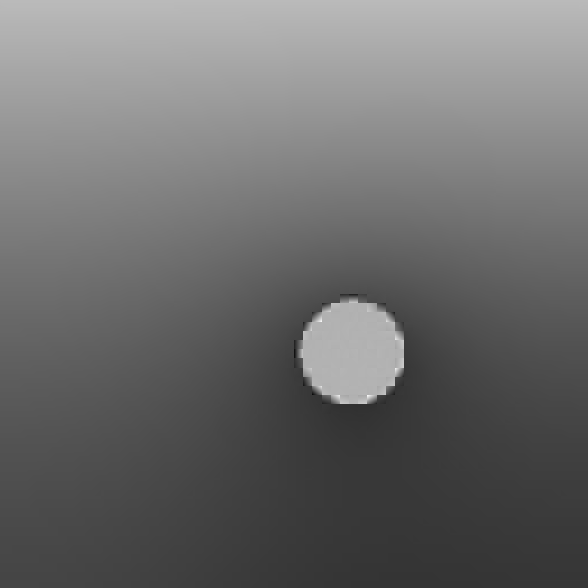}
  }
  \subfigure[$t=.5\,\mathrm{s}$]{
    \includegraphics[width=.35\textwidth]{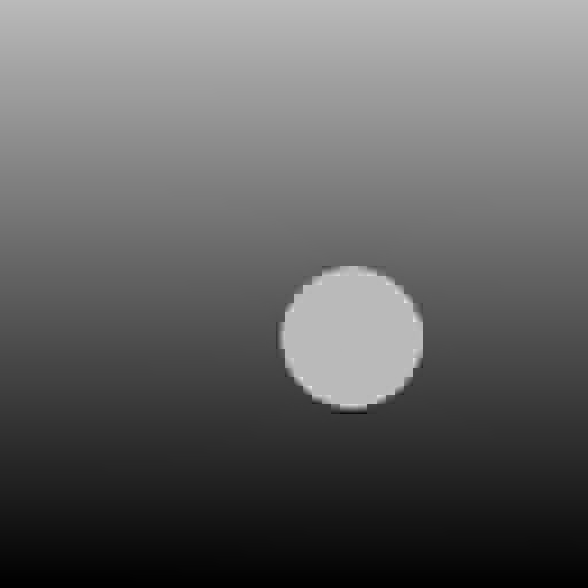}
  }
  \subfigure[$t=1\,\mathrm{s}$]{
    \includegraphics[width=.35\textwidth]{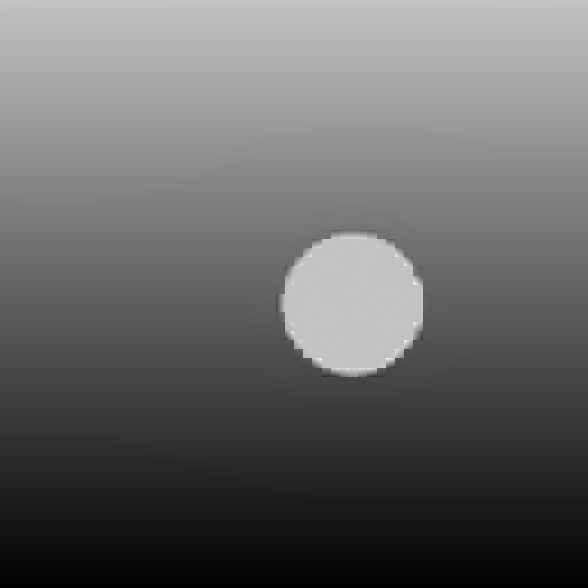}
  }
  \subfigure[$t=1.5\,\mathrm{s}$]{
    \includegraphics[width=.35\textwidth]{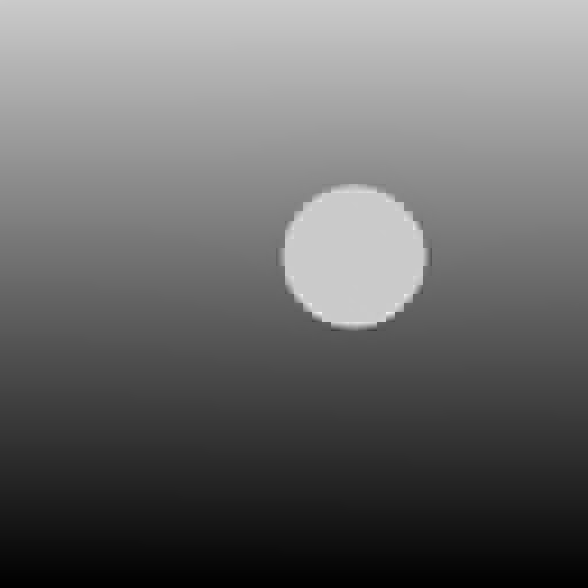}
  }
  \subfigure[$t=2\,\mathrm{s}$]{
    \includegraphics[width=.35\textwidth]{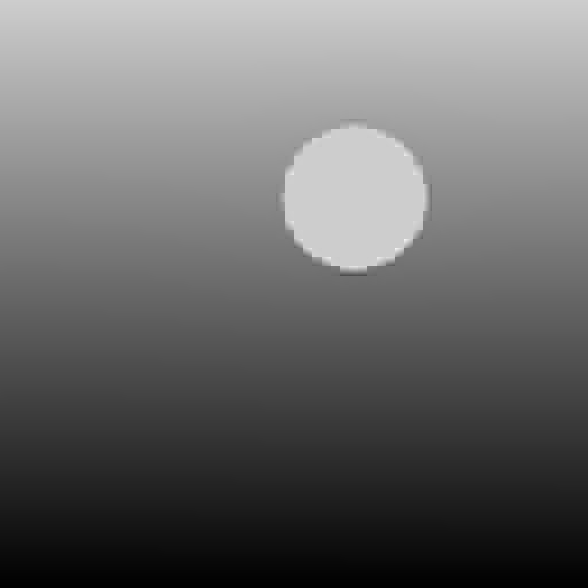}
  }
  \subfigure[$t=2.5\,\mathrm{s}$]{
    \includegraphics[width=.35\textwidth]{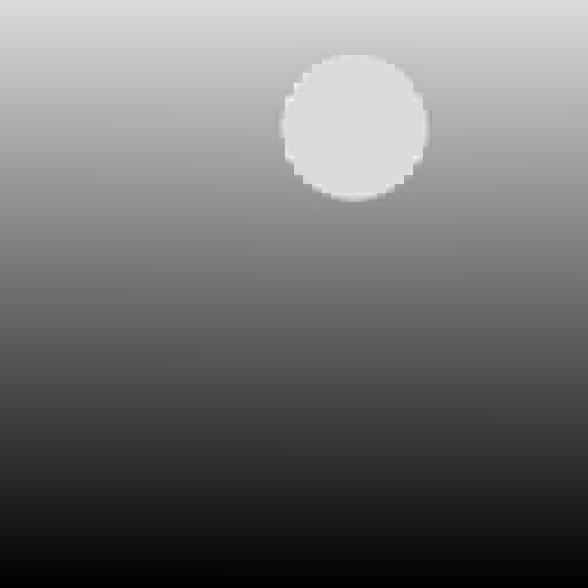}
  }
  \caption{\label{fig:pressure test 1}Pressure evolution.}
\end{figure}
\begin{figure}
  \centering
  \subfigure[$t=0$\label{fig:test 1 expansion state}]{
    \includegraphics[width=.35\textwidth]{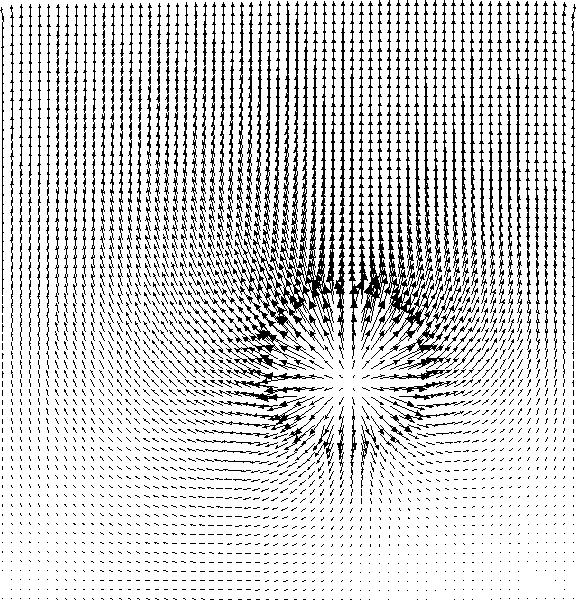}
 }
  \subfigure[$t=.5\,\mathrm{s}$]{
    \includegraphics[width=.35\textwidth]{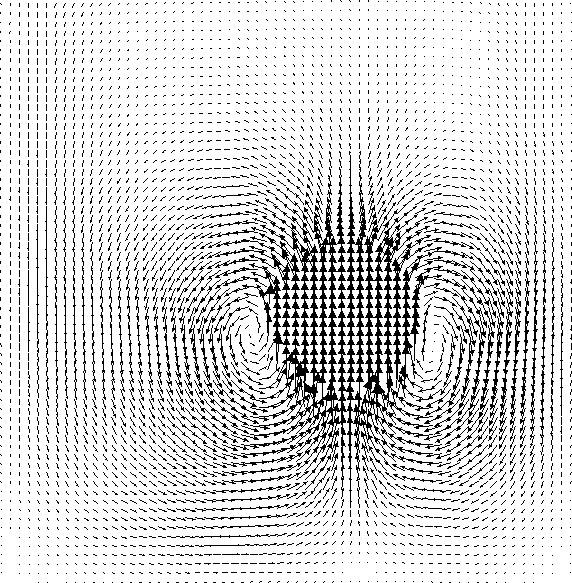}
  }
  \subfigure[$t=1\,\mathrm{s}$]{
    \includegraphics[width=.35\textwidth]{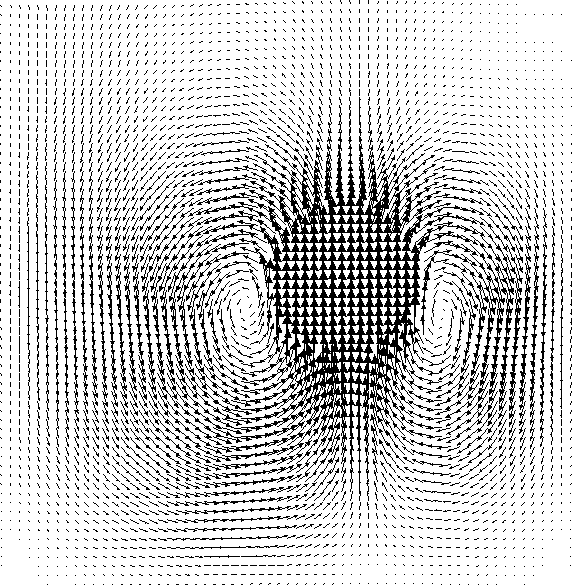}
  }
  \subfigure[$t=1.5\,\mathrm{s}$]{
    \includegraphics[width=.35\textwidth]{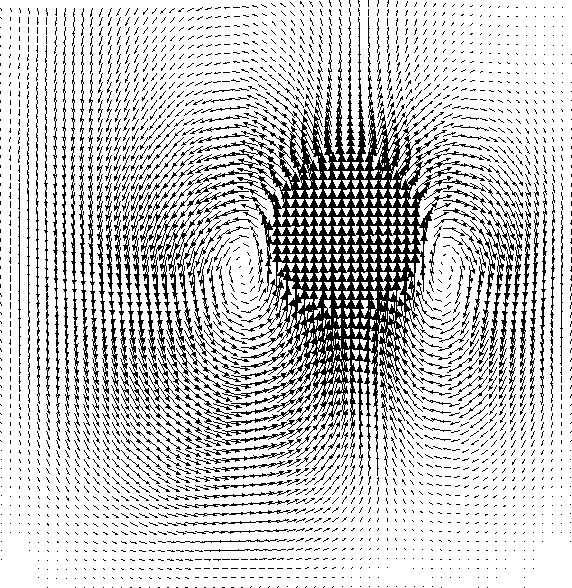}
  }
  \subfigure[$t=2\,\mathrm{s}$]{
    \includegraphics[width=.35\textwidth]{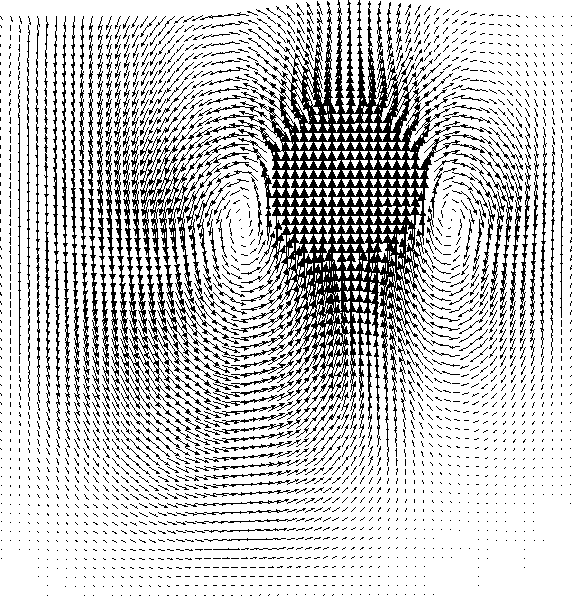}
  }
  \subfigure[$t=2.5\,\mathrm{s}$]{
    \includegraphics[width=.35\textwidth]{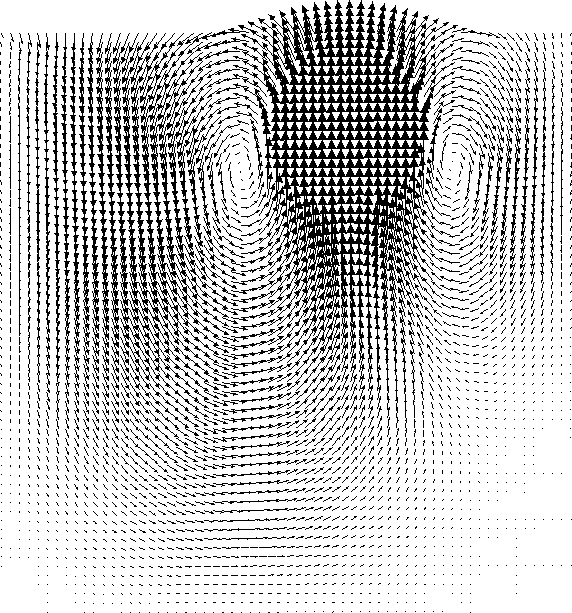}
  }
  \caption{\label{fig:velocity test 1}Velocity evolution.}
\end{figure}
we show the evolution of the pressure and of the velocity fields, respectively. The plots in Fig.~\ref{fig:pressure test 1} show the mean normal stress in both the fluid and the solid, as defined in Eq.~\eqref{eq: Mean normal stress def}. 

In the figures, three different phases can be recognized:
\begin{enumerate}
\item
expansion, for $0 < t < \np[s]{0.3}$; %($0 \rightarrow 0.3s$);

\item
contraction and ascension, for $\np[s]{0.3} < t < \np[s]{2}$; %($0.3s \rightarrow 2s$);

\item
expansion and rising, for $\np[s]{2} < t < \np[s]{3}$. %($2s\rightarrow 3s$).
\end{enumerate}
In phase one, the disk tries to reach an equilibrium state of deformation by quickly expanding and pushing the surrounding fluid. Fig.~\ref{fig:test 1 expansion state} shows a snapshot of the velocity field in this phase: an outflow is present at the top of the box, and the radial velocity in the solid shows the role of compressibility in the solid constitutive behavior. Quantitative measurements can be inferred from Fig.~\ref{fig:test 1 flux} displaying the plot of the total instantaneous flux through the Neumann part of the boundary:
\begin{equation}
  \label{eq:total flux}
  F := \int_{\Gamma_N} \bv{u}\cdot\bv{n}\d s.
\end{equation}

In phase two, the solid has reached a positive buoyancy status, and starts moving towards the top of the box. In this phase, the vertical position of the center of mass of the disk (Fig.~\ref{fig:test 1 center of mass}), increases in an approximately quadratic manner with time. The area of the disk bounces back due to inertial effects (lighter line in Fig.~\ref{fig:test 1 area comparison}), and in phase three, it grows again, both for a bouncing effect and because of the reduced pressure applied on the surface of the disk itself.

By tracking the vertical location of the center of mass of the disk (Fig.~\ref{fig:test 1 center of mass}), we observe that the dynamics of the expansion phase are rather fast, and the disk expands to a positive buoyancy state while remaining substantially still. We monitor the consistency of the method by computing the integral in time of the total flux of fluid through the top side, which should equate the area change of the disk (Fig.~\ref{fig:test 1 area comparison}):
\begin{equation}
  \label{eq:area comparison}
  \delta A_{\f} :=
  \int_{0}^{t} F(\tau) \d \tau 
  =
  \int_{0}^{t} \int_{\Gamma_N} \bv{u}(\tau) \cdot \bv{n} \d s \d \tau \approx \int_B J[\bv{w}(t)] \d v - \int_B J[\bv{w}_0] \d v := \delta A_{\s}.
\end{equation}
While the consistency of the method in phase one is quite good, the fluid flux does not seem to compensate accurately for the small changes in area due to the bouncing of the disk in the remaining two phases. This lack of accuracy is likely due to the combination of the errors in the approximation of the divergence free constraint in the fluid, and to the errors in the computation of the velocity coupling between the fluid and the solid. 
\afterpage{\clearpage}
\begin{figure}[htb]
  \subfigure[Fluid flux through the top of the box.]{
    \includegraphics[width=.48\textwidth]{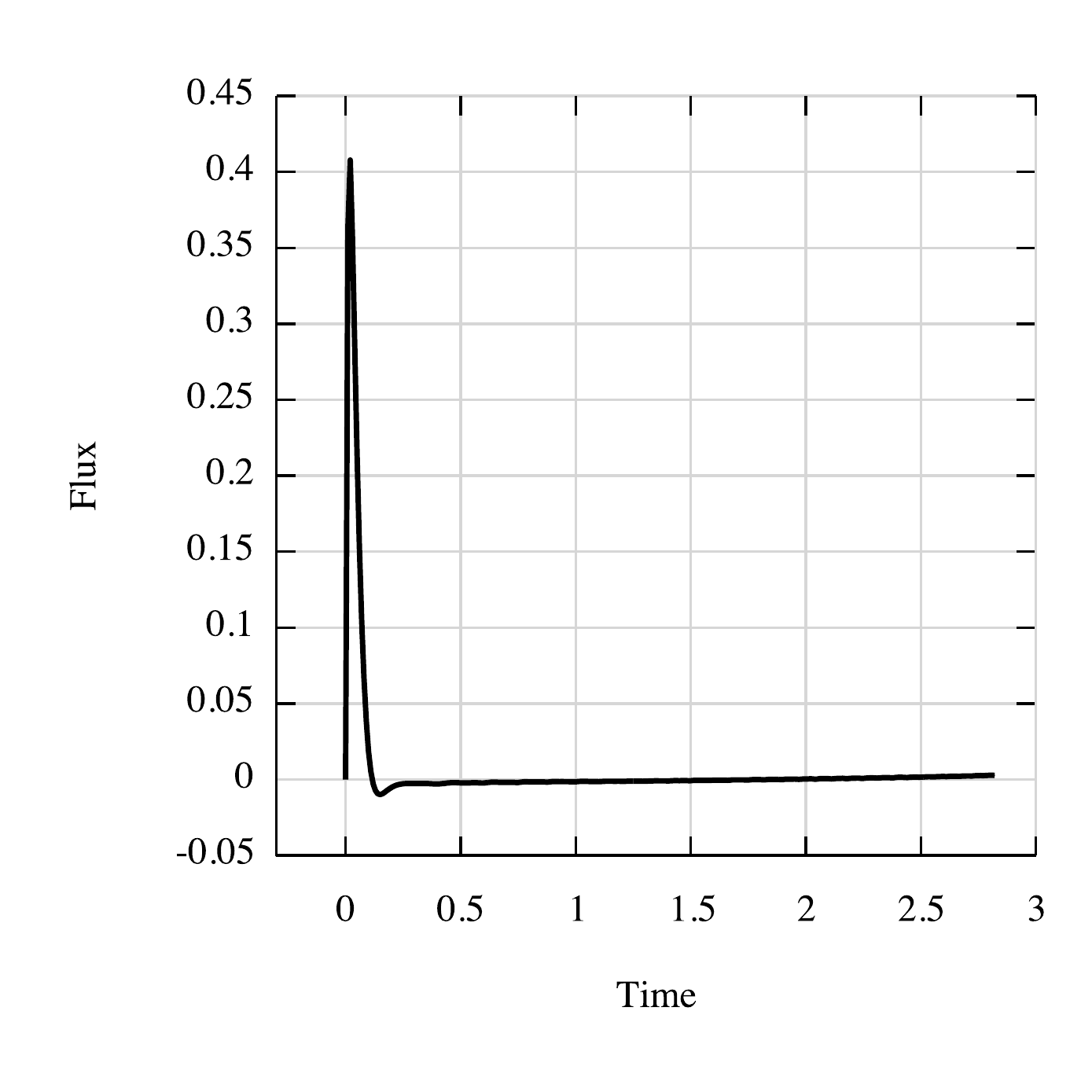}
    \label{fig:test 1 flux}
  }
  \subfigure[Vertical position of the center of mass of the ball.]{
    \includegraphics[width=.48\textwidth]{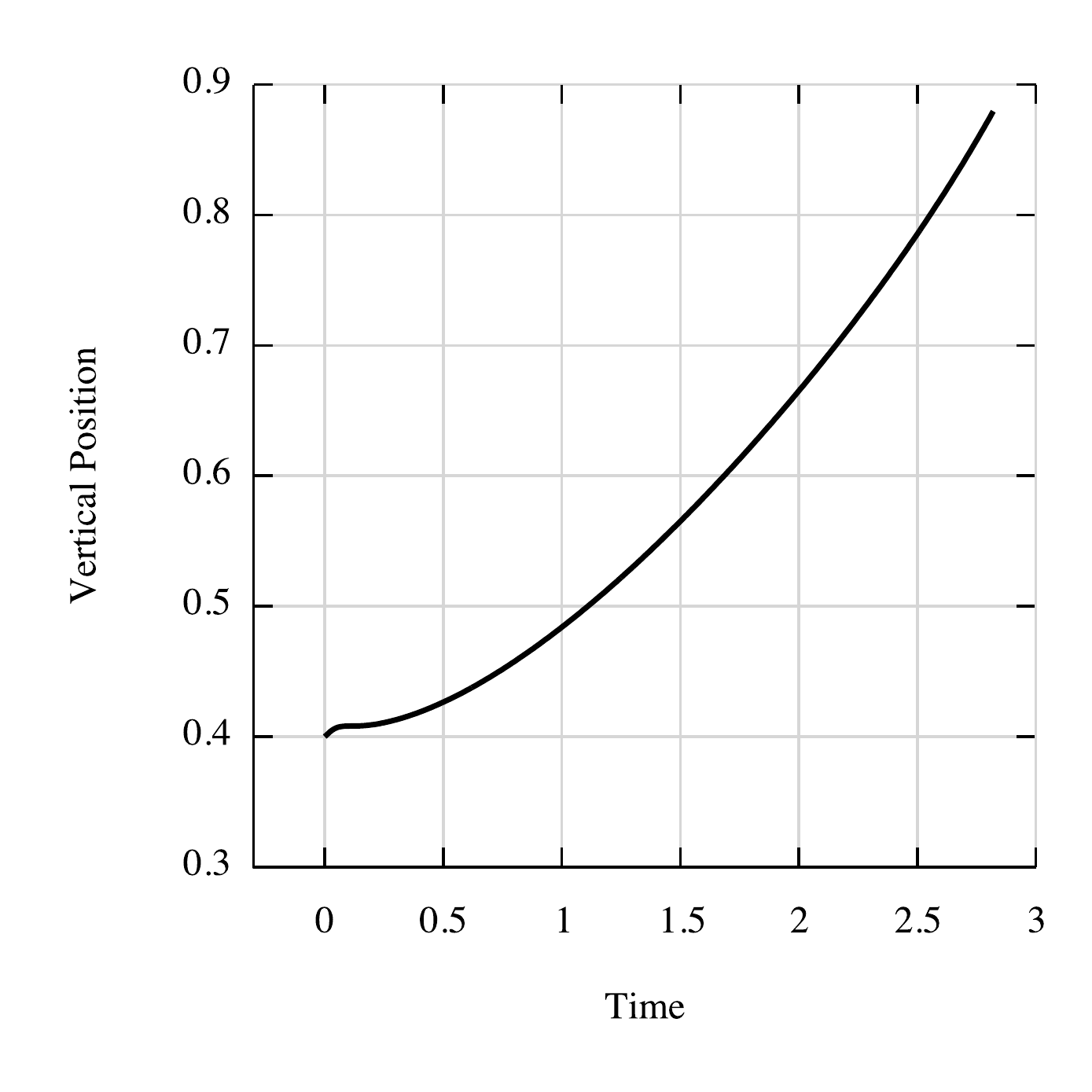}
    \label{fig:test 1 center of mass}
  }
  \subfigure[Area exchange comparison.]{
    \includegraphics[width=.48\textwidth]{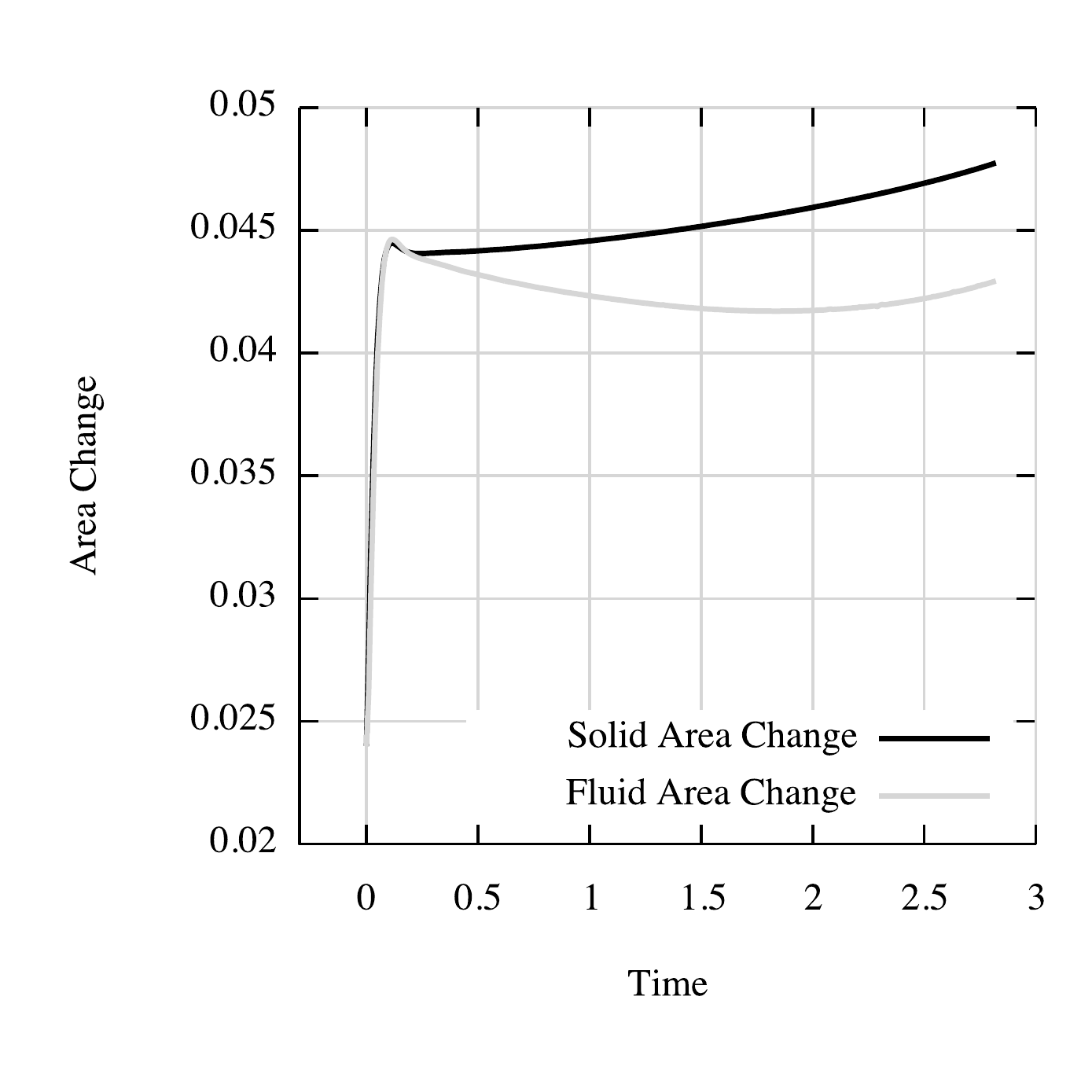}
    \label{fig:test 1 area comparison}
  }
  \subfigure[Area exchange error.]{
    \includegraphics[width=.48\textwidth]{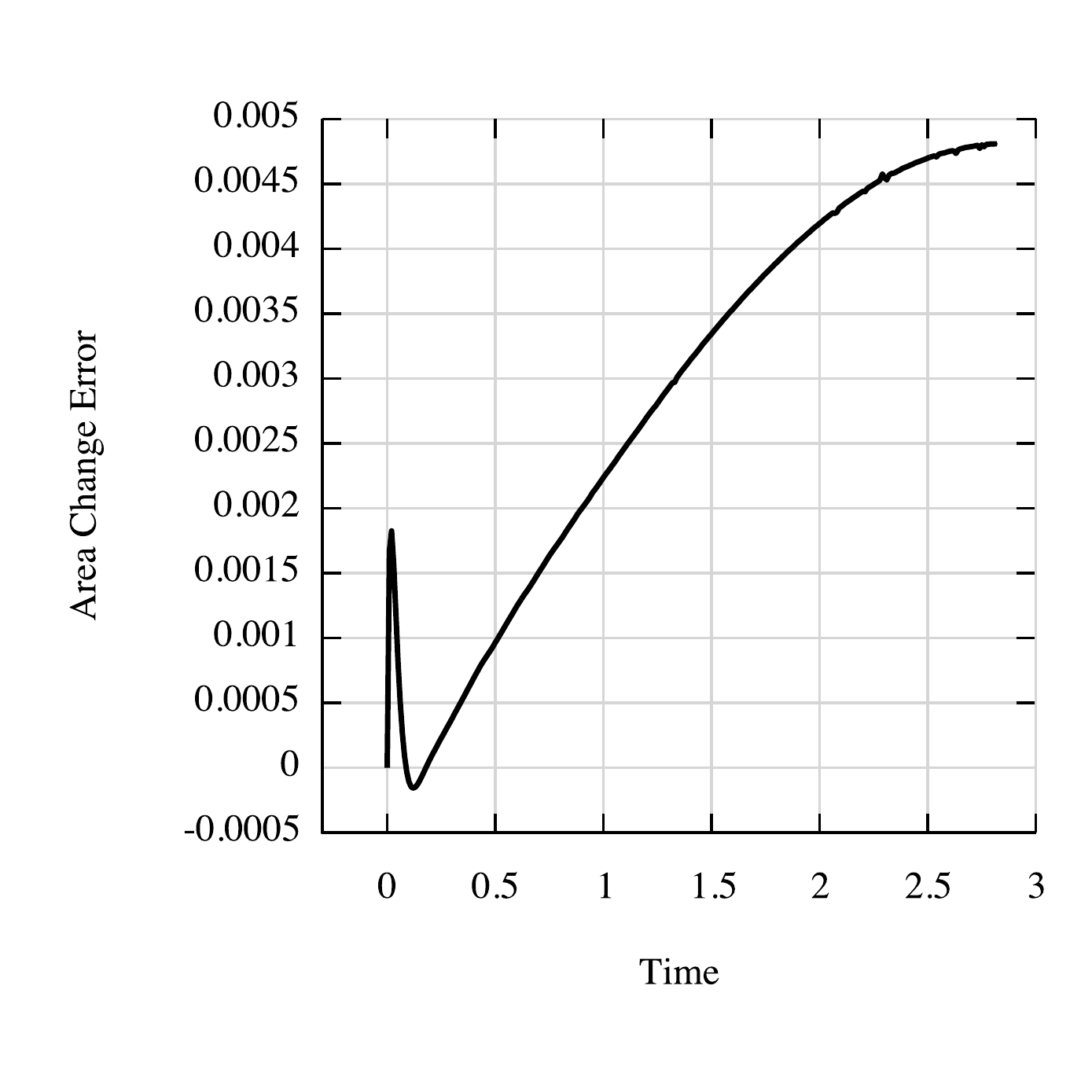}
    \label{fig:test 1 area comparison error}
  }
  \caption{Instantaneous flux, position, and area exchange.}
\end{figure}

\section{Summary and Conclusions}
We presented a fully variational formulation of an immersed method for the solution of \acro{FSI} problems. Like other immersed methods, ours is based on the idea of keeping independent discretizations for the fluid and for the solid domains. The fluid is treated in its natural Eulerian framework, while the solid is modeled using a Lagrangian strategy. Most of the implementations of immersed methods refer to the pioneering work of \cite{Peskin_1977_Numerical_0}, in which a clever reformulation of the continuous coupling between the fluid and the solid domain allows one to construct projection operators between the Lagrangian and the Eulerian framework based on approximated Dirac-$\delta$ distributions.  While the necessity to introduce approximated Dirac-$\delta$ distributions is strongly connected to the particular approximation strategy chosen to discretize the continuous problem (\acro{FD} in the \acro{IBM}), its use has propagated also in the Finite Element community (see, e.g., \citealp{ZhangGerstenberger_2004_Immersed_0}). A variational approach that removed the necessity to approximate the Dirac-$\delta$ distribution has been proposed in \cite{BoffiGastaldi_2003_A-Finite_0}, and later extended in \cite{Heltai_2006_The-Finite_0,BoffiGastaldiHeltai-2007-a,BoffiGastaldiHeltaiPeskin-2008-a,Heltai-2008-a}. 

The formulation we presented extends that of \cite{BoffiGastaldiHeltaiPeskin-2008-a} to general elasticity problems in which the solid and the fluid can have different mass densities.  In addition, the constitutive response function for the solid can be either compressible or incompressible, and either viscoelastic of differential type or purely elastic. The abstract variational formulation we proposed is shown to yield energy estimates that are formally identical to those in the classical context of continuum mechanics.  The numerical approximation we proposed is strongly consistent and stable, with semi-discrete energy estimate that are formally identical to those of the abstract variational formulation and therefore to those in the classical context of continuum mechanics.

We discussed in details the algorithmic strategies for an efficient implementation of the proposed method, and we showed how standard implementations of the finite element method along with some appropriate search algorithm for the determination of the element containing a given point are enough to implement the proposed formulation.

A simple numerical experiment was used to test the novel characteristics of our method. While the results are promising, some work is still necessary to ensure better conservation properties of the method.

\section*{Acknowledgements}
\label{sec:acknowledgements}
The research leading to these results has received specific funding under the ``Young SISSA Scientists' Research Projects'' scheme 2011-2012, promoted by the International School for Advanced Studies (SISSA), Trieste, Italy.

The final version of this paper was much improved thanks to suggestions made by the anonymous reviewers.

\section*{Appendix: Proof of Theorem~\ref{th: immersed in control volume}}
The proof of Theorem~\ref{th: immersed in control volume} is presented at the end of this appendix and is preceded by some useful intermediate results.

The application of Theorem~\ref{th: GTT} to the domains $\Omega$ and $B_{t}$ defined in Section~\ref{subsec: Basic notation and governing equations} yields the following results:
\begin{lemma}[Transport theorem for fixed control volumes and for physical bodies]
\label{lemma: transport th for control volumes and bodies}
Let $\Omega$ and $B_{t}$, with outward unit normals $\bv{m}$ and $\bv{n}$, respectively, be the domains defined in Section~\ref{subsec: Basic notation and governing equations}.  Let $\phi(\bv{x},t)$ and $\xi(\bv{x},t)$ be a smooth field defined over $\Omega$ and $B_{t}$, respectively.  Then we have
\begin{gather}
\label{eq: General transport theorem classic control volume}
\frac{\nsd{}}{\nsd{t}} \int_{\Omega} \phi(\bv{x},t) \d{v} = \int_{\Omega} \frac{\partial \phi(\bv{x},t)}{\partial t} \d{v}
\shortintertext{and}
\label{eq: General transport theorem body}
\frac{\nsd{}}{\nsd{t}} \int_{B_{t}} \xi(\bv{x},t) \d{v} = \int_{B_{t}} \frac{\partial \xi(\bv{x},t)}{\partial t} \d{v} + \int_{\partial B_{t}} \xi(\bv{x},t) \, \bv{u} \cdot \bv{n} \d{a}.
\end{gather}
\end{lemma}
\begin{proof}[Proof of Lemma~\ref{lemma: transport th for control volumes and bodies}]
The results in Eqs.~\eqref{eq: General transport theorem classic control volume} and~\eqref{eq: General transport theorem body} are well known.  The proof is presented simply to facilitate the discussion of subsequent results.  Since $\Omega$ is a fixed control volume, its boundary is time independent.  Hence, Eq.~\eqref{eq: General transport theorem classic control volume} follows directly from Eq.~\eqref{eq: General transport theorem} when we let $\bv{\nu} = \bv{0}$.  Next, we observe that $B_{t}$ is a time-dependent domain such that the velocity field on the boundary of $B_{t}$ coincides with the material velocity field $\bv{u}$.  Hence, Eq.~\eqref{eq: General transport theorem body} follows from Eq.~\eqref{eq: General transport theorem} when we set $\bv{\nu} = \bv{u}$.
\end{proof}
\begin{lemma}[Transport theorem for $\Omega\setminus B_{t}$]
\label{lemma: TT Omega minus Bt}
Let $\Omega$ and $B_{t}$, with outward unit normals $\bv{m}$ and $\bv{n}$, respectively, be the domains defined in Section~\ref{subsec: Basic notation and governing equations}.  Let $\phi(\bv{x},t)$ be a smooth field defined over the domain $\Omega\setminus B_{t}$. Then we have
\begin{equation}
\label{eq: General transport theorem immersed context}
\frac{\nsd{}}{\nsd{t}} \int_{\Omega\setminus B_{t}} \phi(\bv{x},t) \d{v} = \int_{\Omega\setminus B_{t}} \frac{\partial \phi(\bv{x},t)}{\partial t} \d{v} - \int_{\partial B_{t}} \phi(\bv{x},t) \, \bv{u} \cdot \bv{n} \d{a}.
\end{equation}
\end{lemma}
\begin{proof}[Proof of Lemma~\ref{lemma: TT Omega minus Bt}]
We observe that
\begin{equation}
\label{eq: boundary of Omega minus Bt}
\partial(\Omega\setminus B_{t}) = \partial\Omega \cup \partial B_{t}.
\end{equation}
The unit normals outward relative to $\Omega\setminus B_{t}$ on $\partial\Omega$ and $\partial B_{t}$ are $\bv{m}$ and $-\bv{n}$, respectively.  Finally, the velocity field of $\partial\Omega$ is null whereas the velocity field on $\partial B_{t}$ is equal to the material velocity field $\bv{u}$ (on $\partial B_{t}$).  Then the result in Eq.~\eqref{eq: General transport theorem immersed context} follows directly from Eq.~\eqref{eq: General transport theorem} by setting $\bv{\nu} = \bv{0}$ on $\partial\Omega$ and $\bv{\nu} = \bv{u}$ on $\partial B_{t}$.
\end{proof}

The coordinated application of the transport theorems with the balance of mass and the concept of material time derivative yields results that are useful in the derivation of energy estimates.  If $\phi(\bv{x},t)$ is the Eulerian description of a scalar-valued physical quantity, then the material time derivative of $\phi$ is
\begin{equation}
\label{eq: material time derivative}
\dot{\phi} = \frac{\partial \phi}{\partial t} + \grad\phi \cdot \bv{u},
\end{equation}
where we recall that $\bv{u}(\bv{x},t)$ is the Eulerian description of the material velocity field.  We now consider the case in which $\phi$ is a density per unit volume with corresponding density per unit mass $\psi$, so that $\phi(\bv{x},t) = \rho(\bv{x},t) \psi(\bv{x},t)$, where $\rho(\bv{x},t)$ is Eulerian description of the mass density distribution. Then,  Eq.~\eqref{eq: material time derivative} gives
\begin{equation}
\label{eq: rho psi ppt}
\dot{\rho} \psi + \rho \dot{\psi} = \frac{\partial (\rho \psi)}{\partial t} + \grad(\rho \psi) \cdot \bv{u}
\quad \Rightarrow \quad
\frac{\partial (\rho \psi)}{\partial t} = \dot{\rho} \psi + \rho \dot{\psi} - \grad(\rho \psi) \cdot \bv{u}.
\end{equation}
Using Eq.~\eqref{eq: Balance of mass} and recalling that $\ldiv(\phi \bv{u}) = \grad\phi \cdot \bv{u} + \phi \ldiv \bv{u}$, the last of Eqs.~\eqref{eq: rho psi ppt} becomes
\begin{equation}
\label{eq: D/DT pd/pdt BM relation}
\frac{\partial (\rho \psi)}{\partial t} = \rho \dot{\psi} - \ldiv(\rho \psi \bv{u}).
\end{equation}
\begin{lemma}[Transport theorems for densities per unit mass]
\label{lemma: transport with mass densities}
Let $\Omega$ and $B_{t}$ be the domains defined in Section~\ref{subsec: Basic notation and governing equations}.  Let $\psi_{B_{t}}(\bv{x},t)$ and $\psi_{\Omega\setminus B_{t}}(\bv{x},t)$ be the Eulerian descriptions of sufficiently smooth scalar-valued physical density per unit mass defined over $B_{t}$, and $\Omega\setminus B_{t}$, respectively.  Then, Theorem~\ref{th: GTT} and the principle of balance of mass in Eq.~\eqref{eq: Balance of mass} imply
\begin{gather}
\label{eq: TTMB B}
\frac{\nsd{}}{\nsd{t}} \int_{B_{t}} \rho \psi_{B_{t}} \d{v}
=
\int_{B_{t}} \rho \dot{\psi}_{B_{t}} \d{v},
\shortintertext{and}
\label{eq: TTMB ID}
\frac{\nsd{}}{\nsd{t}} \int_{\Omega\setminus B_{t}} \rho \psi_{\Omega\setminus B_{t}} \d{v} + \int_{\partial\Omega} \rho \psi_{\Omega\setminus B_{t}} \bv{u} \cdot \bv{m} \d{a}
=
\int_{\Omega\setminus B_{t}} \rho \dot{\psi}_{\Omega\setminus B_{t}} \d{v}.
\end{gather}
\end{lemma}
\begin{proof}[Proof of Lemma~\ref{lemma: transport with mass densities}]
Equation~\eqref{eq: TTMB B} is a well known result that can be found in many textbooks (see, e.g., \citealp{GurtinFried_2010_The-Mechanics_0}).  It is obtained by setting $\xi = \rho \psi_{B_{t}}$ in Eq.~\eqref{eq: General transport theorem body} and then using Eq.~\eqref{eq: D/DT pd/pdt BM relation} along with the divergence theorem.  This same strategy can be used to obtain Eq.~\eqref{eq: TTMB ID}, that is, substituting $\rho\psi_{\Omega\setminus B_{t}}$ in place of $\phi$ in Eq.~\eqref{eq: General transport theorem immersed context} and then Eq.~\eqref{eq: D/DT pd/pdt BM relation} along with the divergence theorem.
\end{proof}

\begin{proof}[Proof of Theorem~\ref{th: immersed in control volume}]
Equation~\eqref{eq: TT BM Omega and Bt} is obtained by summing Eqs.~\eqref{eq: TTMB B} and~\eqref{eq: TTMB ID}, where $\psi_{B_{t}}$ and $\psi_{\Omega\setminus B_{t}}$ are taken as the restrictions of $\psi$ to $B_{t}$ and $\Omega\setminus B_{t}$, respectively.
\end{proof}

%%%%%%%%%%%%%%%%%%%%%%%%%%%%%%%%%%%%%%%%%%%%%%%
%% Closure of the \allowdisplaybreaks command %
%%%%%%%%%%%%%%%%%%%%%%%%%%%%%%%%%%%%%%%%%%%%%%%
}

%%%%%%%%%%%%%%%%%%%%%%%%%%%%%%%%%%%%%%%%%%%%%%%%%%
%: BIBLIOGRAPHY                                  %
%%%%%%%%%%%%%%%%%%%%%%%%%%%%%%%%%%%%%%%%%%%%%%%%%%
%\bibliographystyle{elsarticle-harv}              %
%\bibliography{ch_feibm} %
%%%%%%%%%%%%%%%%%%%%%%%%%%%%%%%%%%%%%%%%%%%%%%%%%%

\end{document}